\documentclass[10pt]{amsart}
      \usepackage[mathscr]{eucal}
      \usepackage{amsmath,amsfonts}

\usepackage{mathtools}
\usepackage{mathrsfs}
\usepackage{stackengine,scalerel}

\newcommand\obullet[1]{\ThisStyle{\ensurestackMath{%
  \stackon[1pt]{\SavedStyle#1}{\SavedStyle\kern.6\LMpt\bullet}}}}
\newcommand\ocirc[1]{\ThisStyle{\ensurestackMath{%
  \stackon[1pt]{\SavedStyle#1}{\SavedStyle\kern.6\LMpt\circ}}}}

      \parskip\smallskipamount
          \newtheorem{theorem}{Theorem}[section]
      \newtheorem{definition}[theorem]{Definition}
      \newtheorem{proposition}[theorem]{Proposition}
      \newtheorem{corollary}[theorem]{Corollary}
      
      \newtheorem{example}[theorem]{Example}
      
      \newtheorem{remark}[theorem]{Remark}
      \makeatletter
      \@addtoreset{equation}{section}
      \makeatother

      \newcommand{\BB}{{\mathbb B}}
      \newcommand{\CC}{{\mathbb C}}
      \newcommand{\NN}{{\mathbb N}}
      
      \newcommand{\ZZ}{{\mathbb Z}}
      \newcommand{\DD}{{\mathbb D}}
      \newcommand{\RR}{{\mathbb R}}
      \newcommand{\FF}{{\mathbb F}}

\DeclareMathOperator{\Span}{span}

\newcommand{\HH}{{\mathbb H}}

      \newcommand{\cA}{{\mathcal A}}
      \newcommand{\cB}{{\mathcal B}}
      \newcommand{\cC}{{\mathcal C}}
      \newcommand{\cD}{{\mathcal D}}

      \newcommand{\cG}{{\mathcal G}}
      \newcommand{\cH}{{\mathcal H}}
      
      \newcommand{\cJ}{{\mathcal J}}
      \newcommand{\cK}{{\mathcal K}}
      \newcommand{\cL}{{\mathcal L}}
      \newcommand{\cM}{{\mathcal M}}
       \newcommand{\cO}{{\mathcal O}}
      \newcommand{\cQ}{{\mathcal Q}}
      \newcommand{\cN}{{\mathcal N}}
      \newcommand{\cP}{{\mathcal P}}
      \newcommand{\cR}{{\mathcal R}}
      \newcommand{\cS}{{\mathcal S}}
      
      \newcommand{\cU}{{\mathcal U}}
      \newcommand{\cV}{{\mathcal V}}

      \newdimen\expt
      \def\boxit#1{\setbox0\hbox{$\displaystyle{#1}$}
            \hbox{\lower.4\expt
       \hbox{\lower3\expt\hbox{\lower\dp0
            \hbox{\vbox{\hrule height.4\expt
       \hbox{\vrule width.4\expt\hskip3\expt
            \vbox{\vskip3\expt\box0\vskip2\expt}%
       \hskip3\expt\vrule width.4\expt}\hrule height.4\expt}}}}}}
       
\begin{document}
     
       \pagestyle{myheadings}
      \markboth{ Gelu Popescu}{ Noncommutative domains, universal operator models, and operator algebras }

      \title [ Noncommutative weighted shifts, joint similarity, and function theory ]
      {   Noncommutative weighted shifts, joint similarity, and function theory in several variables  }
        \author{Gelu Popescu}
      \date{October 4, 2023}
      \address{Department of Mathematics, The University of Texas
      at San Antonio \\ San Antonio, TX 78249, USA}
      \email{\tt gelu.popescu@utsa.edu}

\begin{abstract}      The goal of this paper is to study  the structure of noncommutative weighted shifts, their properties, and to  understand their  role as {\it models}  (up to similarity)   for   $n$-tuples of operators on Hilbert spaces  as well as  their implications to function theory on  noncommutative  (resp.commutative) Reinhardt  domains.
We  obtain a  Rota type similarity result  concerning the joint similarity  of $n$-tuples of operators to parts of noncommutative weighted multi-shifts and 
   provide a  noncommutative  multivariable   analogue of Foia\c s--Pearcy model for quasinilpotent operators. 
   The model noncommutative weighted multi-shift which  is studied in this paper  is the $n$-tuple $W=(W_1,\ldots, W_n)$, where  $W_i$ are weighted left creation operators of the full Fock space with $n$ generators associated with a  weight sequence  $\boldsymbol \mu=\{\mu_\beta\}_{|\beta|\geq 1}$ of nonnegative  numbers.
   
   We  also  represent the  injective weighted  multi-shifts 
$W_1,\ldots, W_n$ as ordinary multiplications by  $Z_1,\ldots, Z_n$ on a Hilbert space of  noncommutative  formal power series. This  leads naturally to analytic function theory in several complex variables.   One of the goal for the remainder  of the  paper is to analyze the extent to which  our noncommutative formal power series  represent analytic functions in  several noncommutative (resp. commutative)  variables and to develop a functional calculus for  arbitrary $n$-tuples of operators on a Hilbert space.

In particular, we show that,  for any   $n$-tuple of operators $T=(T_1,\ldots, T_n)$,   there exists   a weighted  multi-shift $W=(W_1,\ldots, W_n)$  such that  the map $p(W_1,\ldots, W_n)\mapsto p(T_1,\ldots, T_n)$ extends to a    $w^*$-continuous $F^\infty(\boldsymbol\mu)$-functional calculus $\Psi_T:F^\infty(\boldsymbol\mu)\to B(\cH)$ which is a completely bounded  algebra homomorphism and    $\|\Psi_T\|_{cb}\leq  \frac{\pi}{\sqrt{6}} $. Most of the results of the paper admit commutative versions.
 
\end{abstract}

      \maketitle

\section*{Contents}
{\it

\quad Introduction
\begin{enumerate}
   \item[1.]     Noncommutative weighted shifts
   \item[2.]     Joint similarity of tuples of operators to parts of weighted shifts
   \item[3.]      Multivariable   analogue of Foia\c s--Pearcy model for quasinilpotent operators
   \item[4.]      Noncommutative Hardy spaces associated with weighted shifts
   \item[5.]      Functional calculus for    arbitrary  $n$-tuples of operators
   \item[6.]       Function theory on Reinhardt domains and multiplier algebras
   \item[7.]      Functional calculus for   commuting  $n$-tuples of operators

      \end{enumerate}
    
    \quad References

}

\section*{Introduction}

 An excellent introduction to the theory of weighted shifts  and  analytic function theory  as well as an extensive bibliography related to this topic  can be found in  Shields' comprehensive survey \cite{Sh} and also  in  Halmos' book \cite{H1}.  The study of weighted shifts was initiated  by Kelly  in his thesis \cite{Ke}   and was  continued by 
Geller \cite{Ge1}, \cite{Ge2}, Jewell \cite{J}, Nikol'skii \cite{Ni}, 
Agler \cite{Ag2}, 
M\" uller \cite{M}, and  many others.

In the   multivariable  setting,     commuting weighted multi-shifts  were studied by Jewell and   Lubin \cite{JL} and used as universal operator models by Drury \cite{Dr},    M\" uller-Vasilescu \cite{MV},  Vasilescu \cite{Va},  Pott \cite{Pot},  Olofsson \cite{O2},  Arveson \cite{Arv3-acta}, Ambrozie-Englis-M\" uller \cite{AEM},  Arazy-Englis \cite{AE}, and  others.

In the noncommutative multivariable setting,     weighted multi-shifts on Fock spaces  first appear in the work of Arias and the author  in \cite{APo1} in connection with weighted noncommutative analytic Toeplitz algebras,  Poisson transforms, and noncommutative interpolation.  A class of periodic weighted shifts  was used by Kribs  \cite{Kr1}   in connection with noncommutative multivariable generalizations of the Bunce-Deddens $C^*$-algebras.   Weighted shifts have  played an important role in the author's  work in operator theory on noncommutative  regular domains  \cite{Po-Berezin}, \cite{Po-domains},  \cite{Po-invariant}, noncommutative Bergman spaces 
\cite{Po-Bergman}, and the more recent study on admissible domains  \cite{Po-NC}.  We should also mention the work of Arias \cite{A} and Arias-Latr\' emoli\' er \cite{AL} on  the classification of noncommutative domain algebras.

 Rota's model theorem Rota  \cite{R}  (see also \cite{H1})  asserts that any bounded linear operator on a Hilbert space with spectral radius less than one is similar to the adjoint of the unilateral shift of infinite multiplicity restricted to an invariant subspace. A refinement of this result for similarity  to  strongly stable contractions was  obtained by Foia\c s \cite{Fo} and by de Branges and Rovnyak \cite{BR}. 
  Analogues of Rota's similarity result were obtained by Herrero \cite{He} and Voiculescu \cite{Vo} for operators with spectrum   in a certain class of bounded open sets of the complex plane.
 Clark \cite{C} obtained a several variable version of Rota's model theorem for commuting strict contractions, and Ball  \cite{Ba} extended the result to a more general commutative setting. Joint similarity of $n$-tuples of operators to    parts   of  certain particular weighted shifts   were considered  by the author in \cite{Po-models} and  \cite{Po-similarity}. 

  It is well-known that any operator    $T\in B(\cH)$  similar to a contraction    is polynomially bounded, i.e., there is a
constant $C>0$ such that, for any polynomial $p$,
$$
\|p(T)\|\leq C \|p\|_\infty,
$$
 where $ \|p\|_\infty:= \sup_{|z|=1} |p(z)|$. A remarkable result  obtained
 by Paulsen \cite{Pa} shows that similarity  to a contraction is equivalent
 to complete polynomial boundedness. Halmos'  similarity problem \cite{H2}
   asked  whether any polynomially bounded operator is similar to a contraction.
   This long standing problem was answered  by Pisier
 \cite{Pi} in  a remarkable paper where  he shows that there are  polynomially
  bounded operators which are  not similar to  contractions. For more information
  on similarity problems, operator spaces  and completely bounded maps we refer the reader to the
   excellent books by Pisier \cite{Pi-book}, \cite{Pi-book2} and Paulsen \cite{Pa-book}.

The results mentioned above have inspired us  in writing the present paper on noncommutative weighted multi-shifts, joint similarity, and the implications to function theory is several variables. To present our results we need a few definitions.

Let $\cH$ be a separable infinite dimensional complex Hilbert space with orthonormal basis
$\{e_\alpha\}_{\alpha\in \FF_n^+}$, where $\FF_n^+$ is the unital free semigroup  with $n$ generators $g_1,\ldots, g_n$ and identity $g_0$. An $n$-tuple of bounded linear operators 
$T=(T_1,\ldots, T_n)\in B(\cH)^n$ is said to be a {\it weighted left multi-shift} if $
T_i e_\alpha=\mu_{g_i\alpha}e_{g_i\alpha}$  for $\alpha\in \FF^+_n,
$
where $\boldsymbol \mu=\{\mu_\beta\}_{|\beta|\geq 1}$ is a sequence of complex numbers.  
One can easily see that $T$ is jointly unitarily equivalent  to the weighted left multi-shift associated to the weights      $\{|\mu_\beta|\}_{|\beta|\geq 1}$.  Due to this reason,  throughout this paper, we assume  that $\mu_\beta\geq 0$. We mention that if $\mu_\beta=1$ for all $\beta$, then the corresponding multi-shift is jointly unitarily equivalent to the $n$-tuple 
$(S_1,\ldots, S_n)$ of left creation operators  on the full Fock space with $n$ generators, which has played  an important  role  in noncommutative multivariable operator theory and $C^*$-algebras.
In our study, we employ a  {\it model weighted (left) multi-shift} $W:=(W_1,\ldots, W_n)$  defined on the full Fock space  $F^2(H_n)$ with $n$ generators  by
$$
W_ie_\alpha:=\mu_{g_i\alpha} e_{g_i\alpha},\qquad i\in \{1,\ldots, n\},
$$
where  $\{e_\alpha\}_{\alpha\in \FF_n^+}$ is the orthonormal basis for $F^2(H_n)$.
Since $W_1,\ldots, W_n$ are bounded linear operators if and only if   $\sup_{\alpha\in \FF_n^+}   \mu_{g_i \alpha}<\infty$ for every $i\in \{1,\ldots, n\}$, the later condition is assumed throughout this paper. 

The present paper aims  at studing  the structure of noncommutative weighted shifts, their properties, and  their role as models (up to similarity) for $n$-tuples of operators on Hilbert spaces as well as their  implications to function theory on  noncommutative  (resp.commutative) Reinhardt  domains.

In Section 1, we present some basic results  concerning the structure of weighted   multi-shifts.  We prove (see Theorem \ref{decompo})   that  any weighted   multi-shift     admits   a decomposition   as   a direct sum of injective    multi-shifts (which is equivalent to $\mu_\beta >0$)  and some special truncations of injective  multi-shifts.  This result together with a  joint similarity criterion for noncommutative weighted multi-shifts (see Theorem \ref{similarity})   play an important role  in order to prove the main result (Theorem \ref{power bounded}) of this section which asserts that    a multi-shift $W=(W_1,\ldots, W_n)$ is power 
bounded, i.e. $\|W_\alpha\|\leq M$ for all $\alpha$, if and only if  it is jointly similar to a row contraction $T=(T_1,\ldots, T_n)$, i.e $T_1T_1^*+\cdots +T_nT_n^*\leq I$. The decomposition result mentioned above is  also used  to show that the $C^*$-algebra $C^*(W)$ generated by $W_1,\ldots, W_n$ and the identity is irreducible if and only if 
$W:=(W_1,\ldots, W_n)$ is injective. If $W$ is injective and  there is $i\in \{1,\ldots, n\}$  such that 
$\lim_{|\alpha|\to \infty} (\mu_{g_i\alpha}-\mu_\alpha)=0$, we prove that  $C^*(W)$ contains all compact operators  in $B(F^2(H_n))$. Using  the standard theory of representations of $C^*$-algebras \cite{Arv-book}, we obtain Wold  decompositions for the unital $*$-representations of $C^*$-algebras  (see Theorem \ref{wold}). This leads to an extension of Coburn  result \cite{Co} in our setting (see Theorem \ref{C*}).

Throughout the remainder of this  paper, unless otherwise specified, we assume that the  weighted multi-shifts  are injective. In Section 2, we  provide   sufficient conditions on an $n$-tuple 
$(T_1,\ldots, T_n)\in B(\cH)^n$ to be  jointly similar to the compression of a given  multi-shifts   $ (W_1\otimes I,\ldots, W_n\otimes I)$  
to  a jointly co-invariant subspace under the operators $W_i\otimes I $,  $i\in \{1,\ldots, n\}$. In particular, we obtain a Rota type model theorem in this setting. 

 The main result of this section   asserts that 
for any $n$-tuple $T=(T_1,\ldots, T_n)\in B(\cH)^n$,   there exists     a weighted multi-shift $W=(W_1,\ldots, W_n)$ with   the joint spectral radius  $r(W)= r(T)$ such that
    $(T_1,\ldots, T_n)$ is jointly similar to  $$(P_\cM(W_1\otimes I_\cH)|_\cM,\ldots, P_\cM(W_n\otimes I_\cH)|_\cM),$$
 where  $\cM\subset F^2(H_n)\otimes \cH$  is a joint invariant subspace under the operators $W_i^*\otimes I_\cH$,  $i\in \{1,\ldots, n\}$
   and    the map  defined by 
 $$\Phi_T(p(W_1,\ldots, W_n)):=p(T_1,\ldots, T_n)$$
  is completely bounded with $\|\Phi_T\|_{cb}\leq \frac{\pi}{\sqrt{6}} $.
 If  $T=(T_1,\ldots, T_n)$ is a nilpotent $n$-tuple of  index $m\geq 2$,  then there is a  truncated weighted multi-shift $W=(W_1,\ldots, W_n)$ which is nilpotent of index $m$ such that all the properties above hold and 
 $$\|\Phi_T\|_{cb}\leq \sqrt{\sum_{k=0}^{m-1}\frac{1}{(k+1)^2}}.
 $$

In Section 3, we   establish the existence   of a model (up to a similarity)  for every quasi-nilpotent $n$-tuple of operators on  a  Hilbert space. This is a multivariable noncommutative extension of  the Foia\c s--Pearcy model for quasinilpotent operators \cite{Fo-Pe}. More precisely, we prove that
 if  $T=(T_1,\ldots, T_n)\in B(\cH)^n$ is quasi-nilpotent, i.e. $r(T)=0$,   then there exists     a   quasi-nilpotent weighted multi-shift $W=(W_1,\ldots, W_n)$ of compact operators  and a joint invariant subspace $\cM\subset F^2(H_n)\otimes \cH$ under the operators $W_i^*\otimes I_\cH$,  $i\in \{1,\ldots, n\}$, such that    $(T_1,\ldots, T_n)$ is jointly similar to  
$$(P_\cM(W_1\otimes I_\cH)|_\cM,\ldots, P_\cM(W_n\otimes I_\cH)|_\cM).$$
Moreover, if $\cH$ is a separable infinite-dimensional Hilbert space, we  prove that
  $(T_1,\ldots, T_n) $    is jointly quasi-similar to an $n$-tuple $(L_1,\ldots, L_n)$, where each $L_i\in B(\cH\otimes F^2(H_n))$  is a compact operator. Some consequences  regarding the joint invariant subspaces  under $T_1,\ldots, T_n$ are also presented.

Section 4 is devoted to noncommutative Hardy spaces associated with weighted multi-shifts.
Let  $\boldsymbol\mu=\{\mu_\beta\}_{\beta\in \FF_n^+}$ be a sequence of strictly positive real numbers  with $\mu_{g_0}=1$
 and 
let $F^2(\boldsymbol\mu)$ be the Hilbert space of formal power series in indeterminates $Z_1,\ldots, Z_n$ with orthogonal basis $\{Z_\alpha:\ \alpha\in \FF_n^+\}$ such that 
$$\|Z_\alpha\|_{\boldsymbol\mu}:=\boldsymbol \mu(\alpha, g_0):=\mu_{g_{i_1}\cdots g_{i_k}} \mu_{g_{i_2}\cdots g_{i_k}} \cdots \mu_{g_{i_k}}\qquad \text{ if } \ \alpha=g_{i_1}\cdots g_{i_k}\in \FF_n^+,
$$
 and note that
$$
F^2(\boldsymbol\mu)=\left\{ \zeta=\sum_{\alpha\in \FF_n^+}c_\alpha Z_\alpha: \   \|\zeta\|_{\boldsymbol\mu}^2:=\sum_{\alpha\in \FF_n^+}\boldsymbol \mu(\alpha, g_0)^2 |c_\alpha|^2<\infty, \ c_\alpha\in \CC\right\}.
$$
 The {\it left multiplication} operator  $L_{Z_i}$ on $F^2(\boldsymbol\mu)$ is  defined by  $L_{Z_i}\zeta:=Z_i\zeta$ for all $\zeta\in F^2(\boldsymbol\mu)$.  Similarly, we define the {\it right multiplication} operator     by   setting $R_{Z_i}\zeta:=\zeta Z_i$ for all $\zeta$.
We remark that  the operator $U:F^2(H_n)\to F^2(\boldsymbol\mu)$,  defined by $U(e_\alpha):=\frac{1}{\boldsymbol \mu(\alpha, g_0)}Z_\alpha$ is unitary with
 $UW_i=L_{Z_i}U$ and $U\Lambda_i=R_{Z_i}U$ for any $i\in \{1,\ldots,n\},
 $
 where $W=(W_1,\ldots, W_n)$ is the weighted left multi-shift  and  ${\Lambda}=(\Lambda_1,\ldots, \Lambda_n)$  is the  corresponding weighted right multi-shift associated with the weight sequence
 $\boldsymbol\mu=\{\mu_\beta\}_{\beta\in \FF_n^+ }$.   Throughout  Section 4  we assume that both the weighted left and right  multi-shifts are bounded.
 The right (resp.~left)  multiplier algebra $\cM^r(F^2(\boldsymbol\mu))$    
(resp.~$\cM^\ell(F^2(\boldsymbol\mu))$) becomes a Banach algebra with respect to the multiplier norm. Similarly to the corresponding results from \cite{Po-Bergman} and \cite{Po-NC}, one can prove that
\begin{equation*}
\begin{split}
\{W_1,\ldots, W_n\}' &=R^\infty(\boldsymbol \mu):= {U}^* \cM^r (F^2(\boldsymbol\mu)) U,  \\
\{\Lambda_1,\ldots, \Lambda_n\}' &=F^\infty(\boldsymbol \mu):={U}^* \cM^l (F^2(\boldsymbol\mu)) U, 
\end{split}
\end{equation*}
and $(F^\infty(\boldsymbol \mu))''=F^\infty(\boldsymbol \mu)$, 
where ${}'$ stands for the commutant. Moreover,  
 the noncommutative  Hardy algebra  $F^\infty(\boldsymbol\mu)$  satisfies the relation
  $$
  F^\infty(\boldsymbol\mu)=\overline{\cP({W})}^{SOT}=\overline{\cP({W})}^{WOT}=\overline{\cP({W})}^{w*},
  $$
where $\cP({W})$ stands for the algebra of all polynomials in $W_1,\ldots, W_n$ and the identity.

Consequently,  we  can represent  injective weighted left multi-shifts 
$W_1,\ldots, W_n$ as ordinary multiplications by  $Z_1,\ldots, Z_n$ on a Hilbert space of noncommutative formal power series.   We will use these two viewpoints interchangeable  as convenient. One of the goal for the remainder  of the  paper is to analyze the extent to which  our noncommutative formal power series  represent analytic functions in  several noncommutative (resp. commutative)  variables.

Let $\lambda=(\lambda_1,\ldots, \lambda_n)\in \CC^n$ and define the linear functional  $\varphi_\lambda:\CC\left<Z_1,\ldots, Z_n\right>\to \CC$ of evaluation  at $\lambda$ by setting
$\varphi_\lambda(p):=p(\lambda)$ for $ p\in \CC\left<Z_1,\ldots, Z_n\right>.$
We say that $\lambda\in \CC^n$ is a {\it  bounded point   evaluation} on the Hilbert space $F^2(\boldsymbol\mu)$ if $\varphi_\lambda$ extends to a bounded linear functional   on
$F^2(\boldsymbol\mu)$. We prove that the set  
$$\cD(\boldsymbol\mu):=\{\lambda\in \CC^n:\   \varphi_\lambda \text{ is  a bounded  linear functional on } F^2(\boldsymbol\mu)\}
$$
coincides with   the joint point spectrum
$$
 \sigma_p(L_{Z_1}^*,\ldots, L_{Z_n}^*)=
\left\{ (\lambda_1,\ldots, \lambda_n)\in \CC^n : \  \sum_{\alpha\in \FF_n^+} \frac{|\lambda_\alpha|^2}{\boldsymbol \mu(\alpha, g_0)^2} <\infty\right\},
$$
which is a complete Reinhardt set in  $\CC^n$ containing the origin.  
We show  that  if    $\lambda=(\lambda_1,\ldots, \lambda_n)\in \CC^n$, then the series   
 $\sum_{k=0}^\infty\sum_{\alpha\in \FF_n^+, |\alpha|=k}c_\alpha \lambda_\alpha$ is convergent for any  element $f=\sum_{\alpha\in \FF_n^+}c_\alpha Z_\alpha$ is 
   $F^2(\boldsymbol\mu)$ if and only if 
  $\lambda\in \cD(\boldsymbol\mu)$. In this case, 
  the series
 $f(\lambda):=\sum_{\alpha\in \FF_n^+}c_\alpha \lambda_\alpha$ is absolutely convergent to $\varphi_\lambda(f)$ and 
 $$
 |f(\lambda)|\leq \|f\|_{\boldsymbol\mu}\left(\sum_{\alpha\in \FF_n^+} \frac{|\lambda_\alpha|^2}{\boldsymbol \mu(\alpha, g_0)^2}  \right)^{1/2}.
 $$
 The last result of Section 4 shows that 
 a map $\Phi: F^\infty(\boldsymbol\mu)\to
\CC$ is a $w^*$-continuous multiplicative linear functional  if and
only if there exists $\lambda\in \cD(\boldsymbol\mu)$  such that
$\Phi(\psi(W))= \psi(\lambda)$ for all $ \psi(W)\in F^\infty(\boldsymbol\mu).$ This is an extension of the corresponding result by Davidson-Pitts \cite{DP1}  (for $\mu_\beta=1$) and  by the author \cite{Po-domains} (for the weighted shifts which are universal models of regular domains).

Section 5  deals with a functional calculus for arbitrary  $n$-tuples of operators on a separable Hilbert space.
 Let  $\boldsymbol\mu=\{\mu_\beta\}_{\beta\in \FF_n^+}$ be a sequence of strictly positive real numbers such that the corresponding   weighted (left and right)   multi-shifts are bounded operators. Let 
     $F^\infty(\boldsymbol\mu)$  be the  noncommutative  Hardy algebra  associated  with $\boldsymbol\mu$ and 
     define the   set
     $$
     \cD_{\boldsymbol\mu}(\cH):=\left\{ (X_1,\ldots, X_n)\in B(\cH)^n:  \sum_{\alpha\in \FF_n^+} \frac{1}{\boldsymbol \mu(\alpha, g_0)^2} X_\alpha X_\alpha^* <\infty
     \right\},
     $$
     which is a noncommutative analogue of  $\cD(\boldsymbol\mu)$, the joint spectrum of $L_{Z_1}^*,\ldots, L_{Z_n}^*$.
In Section 5, we    
   present a $w^*$-continuous $F^\infty(\boldsymbol\mu)$-functional calculus  for the elements in the noncommutative Reinhardt set $\cD_{\boldsymbol\mu}(\cH)$, where $\cH$ is a separable Hilbert space.
 We prove that  if  $T=(T_1,\ldots, T_n)\in  \cD_{\boldsymbol\mu}(\cH)$ and    $\Psi_T:F^\infty(\boldsymbol\mu)\to B(\cH)$   is  defined  by 
 $$
 \Psi_T(\varphi(W))=\varphi(T)=: \text{\rm SOT-}\lim_{N\to\infty}\sum_{ |\alpha|\leq N} \left(1-\frac{|\alpha|}{N+1}\right) c_\alpha T_\alpha,
 $$
 where $\varphi(W)$ has the Fourier representation $\sum_{\alpha\in \FF_n}c_\alpha W_\alpha$, then $\Psi_T$ has the following properties.

 \begin{enumerate}
 \item[(i)]  
 $\Psi_T\left(\sum_{|\alpha|\leq m} c_\alpha W_\alpha\right)=\sum_{|\alpha|\leq m} c_\alpha T_\alpha,\qquad m\in \NN.
 $
 \item[(ii)]   $\Psi_T$ is sequentially WOT-(resp. SOT-)continuous.
 \item[(iii)]   $\Psi_T$ is  a completely bounded  algebra homomorphism.  
 \item[(iv)]  $\Psi_T$ is $w^*$-continuous.
 \item[(v)]  $r(\varphi(T))\leq r(\varphi(W))$ for any $\varphi\in F^\infty(\boldsymbol\mu)$,  where  $r(X)$ denotes the spectral radius of   $X$.
  \end{enumerate}
Quite surprisingly, using the results from the previous sections, we show that  for any   $n$-tuple of operators $T=(T_1,\ldots, T_n)\in B(\cH)^n$,   there exists   a weight sequence
 $\boldsymbol\mu=\{\mu_\beta\}_{\beta\in \FF_n^+}$    such that the corresponding weighted multi-shifts are bounded  and  $r(W)= r(T)$, and such that
   the map $\Psi_T:F^\infty(\boldsymbol\mu)\to B(\cH)$  has all the properties listed above and  $\|\Psi_T\|_{cb}\leq  \frac{\pi}{\sqrt{6}} $. 
   
    In this section, 
 we also present a spectral  version of the Schwarz lemma for the noncommutative Hardy algebra $F^\infty(\boldsymbol\mu)$.  Under the assumption that 
  $\mu_\alpha\geq M>0 $ for any $\alpha\in \FF_n^+$, we show that  
  if $\varphi(W)\in F^\infty(\boldsymbol\mu)$ has the properties that $\varphi(0)=0$ and $\|\varphi(W)\|\leq1$, then  
 $$
 r(\varphi(X))\leq r(X),\qquad X\in \cD_{\boldsymbol\mu}(\cH)
 $$
 Finally, if ${\rm Hol}_0(Z)$  is the algebra of all free holomorphic function  on  neighborhoods of the origin we present a ${\rm Hol}_0(Z)$-functional calculus  for the quasi-nilpotent $n$-tuples of operators.

In Section 6,  inspired by previous work by Arveson \cite{Arv3-acta}, Davidson-Pitts \cite{DP3} when $\mu_\beta=1$, we introduce the  {\it symmetric weighted  Fock space} $F_s^2(\boldsymbol\mu)$   associated with a weight sequence 
$\boldsymbol\mu=\{\mu_\beta\}_{|\beta|\geq 1}$ with the property that $\cD(\boldsymbol\mu)$ contains a neighborhood of the origin and prove that it can be
identified with the Hilbert space $\HH^2(\cD(\boldsymbol\mu))$ of all
functions $\varphi:\cD(\boldsymbol\mu)\to \CC$ which admit a power
series representation $\varphi(\lambda)=\sum_{{\bf k}\in \NN_0}
c_{\bf k} \lambda^{\bf k}$ with
$$
\|\varphi\|_2=\sum_{{\bf k}\in \NN_0}|c_{\bf
k}|^2\frac{1}{\omega_{\bf k}}<\infty,
$$
where $\omega_{\bf k}$ are precisely described in terms of $\boldsymbol\mu$.
We prove that $\HH^2(\cD(\boldsymbol\mu))$ is the reproducing kernel Hilbert space with kernel
  $\boldsymbol\kappa:\cD(\boldsymbol\mu)\times
\cD(\boldsymbol\mu)\to \CC$ defined by
$$
\boldsymbol\kappa(\zeta,\lambda):=\sum_{\beta\in \FF_n^+} \frac{1}{\boldsymbol\mu(\beta,g_0)^2}   \zeta_\beta\overline{\lambda}_\beta
\quad \text{ for all }\ \lambda,\zeta\in
\cD(\boldsymbol\mu).
$$
 We define the operators  $B_i\in B(F^2_s(\boldsymbol\mu))$ by setting $B_i:=P_{F^2_s(\boldsymbol\mu)} W_i|_{F^2_s(\boldsymbol\mu)}$ and show that the  
   $n$-tuple $(B_1,\ldots, B_n)$ is unitarily equivalent to $(M_{\lambda_1},\ldots, M_{\lambda_n})$, where   $M_{\lambda_i}$  is  the multiplication on $\HH^2(\cD(\boldsymbol\mu))$ by the coordinate   function $\lambda_i$.  
We also   introduce the Hardy algebra $F_s^\infty(\boldsymbol\mu)$  to be the $w^*$-closed non-self-adjoint algebra generated by  the operators $B_1,\ldots, B_n$ and the identity.   As main results, we prove that $F_s^\infty(\boldsymbol\mu)$ can be identified with the multiplier algebra $\cM(\HH^2(\cD(\boldsymbol\mu)))$ and that it 
 is  a reflexive algebra. 

In Section 7,  we show that  several results of this paper have    commutative versions.  In particular, we prove that   any $n$-tuple $(T_1,\ldots, T_n)\in B(\cH)^n$ of commuting operators admits a $w^*$-continuous  $F_s^\infty(\boldsymbol\mu)$-functional calculus for an appropriate weight sequence $\boldsymbol\mu$.

\bigskip

\newpage

\section{Noncommutative weighted shifts}

Let $H_n$ be an $n$-dimensional complex  Hilbert space with orthonormal
      basis
      $e_1,\dots,e_n$, where $n\in\NN$.        
      The full Fock space  of $H_n$ is defined by
      $$F^2(H_n):=\bigoplus_{k\geq 0} H_n^{\otimes k},$$
      where $H_n^{\otimes 0}:=\CC 1$ and $H_n^{\otimes k}$ is the (Hilbert)
      tensor product of $k$ copies of $H_n$.
      The {\it left creation
      operators} $S_i:F^2(H_n)\to F^2(H_n), \  i\in\{1,\dots, n\}$,  are given  by
      $$
       S_i\varphi:=e_i\otimes\varphi, \quad  \varphi\in F^2(H_n).
      $$
Let $\FF_n^+$ be the unital free semigroup on $n$ generators
$g_1,\ldots, g_n$ and the identity $g_0$.  The length of $\alpha\in
\FF_n^+$ is defined by $|\alpha|:=0$ if $\alpha=g_0$  and
$|\alpha|:=k$ if
 $\alpha=g_{i_1}\cdots g_{i_k}$, where $i_1,\ldots, i_k\in \{1,\ldots, n\}$. We set  $e_\alpha:=e_{g_{i_1}}\otimes \cdots \otimes e_{g_{i_k}}$ and  $e_{g_0}:=1$, and note that $\{e_\alpha\}_{\alpha\in \FF_n^+}$ is an othonormal basis for $F^2(H_n)$.

Let $\boldsymbol\mu=\{\mu_\beta\}_{\beta\in \FF_n^+, |\beta|\geq 1}$ be a sequence of   nonnegative real numbers and define the {\it weighted left creation  operators}
$W_i:F^2(H_n)\to F^2(H_n)$, $i\in \{1,\ldots, n\}$,  associated with  the weight sequence $\boldsymbol\mu$     by setting $W_i:=S_iD_i$, where
 $S_1,\ldots, S_n$ are the left creation operators on the full
 Fock space $F^2(H_n)$ and the diagonal operators $D_i:F^2(H_n)\to F^2(H_n)$
 are defined by the relation
$$
D_ie_\alpha:=\mu_{g_i \alpha}e_\alpha,\qquad
 \alpha\in \FF_n^+.
$$
Note that the weighted left shifts  $W_1,\ldots, W_n$ are  bounded operators if and only if $\sup_{\alpha\in \FF_n^+}   \mu_{g_i \alpha}<\infty$ for every $i\in \{1,\ldots, n\}$.   We call   the $n$-tuple $W=(W_1,\ldots, W_n)$    the {\it weighted  left multi-shift} associated with  the weight sequence  $\boldsymbol\mu=\{\mu_\beta\}_{|\beta|\geq 1}$. We remark that   the operators $W_1,\ldots, W_n$     have orthogonal ranges.  Throughout this paper, we assume that 
$$\mu_\beta\geq 0 \ \text{  and  } \  \sup_{\alpha\in \FF_n^+}   \mu_{g_i \alpha}<\infty\quad \text{  for any  } \quad i\in \{1,\ldots, n\}.
$$

 If $(T_1,\ldots, T_k)\in B(\cH)^n$,
we set $T_\alpha=T_{i_1}\cdots T_{i_k}$ if $\alpha=g_{i_1}\cdots g_{i_k}$, where $i_1,\ldots, i_k\in \{1,\ldots, n\}$,  and  $T_{g_0}:=I_\cH$.
  We recall from \cite{Po-models} that the {\it joint spectral radius} of an $n$-tuple $T=(T_1,\ldots, T_k)\in B(\cH)^n $ is defined by 
$$
r(T):=\lim_{k\to \infty} \left\|\sum_{\alpha\in \FF_n^+, |\alpha|=k} T_\alpha T_\alpha^*\right\|^{1/2k}.
$$
We note that $r(T)\leq \|T\|=\left\|\sum_{i=1}^n T_iT_i^*\right\|^{1/2}$. In what follows, we  denote by $[T_\alpha:\  |\alpha|=k]$  the row operator matrix with entries  $T_\alpha$, where $\alpha\in \FF_n^+$ with $ |\alpha|=k$. For simplicity,  we also denote by $(T_1,\ldots, T_n)$ the row operator acting from $\cH^{(n)}$,  the direct sum of $n$ copies  of $\cH$, to $\cH$.

\begin{proposition} \label{radius} If $W=(W_1,\ldots, W_n)$  is   the weighted left multi-shift  associated with  $\boldsymbol\mu:=\{\mu_\beta\}_{ |\beta|\geq 1}$, where $\mu_\beta\geq0$,  then the  the joint spectral  radius of $W$  satisfies the relation
$$
r(W)=\lim_{k\to \infty} \sup_{\alpha, \beta\in \FF_n^+, |\beta|=k} \boldsymbol\mu (\beta,\alpha)^{1/k},
$$
 where  
 $$\boldsymbol\mu (\beta,\alpha):=\begin{cases}\mu_{g_{i_1}\cdots g_{i_p}\alpha} \mu_{g_{i_2}\cdots g_{i_p}\alpha} \cdots \mu_{g_{i_p}\alpha} & \text{if}\  \beta=g_{i_1}\cdots g_{i_p}\in \FF_n^+\\
 1&\text{if} \ \beta=g_0
 \end{cases}
 $$
  for any   $\alpha\in \FF_n^+$. 
\end{proposition} 
 \begin{proof} A simple calculation reveals that
\begin{equation}\label{WbWb}
W_i e_\alpha=
\mu_{g_i \alpha}
e_{g_i\alpha} \quad
\text{ and }\quad
W_i^* e_\alpha =\begin{cases}
\mu_{g_i \gamma}e_\gamma& \text{ if }
\alpha=g_i\gamma \\
0& \text{ otherwise }
\end{cases}
\end{equation}
 for every $\alpha  \in \FF_n^+$ and  every $i\in \{1,\ldots, n\}$. 
 Note that if $\beta=g_{i_1}\cdots g_{i_p}\in \FF_n^+$ and  $\alpha\in \FF_n^+$, then the relation  above  implies
 $$
 W_\beta e_\alpha=\mu_{g_{i_1}\cdots g_{i_p}\alpha} \mu_{g_{i_2}\cdots g_{i_p}\alpha} \cdots \mu_{g_{i_p}\alpha} e_{g_{i_1}\cdots g_{i_p}\alpha}.
 $$
 Therefore,
 $$
  W_\beta e_\alpha= \boldsymbol\mu (\beta,\alpha) e_{\beta\alpha},\qquad \alpha,\beta\in \FF_n^+,
 $$
 and, for any $\beta,\gamma\in \FF_n^+$,
 \begin{equation*}
 \left< W_\beta^* e_\gamma, e_\alpha\right>=\left< e_\gamma, \boldsymbol\mu (\beta,\alpha) e_{\beta\alpha}\right>
  =\begin{cases}
\boldsymbol\mu (\beta,\alpha)& \text{ if }
\gamma=\beta\alpha \\
0& \text{ otherwise}.
\end{cases}
\end{equation*}
  Hence, we deduce that 
  $$
  W_\beta^* e_\gamma 
  =\begin{cases}
\boldsymbol\mu (\beta,\alpha) e_\alpha& \text{ if }
\gamma=\beta\alpha \  \text{ for some} \  \alpha \in \FF_n^+\\
0& \text{ otherwise}
\end{cases}
$$
and, consequently, 
$$
W_\beta W_\beta^* e_\gamma
=\begin{cases}
\boldsymbol\mu (\beta,\alpha)^2 e_\gamma& \text{ if }
\gamma=\beta\alpha \  \text{ for some} \  \alpha \in \FF_n^+\\
0& \text{ otherwise}.
\end{cases}
$$
 The later relation implies 
 $$
 \left(\sum_{\beta\in \FF_n^+, |\beta|=k} W_\beta W_\beta^*\right) e_\gamma=\boldsymbol\mu (\beta,\alpha)^2 e_\gamma
 $$
if   $\gamma=\beta\alpha $  for some $ \beta \in \FF_n^+$ with $|\beta|=k$.  
 Otherwise, we have  $ \left(\sum_{\beta\in \FF_n^+, |\beta|=k} W_\beta W_\beta^*\right) e_\gamma=0$.
Since the operator $\sum_{\beta\in \FF_n^+, |\beta|=k} W_\beta W_\beta^*$ is diagonal, we deduce that
$$
\left\|\sum_{\beta\in \FF_n^+, |\beta|=k} W_\beta W_\beta^*\right\|=\sup_{\alpha, \beta\in \FF_n^+, |\beta|=k} \boldsymbol\mu (\beta,\alpha)^2.
$$
 Using the definition of the joint spectral radius  for $n$-tuples of operators, we complete the proof.
 \end{proof}

  We identify $M_m(B(\cH))$, the set of
$m\times m$ matrices with entries in $B(\cH)$, with
$B( \cH^{(m)})$, where $\cH^{(m)}$ is the direct sum of $m$ copies
of $\cH$.
Thus we have a natural $C^*$-norm on
$M_m(B(\cH))$. If $\cA$ is an operator space, i.e., a closed subspace
of $B(\cH)$, we consider $M_m(\cA)$ as a subspace of $M_m(B(\cH))$
with the induced norm.
Let $\cA, \cB$ be operator spaces and $u:\cA\to \cB$ be a linear map. Define
the map
$u_m:M_m(\cA)\to M_m(\cB)$ by
$$
u_m ([A_{ij}]_{m\times m}):=[u(A_{ij})]_{m\times m}.
$$
We say that $u$ is completely bounded   if
$$
\|u\|_{cb}:=\sup_{m\ge1}\|u_m\|<\infty.
$$
If $\|u\|_{cb}\leq1$
(resp. $u_m$ is an isometry for any $m\geq1$) then $u$ is completely
contractive (resp. isometric),
 and if $u_m$ is positive for all $m$, then $u$ is called
 completely positive. For more information  on completely bounded maps
  and  the classical von Neumann inequality \cite{vN}, we refer  the reader 
 to \cite{Pa-book} and \cite{Pi}.

\begin{definition} An $n$-tuple of operators $T=(T_1,\ldots, T_n)\in B(\cH)^n$ is called
\begin{enumerate}
\item[(i)] power bounded,
\item[(ii)] row power bounded,
\item[(iii)] polynomially bounded with respect to the weighted left multi-shift $(W_1,\ldots, W_n)$,
\item[(iv)] completely polynomially bounded with respect to $W_1,\ldots, W_n$,
\end{enumerate}
if  there is a constant $M>0$ such that 
\begin{enumerate}
\item[(i')] $\|T_\alpha\|\leq M$ for any $\alpha\in \FF_n^+$,
\item[(ii')] $\|[T_\alpha:\   |\alpha|=k]\|\leq M$ for any $k\in \NN$,
\item[(iii')] $\|p(T_1,\ldots, T_n)\|\leq M \|p(W_1,\ldots, W_n)\|$ for any polynomial $p\in \CC\left<Z_1,\ldots, Z_n\right>$,
\item[(iv')] $\|[p_{s,q}(T_1,\ldots, T_n)]_{k\times k}\|\leq M \|[p_{s,q}(W_1,\ldots, W_n)]_{k\times k}\|$ for any    $p_{s,q}\in \CC\left<Z_1,\ldots, Z_n\right>$, the polynomial algebra in noncommutative indeterminates $Z_1,\ldots, Z_n$,  and  $k\in \NN$,
\end{enumerate}
respectively.
\end{definition} 
 
 \begin{proposition} Let $T=(T_1,\ldots, T_n)\in B(\cH)^n$ with $n\geq 2$. Then the following statements hold.
 \begin{enumerate}
 \item[(i)] If $T$ is row power bounded, then it is  also power bounded and $r(T)\leq 1$.
 \item[(ii)] If  $r(T)\leq 1$, then $T$ is not necessarily  power bounded.
 \item[(iii)] There is a power bounded  $n$-tuple $T$ which is not  row power bounded and $r(T)>1$.
 \end{enumerate}
 \end{proposition}
 \begin{proof}  The first part of item (i) is quite obvious.
 To show that  $r(T)\leq 1$, note that 
  $$
 \left\|\sum_{\alpha\in \FF_n^+, |\alpha|=k} T_\alpha T_\alpha^*\right\|^{1/2}=\|[T_\alpha:\  |\alpha|=k]\|\leq M.
$$
  To prove item (ii), we  give an example of a weighted  left multi-shift  $W$ which is not power bounded  and  $r(W)=1$. Consider the weights  $\{\mu_\beta\}_{\beta\in \FF_n^+}$ defined  by
 $
 \mu_\beta:=\frac{|\beta|+1}{|\beta|}$ if $ |\beta|\geq 1$, 
and $\mu_{g_0}:=1$.   
Note that 
\begin{equation*}
 \begin{split}
\left\|\sum_{\beta\in \FF_n^+, |\beta|=k} W_\beta W_\beta^*\right\|^{1/2}
&= \sup_{\alpha, \beta\in \FF_n^+,  \beta=g_{i_1}\cdots g_{i_k}} \mu_{g_{i_1}\cdots g_{i_k}\alpha}\mu_{g_{i_2}\cdots g_{i_k}\alpha}\cdots \mu_{g_{i_k}\alpha}\\
&=\sup_{\alpha\in \FF_n} \frac{|\alpha|+k+1}{|\alpha|+1}=k+1.
\end{split}
\end{equation*}
Hence,  we deduce that $W$ is not power bounded and $r(W)=1$.
To prove item (iii), consider the $n$-tuple $(S_1^*,\ldots, S_n^*)$, where $S_1,\ldots, S_n$ are the left creation operators on the full Fock space.  Since $\|S_\alpha^*\|=1$ for any $\alpha\in \FF_n$ and 
$$
\left\|\sum_{\beta\in \FF_n^+, |\beta|=k} S_\beta^* S_\beta\right\|=n^k
$$
the result follows.
 \end{proof}
  
 Next, we solve an interpolation problem  for the norms
  $$
 \|\varphi_T^{k}\|=\left\|\sum\limits_{\sigma\in \FF_n^+, |\sigma|=
   k}T_\sigma T_\sigma^*\right\|
   $$
 of the completely positive linear maps $\varphi_T^{k}:=\varphi_T\circ\cdots\circ \varphi_T$ ($k$ times), where 
 $\varphi_T(X):=T_1XT_1^*+\cdots + T_nXT_n^*$.

 \begin{theorem}  A sequence $\{a_k\}_{k\in \NN}\subset \CC$  has the property that there is an $n$-tuple of operators $(T_1,\ldots, T_n)$   such that
 $$a_k=\left\|\sum\limits_{\sigma\in \FF_n^+, |\sigma|=
   k}T_\sigma T_\sigma^*\right\|^{1/2},\qquad k\in \NN,
   $$
   if and only if  $a_k\geq 0$ and $a_{k+m}\leq a_k a_m$ for any $k,m\in \NN$, and either one of the following conditions holds:
   \begin{enumerate}
   \item[(i)] there is $p\in \NN$ such that  $a_k>0$ if $k\leq p$ and $a_k=0$ if $k> p$,
    \item[(ii)] $a_k>0$ for any $k\in \NN$.
   \end{enumerate}
 \end{theorem}
 
  \begin{proof} The direct implication is due to the well-known fact that
  $$
  \left\|\sum\limits_{\sigma\in \FF_n^+, |\sigma|=
   k+m}T_\sigma T_\sigma^*\right\|\leq  \left\|\sum\limits_{\sigma\in \FF_n^+, |\sigma|=
   k}T_\sigma T_\sigma^*\right\|  \left\|\sum\limits_{\sigma\in \FF_n^+, |\sigma|=
   m}T_\sigma T_\sigma^*\right\|.
   $$
  To prove the converse, we construct a weighted left multi-shift 
  having the required property.  Assume that the condition in item (i) holds. Given  the sequence $\{a_k\}_{k\in \NN}$, we define the weights
  $$\mu_\beta:=\frac{a_{|\beta|}}{a_{|\beta|-1}},\qquad |\beta|\geq 1 \text{ and }  |\beta|\leq p,
  $$
  where $a_{0}:=1$, and set $\mu_\beta:=0$ if $|\beta|>p$. The sequence $\{\mu_\beta\}_{|\beta|\geq 1}$ is bounded since  $\mu_\beta\leq \frac{a_{|\beta|-1} a_1}{a_{|\beta|-1}}=a_1$.  Let   $W=(W_1,\ldots, W_n)$  be the     weighted left multi-shift associated with $\{\mu_\beta\}_{|\beta|\geq 1}$.
 According to the proof of Theorem \ref{radius}, if $k\leq p$,  we have 
 \begin{equation*}
 \begin{split}
\left\|\sum_{\beta\in \FF_n^+, |\beta|=k} W_\beta W_\beta^*\right\|^{1/2}
&=\sup_{\alpha, \beta\in \FF_n^+, |\beta|=k} \boldsymbol\mu (\beta,\alpha)\\
&= \sup_{\alpha, \beta\in \FF_n^+,  \beta=g_{i_1}\cdots g_{i_k}} \mu_{g_{i_1}\cdots g_{i_k}\alpha}\mu_{g_{i_2}\cdots g_{i_k}\alpha}\cdots \mu_{g_{i_k}\alpha}\\
&=\sup_{{\alpha, \beta\in \FF_n^+, |\beta|=k} \atop{|\alpha|+|\beta|\leq p}} \frac{a_{|\beta|+|\alpha|}}{a_{|\beta|+|\alpha|-1}}\cdot
 \frac{a_{|\beta|+|\alpha|-1}}{a_{|\beta|+|\alpha|-2}}\cdots  \frac{a_{ |\alpha|+1}}{a_ {|\alpha|}}\\
 &=\sup_{{\alpha, \beta\in \FF_n^+, |\beta|=k} \atop{|\alpha|+|\beta|\leq p}} \frac{a_{|\beta|+|\alpha|}}{a_ {|\alpha|}}
 \leq a_{|\beta|}.
\end{split}
\end{equation*} 
For any $\beta:=g_{i_1}\cdots g_{i_k}\in \FF_n^+$,  the definition of the weights  $\{\mu_\beta\}_{|\beta|\geq 1}$ implies 
$$a_{|\beta|}= \mu_{g_{i_1}\cdots g_{i_k}} a_{k-1}=\cdots=  \mu_{g_{i_1}\cdots g_{i_k}} \mu_{g_{i_2}\cdots g_{i_k}}\cdots  \mu_{ g_{i_k}}.
 $$ 
 Consequently, 
 $$
 \left\|\sum_{\beta\in \FF_n^+, |\beta|=k} W_\beta W_\beta^*\right\|^{1/2}\leq 
  \mu_{g_{i_1}\cdots g_{i_k}} \mu_{g_{i_2}\cdots g_{i_k}}\cdots  \mu_{ g_{i_k}}\leq 
   \left\|\sum_{\beta\in \FF_n^+, |\beta|=k} W_\beta W_\beta^*\right\|^{1/2},
 $$
which implies 
$$
 \left\|\sum_{\beta\in \FF_n^+, |\beta|=k} W_\beta W_\beta^*\right\|^{1/2}= \mu_{g_{i_1}\cdots g_{i_k}} \mu_{g_{i_2}\cdots g_{i_k}}\cdots  \mu_{ g_{i_k}}=a_k,\qquad k\leq p.
$$ 
If we take $p=\infty$ in the considerations above,  we obtain the corresponding  proof  if item (ii) holds.  
 This completes the proof.
  \end{proof}

  \begin{definition}
 We say that $T=(T_1,\ldots, T_n)\in   B(\cH)^n $ is a nilpotent $n$-tuple  if there is $m\in\NN$ such that $T_\alpha=0$ for any $\alpha\in \FF_n^+$ with $|\alpha|=m$.  If $m$ is the smallest natural number having this property, it is called the index of $T$.
    If  $r(T)=0$, we say that  $T$  is  a {\it quasi-nilpotent} $n$-tuple.  
 \end{definition}

We recall that $T=(T_1,\ldots, T_n)\in B(\cH)^n$ is  called row  contraction if $T_1T_1^*+\cdots +T_nT_n^*\leq I$.      A compact operator is said to be in the class $\cC^p \ (0<p<\infty)$ if the eigenvalues of $(A^*A)^{1/2}$ are in $\ell^p$. We also consider the lexicographic order on the unital free semigroup   $\FF_n^+$.

  \begin{proposition}  \label{proprieties}  Let $W=(W_1,\ldots, W_n)$  be   the weighted left multi-shift  associated with the weight sequence   $\boldsymbol\mu=\{\mu_\beta\}_{ |\beta|\geq 1}$ with $\mu_\beta\geq 0$.   Then the following statements hold.
  \begin{enumerate}
   \item[(i)]  The row operator $(W_1,\ldots, W_n)$ is injective if and only if   all the entries  $W_1,\ldots, W_n$ are injective operators, which is equivalent to  all the weights  being  strictly positive. 
  \item[(ii)] If  $W$ is  injective and  $\mu_\beta\to 0$ as 
  $|\beta|\to \infty$, then $r(W)=0$.
  \item[(iii)] $W_i$ is a compact operator if and only if  $\lim_{\alpha\to \infty}  \mu_{g_i\alpha}=0$. Moreover,   $W_i$ is in the class  $\cC^p$ if and only if  $\{\mu_{g_i\alpha}\}_{\alpha\in \FF_n^+}\in \ell^p$.
  \item[(iv)] $W$ is a row contraction if and only if $\mu_\alpha\leq 1$ for any $\alpha\in \FF_n^+$, $|\alpha|\geq 1$.
  \item[(v)] If  each $W_i$ is a compact operator, then $W=(W_1,\ldots, W_n)$ is  a
   norm-limit of nilpotent  weighted left multi-shifts.
  \item[(vi)]    If $W$ is  injective, then any non-trivial joint invariant subspace  under $W_1,\ldots, W_n$  is infinite dimensional.
  \item[(vii)] Every weighted multi-shift is the limit in norm of injective  weighted multi-shifts.
     \end{enumerate}
  \end{proposition}
  \begin{proof}   Part (i)   follows from the fact that $W_i:=S_iD_i$, where
 $S_1,\ldots, S_n$ are the left creation operators on the full
 Fock space $F^2(H_n)$ and the diagonal operator $D_i:F^2(H_n)\to F^2(H_n)$ is  defined by   relation
$
D_ie_\alpha:=\mu_{g_i \alpha}e_\alpha$,  $\alpha\in \FF_n^+. 
$
It is clear that $D_i$ is injective if and only if  $\mu_{g_i \alpha}>0$. Since $S_1,\ldots, S_n$ have orthogonal ranges, the result follows.
  
  To prove item(ii), we 
    recall from Proposition \ref{radius}  that,   for any $\alpha\in \FF_n^+$ and $\beta=g_{i_1}\cdots g_{i_k}\in \FF_n^+$, 
  $\boldsymbol\mu (\beta,\alpha)= \mu_{g_{i_1}\cdots g_{i_k}\alpha} \mu_{g_{i_2}\cdots g_{i_k}\alpha} \cdots \mu_{g_{i_k}\alpha}$.
  Setting  $b_k:=\max\{\mu_\beta: \ \beta\in \FF_n^+,  |\beta|=k\}$, we have  $b_k\to 0$ as $k\to \infty$. Due to a standard Stolz-Ces\` aro argument, we deduce that
  $$
  \lim_{k\to \infty} (b_{k+p} b_{k-1+p}\cdots b_{1+p})^{1/k}=0
  $$
 for each $p\in \NN$.   This implies 
 $$
 \lim_{k\to \infty} \sup_{|\beta|=k, \beta=g_{i_1}\cdots g_{i_k}} \left(\mu_{g_{i_1}\cdots g_{i_k}\alpha} \mu_{g_{i_2}\cdots g_{i_k}\alpha} \cdots \mu_{g_{i_k}\alpha}\right)^{1/k}=0
  $$
for each $\alpha\in \FF_n^+$.   Given $\epsilon>0$ and $\alpha\in \FF_n^+$, we choose  $k_0(\alpha, \epsilon)\in \NN$ such that  
$$\sup_{|\beta|=k, \beta=g_{i_1}\cdots g_{i_k}} \left(\mu_{g_{i_1}\cdots g_{i_k}\alpha} \mu_{g_{i_2}\cdots g_{i_k}\alpha} \cdots \mu_{g_{i_k}\alpha}\right)^{1/k}<\epsilon$$
for any $k\geq k_0(\alpha, \epsilon)$.
We also choose 
 $m_0\in \NN$ such that
$\mu_\beta <\epsilon$ for any $\beta\in \FF_n^+$ with $|\beta|\geq m_0$.
Note that if $k\geq m_0$ and  $k\geq k_0(\alpha, \epsilon)$ for any  $\alpha\in \FF_n^+$ with $|\alpha|\leq m_0$, then 
 $$
 \sup_{{\alpha, \beta\in \FF_n^+}\atop{ |\beta|=k, \beta=g_{i_1}\cdots g_{i_k}}}\left(\mu_{g_{i_1}\cdots g_{i_k}\alpha} \mu_{g_{i_2}\cdots g_{i_k}\alpha} \cdots \mu_{g_{i_k}\alpha}\right)^{1/k}<\epsilon
 $$ 
   On the other hand, if $|\alpha|\geq m_0$, then 
  $\left(\mu_{g_{i_1}\cdots g_{i_k}\alpha} \mu_{g_{i_2}\cdots g_{i_k}\alpha} \cdots \mu_{g_{i_k}\alpha}\right)^{1/k}<(\epsilon^k)^{1/k}=\epsilon$ as well. Using Proposition \ref{radius}, we complete the proof of item (ii).
  Part (iii) of this proposition is due to the fact that   $W_i^*W_ie_\alpha= \mu_{g_i\alpha}^2 e_\alpha$ for any $\alpha\in \FF_n^+$. Since   $\left(\sum_{j=1}^n W_jW_j^*\right)e_\alpha =\mu_\alpha^2 e_\alpha$ if $\alpha\in \FF_n^+$ with $|\alpha|\geq 1$ and  $0$ otherwise,  item (iv) follows.

To prove part (v), let $\cL_m:=\Span\{e_\alpha: |\alpha|\leq m\}$ and   consider the $n$-tuple $W^{(m)}=(W_1^{(m)},\ldots, W_n^{(m)})$ where $W_i^{(m)}:=P_{\cL_m} W_i$ for $i\in \{1,\ldots, n\}$. It is easy to see that $W^{(m)}$ is a nilpotent weighted left multi-shift  and
$\|W-W^{(m)}\|\leq \sup_{|\alpha|\geq m} \mu_{\alpha}$. Since $\mu_\beta\to 0$ as 
  $|\beta|\to \infty$, we deduce that $\|W-W^{(m)}\|\to 0$ as $m\to \infty$.
  
  To prove item (vi), assume that $\cM\subset F^2(H_n)$, $\cM\neq \{0\}$, is  a joint invariant subspace  under $W_1,\ldots, W_n$ and let $f\in \cM$, $f\neq 0$.  Since the operators $\{W_\beta\}_{|\beta|=k}$ have orthogonal ranges and $W_1,\ldots, W_n$ are injective operators, the vectors $\{W_\beta f:\ |\beta|=k\}$ are  orthogonal and  non-zero. Consequently, $\dim \cM\geq k$ for any $k\in \NN$.
  To prove part (vii),  it is enough to replace the zero weights by small positive numbers.
The proof is complete.
  \end{proof}

\begin{definition} 
We say that  an $n$-tuple $T=(T_1,\ldots, T_k)\in B(\cH)^n $
   is  jointly similar to    an $n$-tuple $A=(A_1,\ldots, A_n)\in B(\cK)^{n}$
  if there is an invertible operator 
$X:\cH\to \cK$ such that
$$
T_{i}= X^{-1} A_{i} X,\qquad i\in \{1,\ldots, n\}.
$$
\end{definition}
 
 We say that $W=(W_1,\ldots, W_n)$ is injective  if each weighted shift $W_i$ is injective. This is equivalent to all the weights  $\{\mu_\beta\}_{ |\beta|\geq 1}$ being different from zero.
 
\begin{definition}   Let $W=(W_1,\ldots, W_n)$  be  an injective weighted left multi-shift associated with  the weight sequence   $\boldsymbol\mu=\{\mu_\beta\}_{ |\beta|\geq 1}$.  Let $\Lambda$ be a  subset in $ \FF_n^+\backslash \{g_0\}$ and 
let $\Omega:=\{\alpha\gamma:\ \gamma\in \Lambda, \alpha\in \FF_n^+\}$. 
The weighted left multi-shift $V=(V_1,\ldots, V_n)$ associated with the weight sequence
${\bf v}=\{v_\beta\}_{|\beta|\geq 1}$  defined by
$$
v_\beta :=\begin{cases} \mu_\beta & \text{ if } \  \beta\notin \Omega\\
0& \text{ and } \ \beta\in \Omega
\end{cases}
$$
is called  truncation of the injective weighted left multi-shift $W$. If $\Lambda=\emptyset$, then $W=V$.
\end{definition}
\begin{remark}  \label{succesor}
For any   $\beta=g_{i_1} g_{i_2}\cdots g_{i_k}\in \FF_n^+$, if $v_\beta\neq 0$, then 
$v_{ g_{i_2}\cdots g_{i_k}}\neq 0$,  $v_{ g_{i_3}\cdots g_{i_k}}\neq 0$, $\ldots$,  $v_{g_k}\neq 0$.  Moreover,  this property characterizes the truncations  of injective weighted left multi-shifts.
\end{remark}

 \begin{theorem} \label{decompo}
 Any weighted  left multi-shift  is unitarily equivalent to  a direct sum of injective weighted  left multi-shifts and truncations of injective weighted  left multi-shifts.
 \end{theorem}
 \begin{proof}
 Let $k\geq 1$ be the smallest natural number such that there exists $\beta\in \FF_n^+$ with $|\beta|=k$ and $\mu_\beta=0$. For such a word  $\beta$, we consider the subspace 
 $$
  \cM_\beta:=\Span \{e_{\omega\beta}: \ \omega\in \FF_n^+\}\subset F^2(H_n)
 $$
 and the set $\Lambda_\beta:=\{\omega\beta : \ \omega\in \FF_n^+\}$. Note that $\cM_\beta$ is a reducing subspace under $W_1,\ldots, W_n$. Indeed, if $|\omega|\geq 1$,
 then $W_i^* e_{\omega\beta}\in \cM_\beta$ for any $i\in \{1,\ldots, n\}$.  Setting $\beta=g_{i_1}\cdots g_{i_k}$, we have $W_{i_1}^*e_\beta=\mu_\beta e_{g_{i_2}\cdots g_{i_k}}=0$ and $W_j^* e_\beta=0$ if   $j \in \{1,\ldots, n\}$ and $j\neq i_1$. Since $W_i \cM_\beta\subset \cM_\beta$,  our assertion follows.
 
 Consider the set $\Lambda_k:=\{ \beta\in \FF_n^+: \ |\beta|=k \ \text{ and }\ \mu_\beta=0\}$ and note that if $\beta_1,\beta_2\in \Lambda_k$, $\beta_1\neq \beta_2$, then $\cM_{\beta_1}\perp \cM_{\beta_2}$.
 Due to the considerations above, we have the orthogonal decomposition 
 \begin{equation}\label{decomp}
 F^2(H_n)=\bigoplus_{\beta\in \Lambda_k}\cM_\beta \oplus \cN_k,
 \end{equation}
 where $\cM_\beta $ and $\cN_k$ are reducing subspaces under the operators $W_1,\ldots, W_n$. We also have
 $$
 \cN_k=\overline{\Span}\{e_\alpha: \ \alpha\in \FF_n^+\backslash \{\omega\beta:\ \beta\in \Lambda_k, \omega\in \FF_n^+\}\}
 $$
 Now, if $\beta\in \Lambda_k$ and there is no word $\sigma\in \Lambda_\beta$ such that $\mu_\sigma=0$, then  the $n$-tuple $(W_1|_{\cM_\beta},\ldots, W_n|_{\cM_\beta})$ is unitarily equivalent to an injective weighted left multi-shift. Indeed, consider the injective multi-shift $(V_1,\ldots, V_n)$ defined by $V_ie_\omega:=\mu_{g_i\omega\beta} e_{g_i\omega}$ for $\omega\in \FF_n$ and $i\in\{1,\ldots, n\}$, and the unitary operator $U:F^2(H_n)\to \cM_\beta$ defined by  $Ue_\omega:= e_{\omega\beta}$ for all $\omega\in \FF_n^+$. Since  $V_i=U^* W_i U$ for any $i\in \{1,\ldots, n\}$, our assertion follows.
Similarly, one can show that if there is no word $\sigma\in \FF_n^+$ such that   $e_\sigma\in \cN_k$ and $\mu_\sigma=0$, then  the $n$-tuple $(W_1|_{\cN_k},\ldots, W_n|_{\cN_k})$ is unitarily equivalent to a truncation of an  injective weighted left multi-shift.
 
The next step is to consider the smallest natural number $N>k$ such that there exists $\sigma\in \FF_n^+$ with $|\sigma|=N$ and $\mu_\sigma=0$. If there is no such $N$,  the proof of the theorem is complete. Otherwise, we fix such a $\sigma$ and note that $e_\sigma$ belongs to a unique subspace $\cM_\beta$ or $\cN_k$ in the decomposition \eqref{decomp}.    
 First, assume that $e_\sigma\in \cM_\beta$ for some $\beta\in \Lambda_k$. Then $\sigma=\omega\beta$ for some $\omega\in \FF_n^+$ with $|\omega|=N-k$. Consider  the subspace  of $\cM_\beta$ defined by
$$
\cM_\beta^\sigma:=\overline{\Span}\{ e_{\gamma\sigma}: \ \gamma\in \FF_n^+\}.
$$
 As in the first part of the proof, since $\mu_\sigma=0$, one can show that 
 $\cM_\beta^\sigma$ is a reducing subspace for $W_1,\ldots, W_n$. Consequently,
 $\cM_\beta=\cM_\beta^\sigma\oplus \cN_N^\sigma$ and $\cN_N^\sigma$ is also  a reducing subspace for $W_1,\ldots, W_n$.
 
 When $e_\sigma\in \cN_k$, a similar argument leads  to an orthogonal decomposition of $\cN_k$  into two reducing subspaces under $W_1,\ldots, W_n$.
 Consequently, the orthogonal decomposition \eqref{decomp} is refined  as we apply our procedure for each $\sigma\in \FF_n^+$ with $|\sigma|=N$ and $\mu_\sigma=0$. Since we have at most countable many zero weights, and inductive  argument leads to an orthogonal decomposition of the full Fock space $F^2(H_n)$  having at most countably many reducing subspaces $\cM$, $\cN$ under $W_1,\ldots, W_n$  with the property that  the
 $n$-tuple $(W_1|_{\cM },\ldots, W_n|_{\cM })$ is unitarily equivalent to an injective weighted left multi-shift and $(W_1|_{\cN },\ldots, W_n|_{\cN })$ is unitarily equivalent to a truncated injective weighted left multi-shift.
 The proof is complete.
   \end{proof}

The following joint similarity criterion  for noncommutative weighted left multi-shifts extends the corresponding classical result for weighted unilateral shifts \cite{H1}, \cite{Sh}, and the extension to the  particular class of weighted shifts which are universal models for regular noncommutative domains \cite{Po-domains}.
 
 \begin{theorem}\label{similarity} Let $W=(W_1,\ldots, W_n)$ and $W'=(W'_1,\ldots, W'_n)$ be truncated weighted left multi-shifts associated with  the weights  $\{\mu_\beta\}_{ |\beta|\geq 1}$ and $\{\mu'_\beta\}_{ |\beta|\geq 1}$, respectively. Assume that,  for each  
 $\beta\in \FF_n^+$  with $ |\beta|\geq 1$,  $\mu_\beta=0$  if and only if $\mu'_\beta=0$.  Then  the following statements hold.
 \begin{enumerate}
 \item[(i)]
 $W$ is jointly similar to $W'$ if and only if there are constants $C_1,C_2$ such that
 $$
 0<C_1\leq  \frac{\mu'_{g_{i_1}\cdots g_{i_p}} \mu'_{g_{i_2}\cdots g_{i_p}} \cdots \mu'_{g_{i_p}}}{\mu_{g_{i_1}\cdots g_{i_p}} \mu_{g_{i_2}\cdots g_{i_p}} \cdots \mu_{g_{i_p}}}\leq C_2
 $$
 for any $\sigma=g_{i_1}\cdots g_{i_p}\in \FF_n^+$ and $p\in \NN$ such $\mu_\sigma\neq 0$. Moreover, the operator $A$ implementing  the similarity can be chosen such that $\|A\|\|A^{-1}\|\leq   \frac{C_2}{C_1}$.
 \item[(ii)]
 $W$ is unitarily equivalent to $W'$  if and only if $\mu_\beta=\mu_\beta'$ for any 
 $\beta\in \FF_n^+$  with $ |\beta|\geq 1$.
 \end{enumerate}
 \end{theorem}
 \begin{proof}
 Assume that $A\in B(F^2(H_n))$ is an invertible operator having the matrix representation $[a_{\alpha, \beta}]$ and such that $AW_i=W_i' A$ for any $i\in \{1,\ldots,n\}$.
 Using the definition of the  weighted left multi-shifts, we deduce that
 $$\left<AW_ie_\alpha, e_\beta\right>=\left<A\mu_{g_i\alpha} e_{g_i\alpha}, e_\beta\right>
 =\mu_{g_i\alpha}a_{\beta,g_i\alpha}
 $$
 and 
 \begin{equation*} \begin{split}
  \left< W_i'A e_\alpha, e_\beta\right>
 &=\begin{cases}
 \left<Ae_\alpha, \mu'_{g_i\gamma} e_\gamma \right>  &\text{ if } \beta=g_i\gamma \\
 0 &\text{ otherwise } 
 \end{cases}\\
 &=\begin{cases}
  \mu'_{g_i\gamma}a_{\gamma, \alpha}  &\text{ if } \beta=g_i\gamma\\
 0 &\text{ otherwise. }
 \end{cases}
\end{split}
 \end{equation*}
 Consequently, we must have  $\mu_{g_i\alpha}a_{\beta,g_i\alpha}=\mu'_{g_i\gamma}a_{\gamma, \alpha}$ if  $\beta=g_i\gamma$ for some $\gamma\in \FF_n^+$, and  $a_{\beta,g_i\alpha}=0$  otherwise.  Hence, we have $a_{g_0,\sigma}=0$ for any  $\sigma\in \FF_n^+$ with $|\sigma|\geq 1$, and 
 $$
 \mu_{g_i\alpha}a_{g_i\gamma, g_i \alpha}=\mu'_{g_i\gamma} a_{\gamma, \alpha}, \qquad i\in \{1,\ldots, n\}.
 $$
Since $A$ is invertible, we must have $a_{g_0,g_0}\neq 0$.  On the other hand, the relation above implies 
\begin{equation}
\label{sisi}
a_{\sigma,\sigma}=\frac{\mu'_{g_{i_1}\cdots g_{i_k}}}{\mu_{g_{i_1}\cdots g_{i_k}}}\cdot \frac{\mu'_{g_{i_2}\cdots g_{i_k}}}{\mu_{g_{i_2}\cdots g_{i_k}}}\cdots  \frac{\mu'_{g_{i_k}}}{\mu_{g_{i_k}}}a_{g_0,g_0}
\end{equation}
for any $\sigma=g_{i_1}\cdots g_{i_k}\in \FF_n^+$ and $k\in \NN$ such that $\mu_\sigma\neq 0$. Here we use  Remark \ref{succesor} which asserts that if $\mu_\sigma\neq 0$, then $\mu_{g_{i_2}\cdots g_{i_k}}\neq 0$, $\mu_{g_{i_3}\cdots g_{i_k}}\neq 0$, $\ldots$, $\mu_{g_{i_k}}\neq 0$. On the other hand, according to the hypothesis, $\mu_\beta=0$  if and only if $\mu'_\beta=0$. Consequently, we also have  $\mu'_{g_{i_2}\cdots g_{i_k}}\neq 0$, $\mu'_{g_{i_3}\cdots g_{i_k}}\neq 0$, $\ldots$, $\mu'_{g_{i_k}}\neq 0$.

 Using the fact that  $|a_{\sigma,\sigma}|=|\left<Ae_\sigma, e_\sigma\right>|\leq \|A\|$  and $a_{g_0,g_0}\neq 0$, we deduce one of the inequalities  in the theorem. The other inequality can be obtained  in a similar manner using the relation   $A^{-1}W_i'=W_iA^{-1}$, $i\in \{1,\ldots, n\}$.

 To prove the converse, define the diagonal operator $A\in B(F^2(H_n))$ by setting 
 $De_\alpha:=d_\alpha e_\alpha$,  where  $d_{g_0}:=1$ and 
 \begin{equation}\label{d}
 d_\sigma:=\frac{\mu'_{g_{i_1}\cdots g_{i_k}}}{\mu_{g_{i_1}\cdots g_{i_k}}}\cdot \frac{\mu'_{g_{i_2}\cdots g_{i_k}}}{\mu_{g_{i_2}\cdots g_{i_k}}}\cdots  \frac{\mu'_{g_{i_k}}}{\mu_{g_{i_k}}} d_{g_0}
 \end{equation}
 for any $\sigma=g_{i_1}\cdots g_{i_k}\in \FF_n^+$ and $k\in \NN$ such that 
 $\mu_\sigma\neq 0$. On the other hand, if $\alpha\in \FF_n^+$ is such that $\mu_\alpha=0$, we set $d_\alpha:=1$. If $D$ has the    the matrix representation 
 $[d_{\alpha, \beta}]$, then $d_{\alpha,\beta} =0$ if $\alpha\neq \beta$ and $d_{\alpha,\alpha}=d_\alpha$.
As in the first part of the proof, one needs to  check that $DW_i=W_i' D$ for any 
$i\in \{1,\ldots, n\}$. Consequently, we need to show that  $\mu_{g_i\alpha}d_{\beta,g_i\alpha}=\mu'_{g_i\gamma}d_{\gamma, \alpha}$ if  $\beta=g_i\gamma$ for some $\gamma\in \FF_n^+$, and  $d_{\beta,g_i\alpha}=0$  otherwise.  According to the definition of $D$, it is clear that $d_{\beta,g_i\alpha}=0$ if   $\beta\neq g_i\gamma$ for some $\gamma\in \FF_n^+$.  It remains to prove that
$\mu_{g_i\alpha}d_{g_i\gamma, g_i\alpha}=\mu'_{g_i\gamma}d_{\gamma, \alpha}$ for any $\alpha,\gamma\in \FF_n^+$ and $i\in \{1,\ldots, n\}$.

If $\gamma\neq \alpha$, we have  $d_{g_i\gamma, g_i\alpha}=d_{\gamma, \alpha}=0$. Consequently,  the relation above holds. Now, consider the case when $\gamma= \alpha$. 
If $\mu_{g_i\alpha}=0$, then $\mu'_{g_i\alpha}=0$, so the relation above holds. In case  $\mu_{g_i\alpha}\neq 0$,  we have $\mu'_{g_i\alpha}\neq 0$ and  relation \eqref{d} implies   $d_{g_i \alpha}=\frac{\mu'_{g_i\alpha}}{\mu_{g_i\alpha}} d_\alpha$ which completes the proof.

To prove part (ii),  assume that the operator $A$  employed in the proof of  item (i) is unitary.
According to the relations above, we have $a_{\beta,\gamma}=0$ for any $\beta, \gamma\in \FF_n^+$, $\beta\neq \gamma$, and relation  \eqref{sisi} holds.
On the other hand, since  $A^{-1}=A^*$, the matrix of $A^*$ is  $[c_{\alpha, \beta}]$  where the entries satisfy relation $c_{\alpha, \beta}=\bar a_{\beta, \alpha}$.
Since $A^*W_i'=W_i A^*$ for any  $i\in \{1,\ldots, n\}$, we obtain, as above, that
\begin{equation*}
c_{\sigma, \sigma}=
\frac{\mu_{g_{i_1}\cdots g_{i_k}}}{\mu'_{g_{i_1}\cdots g_{i_k}}}\cdot \frac{\mu_{g_{i_2}\cdots g_{i_k}}}{\mu'_{g_{i_2}\cdots g_{i_k}}}\cdots  \frac{\mu_{g_{i_k}}}{\mu'_{g_{i_k}}}c_{g_0,g_0}\neq 0
\end{equation*}
for any $\sigma=g_{i_1}\cdots g_{i_k}\in \FF_n^+$ and $k\in \NN$ such that 
 $\mu_\sigma\neq 0$. 
Using also relation \eqref{sisi}, we conclude that
 $$
 \mu_{g_{i_1}\cdots g_{i_k}} \mu_{g_{i_2}\cdots g_{i_k}}\cdots \mu_{g_{i_k}}
 =\mu'_{g_{i_1}\cdots g_{i_k}} \mu'_{g_{i_2}\cdots g_{i_k}}\cdots \mu'_{g_{i_k}}
 $$
for any $\sigma=g_{i_1}\cdots g_{i_k}\in \FF_n^+$ and $k\in \NN$ such that 
 $\mu_\sigma\neq 0$. 
  The relation above is equivalent to $\mu_\alpha=\mu'_\alpha$ for any $\alpha\in \FF_n^+$ with $|\sigma|\geq 1$ such that 
 $\mu_\sigma\neq 0$. Since  $\mu_\beta=0$  if and only if $\mu'_\beta=0$,
the proof is complete.
 \end{proof}

   We say that  an $n$-tuple $T=(T_1,\ldots, T_k)\in B(\cH)^n $ is {\it pure} if $\text{\rm SOT-}\lim_{k\to \infty}\sum_{\alpha\in \FF_n^+, |\alpha|=k} T_\alpha T_\alpha^*=0$. 
   In what follows, we show that  a weighted left multi-shift is jointly similar to a row contraction if and only  if  it is power bounded.

 \begin{theorem}  \label{power bounded} Let $W=(W_1,\ldots, W_n)$  be  a  weighted left multi-shift  associated with the weight sequence    $\boldsymbol\mu=\{\mu_\beta\}_{  |\beta|\geq 1}$. 
 Then the following statements are equivalent.
 \begin{enumerate}
 \item[(i)] $W$ is a row power bounded $n$-tuple of operators.
 \item[(ii)] There exists $M>0$ such that 
 $$\sup_{\alpha, \beta\in \FF_n^+, |\beta|=k} \boldsymbol\mu (\beta,\alpha)\leq M\quad \text{  for any } \ k\in \NN,
 $$
where  $ \boldsymbol\mu (\beta,\alpha)$ is defined in Proposition \ref{radius}.
 \item[(iii)] $W$ is jointly similar to a row contraction.
 \item[(iv)]  $W$ is jointly similar to a pure row contraction.
  \item[(v)]  $W$ is   completely polynomially  bounded with respect to the left creation 
  operators $S_1,\ldots, S_n$.
  \item[(vi)]  $W$ is     polynomially  bounded with respect to the left creation operators.
 \end{enumerate}
 In this case, the operator $A$ implementing  the similarity can be chosen such that $\|A\|\|A^{-1}\|\leq  M$.
 \end{theorem}
 \begin{proof}  First, we note the the equivalence of (i) with (ii) is due to the proof of Proposition \ref{radius}, where we showed that
 $$
\left\|\sum_{\beta\in \FF_n^+, |\beta|=k} W_\beta W_\beta^*\right\|^{1/2}=\sup_{\alpha, \beta\in \FF_n^+, |\beta|=k} \boldsymbol\mu (\beta,\alpha),\qquad k\in \NN.
$$
In order to prove that (ii)$\implies$(iii), we  consider first the case when $W$ is either an injective   left multi-shift or  a truncation of  an  injective   left multi-shift.
 Now, assume that there exists $M>0$ such that 
 $\sup_{\alpha, \beta\in \FF_n^+, |\beta|=k} \boldsymbol\mu (\beta,\alpha)\leq M$   for any  $  k\in \NN$.
 Then 
 \begin{equation}
 \label{M}
 \mu_{g_{i_k}\cdots g_{i_1}\alpha} \mu_{g_{i_{k-1}}\cdots g_{i_1}\alpha} \cdots \mu_{g_{i_1}\alpha} \leq M
  \end{equation}
 for any $\alpha\in \FF_n^+$, $g_{i_1},\ldots, g_{i_k}\in \{1,\ldots, n\}$, and  $ k\in \NN$.
In what follows, we define inductively  a sequence of weights ${\bf v}=\{v_{\beta}\}_{|\beta|\geq 1}$ such that   
\begin{equation} \label{simi}
1\leq \Gamma(g_{i_k}\cdots g_{i_1}):=\frac{\mu_{g_{i_k}\cdots g_{i_1}} \mu_{g_{i_{k-1}}\cdots g_{i_1}} \cdots \mu_{g_{i_1}}}{v_{g_{i_k}\cdots g_{i_1}} v_{g_{i_{k-1}}\cdots g_{i_1}}   \cdots v_{g_{i_1}}}\leq M
\end{equation}
for any $\sigma=g_{i_k}\cdots g_{i_1}\in \FF_n^+$ and $k\in \NN$ such that $\mu_\sigma\neq 0$.  First, we define $v_{g_{i_1}}$ for any $i_1\in \{1,\ldots, n\}$. There are to cases to consider. If $\mu_{g_{i_1}}=0$, we set $v_{g_{i_1}}=0$.  If  $\mu_{g_{i_1}}\neq 0$
we set
$$
v_{g_{i_1}}:=\begin{cases}
1& \text{if } \ \mu_{g_{i_1}}\geq 1\\
\mu_{g_i}& \text{if }\   \mu_{g_{i_1}} < 1
\end{cases} 
\   \text{ and note that  }\ 
 \Gamma(g_{i_1})=\begin{cases}
 \mu_{g_{i_1}}& \text{if } \ \mu_{g_{i_1}}\geq 1\\
1& \text{if }\   \mu_{g_{i_1}} < 1.
\end{cases} 
$$
Assume that $u_{g_{i_{k+1}}\cdots g_{i_1}}>0, u_{g_{i_{k}}\cdots g_{i_1}}>0,   \ldots, u_{g_{i_1}}>0$  and  that we have defined the weights $v_{g_{i_k}\cdots g_{i_1}}$, $v_{g_{i_{k-1}}\cdots g_{i_1}}$,  $ \ldots$, $v_{g_{i_1}}$.
Now, we define
$$
v_{g_{i_{k+1}}\cdots g_{i_1}}
:=\begin{cases}
1& \text{ if } \quad \mu_{g_{i_{k+1}}\cdots g_{i_1}}\Gamma(g_{i_k}\cdots g_{i_1})\geq 1\\
\mu_{g_{i_{k+1}}\cdots g_{i_1}}\Gamma(g_{i_k}\cdots g_{i_1})& \text{ if }\quad   \mu_{g_{i_{k+1}}\cdots g_{i_1}}\Gamma(g_{i_k}\cdots g_{i_1}) < 1.
\end{cases} 
$$
On the other hand, if $\mu_\sigma=0$, then we set $v_\sigma:= 0$.
In what follows, we prove by induction over $k\in \NN$ that, for any $g_{i_1},\ldots, g_{i_k}\in \FF_n^+$   such that $\mu_{g_{i_k}\cdots g_{i_1}}\neq 0$,   $\Gamma(g_{i_k}\cdots g_{i_1})$ is always  a product  of one the following  forms
$$
\begin{cases} \mu_{g_{i_k}\cdots g_{i_{p+1}}g_{i_p}\cdots g_{i_1}}  \mu_{g_{i_{k-1}}\cdots g_{i_{p+1}}g_{i_p}\cdots g_{i_1}}\cdots \mu_{g_{i_{p+1}g_{i_p}\cdots g_{i_1}}}, &\   \text { where } \ p<k, \text{ or }\\
\mu_{g_{i_k}\cdots g_{i_1}}, &  
\end{cases} 
$$
or  $\Gamma(g_{i_k}\cdots g_{i_1})=1$, and 
$1\leq \Gamma(g_{i_k}\cdots g_{i_1})\leq M$.

Indeed, due to relation \eqref{M} and the definition of $v_{g_{i_1}}$, we know that $ \Gamma(g_{i_1})$ satisfies the conditions  above if $ v_{g_{i_1}}\neq 0$ and also that  
$1\leq  \Gamma(g_{i_1})\leq M$. Assume that  $\mu_{g_{i_k}\cdots g_{i_1}}\neq 0$ and  $ \Gamma(g_{i_k}\cdots g_{i_1})$ satisfies the conditions  above and    $1\leq \Gamma(g_{i_k}\cdots g_{i_1})\leq M$.  Assume that $\mu_{g_{i_{k+1}}\cdots g_{i_1}}\neq 0$.
Since
$$
\Gamma(g_{i_{k+1}}\cdots g_{i_1})=\frac{\mu_{g_{i_{k+1}}\cdots g_{i_1}}}{v_{g_{i_{k+1}}\cdots g_{i_1}}} \cdot \Gamma(g_{i_{k}}\cdots g_{i_1}), 
$$
we can take  into account the definition of  $v_{g_{i_{k+1}}\cdots g_{i_1}}
$ and the induction hypothesis, to  deduce that 
$$
\Gamma(g_{i_{k+1}}\cdots g_{i_1})
:=\begin{cases}
\mu_{g_{i_{k+1}}\cdots g_{i_1}} \Gamma(g_{i_{k}}\cdots g_{i_1})& \text{ if } \quad \mu_{g_{i_{k+1}}\cdots g_{i_1}}\Gamma(g_{i_k}\cdots g_{i_1})\geq 1\\
1& \text{ if }\quad   \mu_{g_{i_{k+1}}\cdots g_{i_1}}\Gamma(g_{i_k}\cdots g_{i_1}) < 1
\end{cases} 
$$
and, consequently,  $ \Gamma(g_{i_{k}+1}\cdots g_{i_1})$ satisfies the conditions  above. Using     relation \eqref{M},  we conclude that     $1\leq \Gamma(g_{i_{k+1}}\cdots g_{i_1})\leq M$  for any $g_{i_1},\ldots, g_{i_{k+1}}\in \FF_n^+$ such that $\mu_{g_{i_{k+1}}\cdots g_{i_1}}\neq 0$. This proves our assertion.  

Since relation \eqref{simi} holds, we apply  Theorem  \ref{similarity} and  deduce that $W=(W_1,\ldots, W_n)$ is jointly similar to the weighted left multi-shift $V=(V_1,\ldots, V_n)$ associated with the weights
  ${\bf v}=\{v_{\beta}\}_{|\beta|\geq 1}$. On the other hand, since $v_{\beta}\leq 1$ for any 
  $|\beta|\geq1$,  Proposition \ref{proprieties} shows that $V$ is a row contraction. Therefore, item  (iii) holds. Moreover, due to Theorem \ref{similarity},  the operator $A$ implementing  the similarity can be chosen such that $\|A\|\|A^{-1}\|\leq   M$.
 
 Now, in order to prove the implication (ii)$\implies$(iii)   when $W$ is an arbitrary weighted  left multi-shift, we apply Theorem \ref{decompo} and decompose  $W$     into  a direct sum $\bigoplus_n W^{(n)}:=(\bigoplus_n W^{(n)}_1,\ldots, \bigoplus_n W^{(n)}_n)$ of injective   left multi-shifts and truncations of injective   left multi-shifts.    Applying the result above  to each of the truncated multi-shifts $W^{(n)}:=(W^{(n)}_1,\ldots, W^{(n)}_n)$ in this direct sum, we find row contractions 
 $C^{(n)}:=(C^{(n)}_1,\ldots, C^{(n)}_n)$ and invertible operators  $A^{(n)}$ such that
 $W^{(n)}_i=(A^{(n)})^{-1} C^{(n)}_i A^{(n)}$ for any $i\in \{1,\ldots, n\}$, and such that $\|A^{(n)}\|\|(A^{(n)})^{-1}\|\leq   M$   for all $n$. Taking
  $C:=(\bigoplus_n C^{(n)}_1,\ldots , \bigoplus_n C^{(n)}_n)$ and
    $A:=\bigoplus_n  A^{(n)}_n$, it is clear that $A$ is invertible  $C$ is a row contraction and 
    $$
    \bigoplus_n W^{(n)}_i=\left(\bigoplus_n A^{(n)}\right)^{-1}   \left(\bigoplus_n C^{(n)}_i\right) \left(\bigoplus_n A^{(n)}\right),\qquad i\in \{1,\ldots, n\},
    $$
which completes the proof.

   To prove the implication (iii)$\implies$(i), let $X$ be an invertible operator  and $W_i =X^{-1}V_i  X$  for any $i\in \{1,\ldots, n\}$, where $V=(V_1,\ldots, V_n)$ is a row contraction. Note that 
\begin{equation*}
\begin{split}
\left\|\sum_{\beta\in \FF_n^+, |\beta|=k} W_\beta W_\beta^*\right\|
&=  
\left\|\sum_{\beta\in \FF_n^+, |\beta|=k}X^{-1} V_\beta X X^* V_\beta^* (X^{-1})^*\right\|\\
&\leq
 \|X\|^2 \|X^{-1}\|^2 \left\|\sum_{\beta\in \FF_n^+, |\beta|=k} V_\beta   V_\beta^*  \right\|\leq  \|X\|^2 \|X^{-1}\|^2
\end{split}
\end{equation*}
for any $k\in \NN$. Therefore, item (i) holds.

 Due to the definition of the weighted left multi-shift, for any $\alpha\in \FF_n^+$, we have 
 $
 \sum_{\beta\in \FF_n^+, |\beta|=k} W_\beta W_\beta^* e_\alpha=0
$
if $k>|\alpha|$. If we assume that $\left\|\sum_{\beta\in \FF_n^+, |\beta|=k} W_\beta W_\beta^*\right\|^{1/2}\leq M$ for any $k\in \NN$, then a standard approximation  argument shows that  
$$
\text{\rm SOT-}\lim_{k\to \infty} \sum_{\beta\in \FF_n^+, |\beta|=k} W_\beta W_\beta^*=0,
$$
Conversely, if the later relation holds, then using the principle of uniform boundedness,
we conclude that  there is $M>0$ such that $\left\|\sum_{\beta\in \FF_n^+, |\beta|=k} W_\beta W_\beta^*\right\|^{1/2}\leq M$ for any $k\in \NN$.
 This proves  that item  (i) is equivalent to (iv).

The equivalence of item (iii) with (v) was proved in \cite{Po-disc} using  the noncommutative von Neumann inequality for row contractions \cite{Po-von}  and Paulsen's similarity criterion \cite{Pa}.     Since the implications (v)$\implies$(vi) and  (vi)$\implies$(i)  are obvious,
the proof is complete.
 \end{proof}
We remark that, according  to the proof of Theorem \ref{power bounded},  a  weighted left multi-shift  is pure if and only if it is row power bounded.

  We recall that the noncommutative disc algebra $\cA_n$ is the norm-closed  algebra generated by the left creation operators $S_1,\ldots, S_n$ and the identity (see \cite{Po-disc}).

 \begin{corollary}  If $W=(W_1,\ldots, W_n)$  is   a row power bounded injective weighted left multi-shift with bound $M>0$, then the map $\Phi_W:\cA_n\to B(F^2(H_n)$ defined by $\Phi_W(p(S_1,\ldots, S_n)):=p(W_1,\ldots, W_n)$ is completely bounded and $\|\Phi_W\|_{cp}\leq M$.
 \end{corollary}
 \begin{proof}  Due to Thorem \ref{power bounded}, the weighted  multi-shilt  $W:=(W_1,\ldots, W_n)$ is  similar to a row contraction $T=(T_1,\ldots, T_n)$  and the operator $A$ implementing  the similarity can be chosen such that $\|A\|\|A^{-1}\|\leq  M$. 
 Hence, $\|[p_{st}(W_1,\ldots, W_n)]\| \leq M \|[p_{st}(T_1,\ldots, T_n)]\|$ for any  matrices $[p_{st}(Z_1,\ldots, Z_n)]$ of polynomials in  
 $  \CC\left<Z_1,\ldots, Z_n\right>$, the polynomial algebra in  noncommutative indeterminates $Z_1,\ldots, Z_n$. Applying the  noncommutative von Neumann inequality \cite{Po-von}, \cite{Po-disc} we obtain
 $$
 \|[p_{st}(W_1,\ldots, W_n)]\| \leq M \|[p_{st}(S_1,\ldots, S_n)]\|.
 $$
 The extension  to the noncommutative disc algebra $\cA_n$ follows  easily by  a standard approximation   argument.
 \end{proof}

We denote by $C^*(W)$ the $C^*$-algebra generated by $W_1,\ldots W_n$ and the  identity. 

\begin{theorem}  \label{irreducible} If $W=(W_1,\ldots, W_n)$  is  a weighted left multi-shift  associated with  $\{\mu_\beta\}_{ |\beta|\geq 1}$, then  the following statements hold.
\begin{enumerate}
\item[(i)] The $C^*$-algebra $C^*(W)$ is irreducible if and only if $\mu_\beta>0$  for any $\beta\in \FF_n^+ $. 
\item[(ii)] If   $W$ is injective and   there is $i\in \{1,\ldots, n\}$ such that
$$
\lim_{\alpha\to \infty} \left(\mu_{g_i\alpha}-\mu_{\alpha}\right) =0\quad   
$$
then  the $C^*$-algebra $C^*(W)$ contains all the compact operators in $B(F^2(H_n))$. In particular, if $\{\mu_\alpha\}_\alpha$ is a convergent sequence, then the condition above holds for any  $i\in \{1,\ldots, n\}$.
   \end{enumerate}
\end{theorem}
\begin{proof} If $W$ is not injective, then there exists $\beta\in \FF_n^+$ with  $|\beta|\geq 1$ such that $\mu_\beta=0$. As we saw in the proof of Theorem \ref{decompo}, in this case, 
  there exist   nontrivial reducing subspaces   under $W_1,\ldots, W_n$.

Now, assume that $W$ is  injective.
Let $A\in B(F^2(H_n))$  be  commuting with each operator in $C^*(W)$.
Since $A(1)\in F^2(H_n)$, we have
$A(1)=\sum_{\beta\in \FF_n^+} c_{
\beta}\boldsymbol\mu(\beta,g_0) e_{ \beta}$ for some
coefficients $\{c_\beta\}_{\FF_n^+}\subset \CC$ with $\sum_{\beta\in\FF_n^+}
|c_\beta|^2\boldsymbol\mu(\beta,g_0)^2<\infty$. On the other hand, since $
AW_i=W_iA$ for any  $i\in \{1,\ldots,n\}$, we deduce that 
\begin{equation*}
\begin{split}
Ae_\alpha &=\frac{1}{\boldsymbol\mu(\alpha,g_0)}AW_{\alpha}(1)=\frac{1}{\boldsymbol\mu(\alpha,g_0)}W_{\alpha}
A(1)\\
&=\sum_{\beta\in \FF_n^+} c_{ \beta}
\frac{\boldsymbol\mu(\beta,g_0)\boldsymbol\mu(\alpha,\beta)}{ \boldsymbol\mu(\alpha,g_0)} e_{\alpha\beta}\\
&=\sum_{\beta\in \FF_n^+} c_{ \beta}
\frac{\boldsymbol\mu(\alpha\beta, g_0)}{ \boldsymbol\mu(\alpha,g_0)} e_{\alpha\beta}. 
 \end{split}
\end{equation*}
Similarly, since  $A^*W_i=W_iA^*$ for  any $i\in\{1,\ldots,n\}$, there are  some scalars
 $\{c_\gamma'\}_{\gamma\in \FF_n^+}\subset \CC$ with $\sum_{\gamma\in\FF_n^+}
|c_\gamma' |^2 \boldsymbol\mu(\gamma,g_0)^2<\infty$   such that 
$$
A^* e_\sigma=\sum_{\gamma\in \FF_n^+} c_{ \gamma}'
\frac{\boldsymbol\mu(\sigma\gamma, g_0)}{ \boldsymbol\mu(\sigma,g_0)} e_{\sigma\gamma}.
 $$
Since $\left< Ae_\alpha, e_\sigma\right>=\left< e_\alpha, A^* e_\sigma\right>$, the relations above imply
$$
\left< \sum_{\beta\in \FF_n^+} c_{ \beta}
\frac{\boldsymbol\mu(\alpha\beta, g_0)}{ \boldsymbol\mu(\alpha,g_0)} e_{\alpha\beta}, e_\sigma\right>=\left< e_\alpha, \sum_{\gamma\in \FF_n^+} c_{ \gamma}'
\frac{\boldsymbol\mu(\sigma\gamma, g_0)}{ \boldsymbol\mu(\sigma,g_0)} e_{\sigma\gamma}\right>
$$
for any $\alpha,\sigma\in \FF_n^+$.
Assume that $\sigma=\alpha\beta_0$ where $\beta_0\in \FF_n^+$ and $|\beta_0|\geq 1$.
The relation above becomes
$$
c_{\beta_0}\frac{\boldsymbol\mu(\alpha\beta_0, g_0)}{ \boldsymbol\mu(\alpha,g_0)} =\left< e_\alpha,\sum_{\gamma\in \FF_n^+} c_{ \gamma}'
\frac{\boldsymbol\mu(\alpha\beta_0\gamma, g_0)}{ \boldsymbol\mu(\alpha\beta_0,g_0)} e_{\alpha\beta_0\gamma}\right>=0.
$$
Hence $c_{\beta_0}=0$ for any  $\beta_0\in \FF_n^+$ with $|\beta_0|\geq 1$.
Therefore, $Ae_\alpha=c_{g_0}e_\alpha$ for any $\alpha\in \FF_n^+$, thus $A=c_{g_0}I$. This shows that 
$C^*(W)$ is irreducible. Consequently, item (i) holds.
 
 To prove item (ii),  note that, due to the proof of Proposition \ref{radius}, we have 
 $\left(\sum_{j=1}^n W_jW_j^*\right)e_\alpha =\mu_\alpha^2 e_\alpha$ if $\alpha\in \FF_n^+$ with $|\alpha|\geq 1$ and  $0$ otherwise. On the other hand, $W_i^*W_ie_\alpha= \mu_{g_i\alpha}^2 e_\alpha$ for any $\alpha\in \FF_n^+$. Consequently,  
  \begin{equation*}
 \left( W_i^*W_i-\sum_{j=1}^n W_jW_j^*\right)
e_\alpha =\begin{cases}(\mu_{g_i\alpha}^2-\mu_\alpha^2)e_\alpha& \text{ if
}
|\alpha|\geq 1 \\
\mu_{g_i}& \text{ if } \alpha=g_0.
\end{cases}
\end{equation*}
Now, it clear that the  condition 
  $
\lim_{\alpha\to \infty} \left(\mu_{g_i\alpha}^2-\mu_{\alpha}^2\right) =0$  is equivalent to the fact that the diagonal operator    
  $W_i^*W_i-\sum_{j=1}^n W_jW_j^*$ is    compact. 
Since the $C^*$-algebra $C^*(W)$ is irreducible, we conclude that it  contains all the compact operators in $B(F^2(H_n))$.
This completes the proof.
\end{proof}

Due  the standard theory of
representations of  $C^*$-algebras \cite{Arv-book}, we obtain the following  Wold  \cite{W} decomposition for  the unital 
$*$-representations  of  the $C^*$-algebra $C^*(W)$. 

\begin{theorem}  \label{wold}  Let $W=(W_1,\ldots, W_n)$  be  an injective     weighted left multi-shift  associated with  the weights $\{\mu_\beta\}_{ |\beta|\geq 1}$ having the property that there is a nonempty subset  $\Omega\subset \{1,\ldots, n\}$  such that  
$$
\lim_{\alpha\to \infty} \left(\mu_{g_i\alpha}-\mu_{\alpha}\right) =0,\qquad i\in\Omega.
$$
 If   
$\pi:C^*(W)\to B(\cK)$ is  a unital
$*$-representation  of $C^*(W)$ on a separable Hilbert
space  $\cK$, then $\pi$ decomposes into a direct sum
$$
\pi=\pi_0\oplus \pi_1 \  \text{ on  } \ \cK=\cK_0\oplus \cK_1,
$$
where $\pi_0$ and  $\pi_1$  are disjoint representations of
$C^*(W)$ on the Hilbert spaces
$$\cK_0:=\overline{\text{\rm span}}\{\pi(X)\cK:\ X\in  \boldsymbol\cK(F^2(H_n))\}
\quad \text{ and  }\quad  \cK_1:=\cK_0^\perp,
$$
 respectively, where $\boldsymbol\cK(F^2(H_n))$ is the ideal of compact operators in $B(F^2(H_n))$. Moreover, up to an isomorphism,
\begin{equation*}
\cK_0\simeq F^2(H_n)\otimes \cG, \quad  \pi_0(X)=X\otimes I_\cG \quad
\text{ for } \  X\in C^*(W),
\end{equation*}
 and $\pi_1$ is a $*$-representation  which annihilates the compact operators  in $C^*(W)$ and 
 $$
  V_i^*V_i=\sum_{j=1}^n V_jV_j^*,\qquad i\in \Omega,
  $$
  where $V_j:=\pi_1(W_j)$.
    If $\pi'$ is another $*$-representation  of $C^*(W)$  on a separable Hilbert space $\cK'$, then $\pi$ is unitarily equivalent to  $\pi'$ if and only if $\dim \cG=\dim \cG'$ and $\pi_1$ is unitarily equivalent to $\pi_1'$.
 \end{theorem}
 \begin{proof}
 According to Theorem \ref{irreducible},   all the
compact operators $ \boldsymbol\cK(F^2(H_n))$ in $B(F^2(H_n))$ are contained in the
$C^*$-algebra $C^*(W)$.
 Due to  standard theory of
representations of  $C^*$-algebras,
the representation $\pi$ decomposes into a direct sum
$\pi=\pi_0\oplus \pi_1$ on $ \cK=\cK_0\oplus \cK_1$,
where $\pi_0$, $\pi_1$  are disjoint representations of
$C^*({\bf W})$ on the Hilbert spaces
$$\cK_0:=\overline{\text{\rm span}}\{\pi(X)\cK:\ X\in  \boldsymbol\cK(F^2(H_n))\}
\quad \text{ and  }\quad  \cK_1:=\cK_0^\perp,
$$
respectively,
such that
 $\pi_1$ annihilates  the compact operators in $B(F^2(H_n))$ and
  $\pi_0$ is uniquely determined by the action of $\pi$ on the
  ideal $\boldsymbol \cK(F^2(H_n))$ of compact operators.
Since every representation of $  \boldsymbol\cK(F^2(H_n))$ is equivalent to a
multiple of the identity representation, we deduce
 that
\begin{equation*}
\cK_0\simeq\cN_J\otimes \cG, \quad  \pi_0(X)=X\otimes I_\cG, \quad
X\in C^*({\bf W}),
\end{equation*}
 for some Hilbert space $\cG$.
 The uniqueness of the decomposition is  also due 
  to the standard theory of representations of  $C^*$-algebras.
  \end{proof}
We remark that taking  $\mu_\alpha=1$ for any $\alpha\in \FF_n^+$ in Theorem  \ref{wold},
we find again the Wold decomposition for row isometries \cite{Po-isometric}.  
  
 \begin{corollary}  Let $W=(W_1,\ldots, W_n)$  be  an injective weighted left multi-shift  associated with  the weights $\{\mu_\beta\}_{|\beta|\geq 1 }$ having the property that 
$\mu_\alpha\to 1$.  If   
$\pi:C^*(W)\to B(\cK)$ is  a unital
$*$-representation  of $C^*(W)$ on a separable Hilbert
space  $\cK$  and $
\pi=\pi_0\oplus \pi_1$ is the Wold decomposition of Theorem \ref{wold}, then
$[\pi_1(W_1)\cdots \pi_1(W_n)]$ is a Cuntz isometry. 
 \end{corollary} 
  \begin{proof} Set $V_j:=\pi_1(W_j)$ for $j\in \{1,\ldots, n\}$. Due to Theorem \ref{wold}, we have   $V_i^*V_i=\sum_{j=1}^n V_jV_j^*$ for any $i\in \{1,\ldots, n\}$. On the other hand,
we have
 \begin{equation*}
 \left(I-\sum_{j=1}^n W_jW_j^*\right)
e_\alpha =\begin{cases}(1-\mu_\alpha^2)e_\alpha& \text{ if
}
|\alpha|\geq 1 \\
1& \text{ if } \alpha=g_0.
\end{cases}
\end{equation*}  
Since $\mu_\alpha\to 1$,    the operator $I-\sum_{j=1}^n W_jW_j^*$ is compact and using again Theorem \ref{wold}, we deduce that $\sum_{j=1}^n V_j V_j^*=I$. This completes the proof.
  \end{proof}

A   unital
$*$-representation   $\pi$ of $C^*(W)$ on a separable Hilbert
space  $\cK$ is said to be of   Cuntz type  \cite{Cu} if   $\cK_0=\{0\}$ in  the Wold decomposition of $\pi$.

\begin{theorem} \label{C*}   Let $W=(W_1,\ldots, W_n)$  be  an injective weighted left multi-shift and   let $\pi$ be   a  unital
$*$-representation  of $C^*(W)$ on a separable Hilbert
space  $\cK$. If   $V=(V_1,\ldots, V_n)$, where $V_{i}:=\pi({W}_{i})$, and 
   $\pi$ is not a Cuntz type  $*$-representation, then  the $C^*$-algebras $C^*(W)$  and   $C^*(V)$ are $*$-isomorphic.
\end{theorem}

\begin{proof} According to  Theorem \ref{wold}, we have the Wold decomposition $\cK=\cK^{(0)}\oplus \cK^{(1)}$ and 
$$
V_{i}=({W}_{i}\otimes I_\cD)\bigoplus V_{i}',\qquad i\in  \{1,\ldots, k\},
$$
 where   $V'_{i}:=V_{i}|_{\cK^{(1)}}$  and 
$ \cD$ is a Hilbert space with $\dim \cD\geq 1$.
  As a consequence, we deduce that
 $$
 q(V, V^*)=\left(q(W, W^*)\otimes I_\cD\right)
 \bigoplus q(V',
  V'^*)
 $$
 for any noncommutative polynomial  $q ({Z}, {Y})$  in noncommutative indeterminates ${Z}=\{Z_1,\ldots, Z_n\}$ and ${Y}=\{Y_1,\ldots, Y_n\}$.   Since   $\pi|_{\cK^{(1)}}$ is a $*$-representation,  we have
$
\|q(V', V'^*)\|\leq  \|q(W, W^*)\|.
$
On the other hand, since  $\cD\neq \{0\}$, we have
$$
\| q(V, V^*)\|=\max\left\{ \| q(W, W^*))\|,
 \|  q(V',V'^*))\|\right\}=\| q(W, W)\|.
$$
 Due to the fact that  $\pi: C^*(W)\to C^*(V)$ is surjective,  we have  
$\|\pi(g)\|=\|g\|$ for all $g\in  C^*(W)$. Therefore,   $\pi$ is a  $*$-isomorphism of $C^*$-algebras.
The proof is complete.
\end{proof}

 \bigskip
 
 \section{Joint similarity of  tuples of operators  to parts of  weighted shifts}

Throughout the remainder of this  paper, unless otherwise specified, we assume that the  weighted multi-shifts  are injective, i.e.  all the   weights are strictly positive.  
Let 
$\cA(\boldsymbol\mu)$  be the norm-closed nonself-adjoint algebra generated by the weighted  left  multi-shifts $W_1,\ldots, W_n$ and the identity.

 \begin{theorem} \label{Rota}  Let  $\boldsymbol\mu=\{\mu_\beta\}_{ |\beta|\geq 1}$ be a  weight sequence and  let  $W=(W_1,\ldots, W_n)$  be the associated    weighted left multi-shift.  If $T=(T_1,\ldots, T_n) \in B(\cH)^n$  and   there is a positive invertible operator $Q\in B(\cH)$  and positive constants $0<a\leq b$ such that 
$$
a I\leq  \sum_{k=0}^\infty\sum_{\alpha\in \FF_n^+, |\alpha|=k} \frac{1}{\boldsymbol\mu (\alpha,g_0)^2} T_\alpha Q T_\alpha^*\leq bI,
$$
then 
    $(T_1,\ldots, T_n)$ is jointly similar to  $$(P_\cM(W_1\otimes I_\cH)|_\cM,\ldots, P_\cM(W_n\otimes I_\cH)|_\cM),$$
 where  $\cM\subset F^2(H_n)\otimes \cH$  is a joint invariant subspace under the operators $W_i^*\otimes I_\cH$,  $i\in \{1,\ldots, n\}$.
 In this case, we have $r(T)\leq r(W)$ and   the map $\Phi_T:\cA(\boldsymbol\mu)\to B(F^2(H_n))$ defined by 
 $$\Phi_T(p(W_1,\ldots, W_n)):=p(T_1,\ldots, T_n)$$
  is completely bounded and $\|\Phi_T\|_{cb}\leq \sqrt{\frac{b}{a}} $.
 \end{theorem}
 \begin{proof}
  
  Define the operator
 $K_{\boldsymbol\mu}:\cH\to F^2(H_n)\otimes \cD$ by setting
 $$
 K_{\boldsymbol\mu} h:= \sum_{\alpha\in \FF_n^+} \frac{1}{\boldsymbol\mu (\alpha,g_0)} e_\alpha \otimes Q^{1/2}T_\alpha^* h, \qquad h\in \cH, 
 $$
 where $\cD:=\overline{Q^{1/2}\cH}$ and $\boldsymbol\mu (\alpha,g_0)$ is defined in Proposition \ref{radius}.
 Since 
 $$
 \|K_{\boldsymbol\mu} h\|^2=\sum_{k=0}^\infty\sum_{\alpha\in \FF_n^+, |\alpha|=k} \frac{1}{\boldsymbol\mu (\alpha,g_0)^2} \left<T_\alpha Q T_\alpha^*h,h\right>,
 $$
 it is clear that  $K_{\boldsymbol\mu}$  is  a bounded operator.   On the other hand, since  $\boldsymbol\mu (g_0,g_0):=1$, we  also have $\|K_{\boldsymbol\mu} h\|\geq \|Q^{1/2}h\|\geq \frac{1}{\|Q^{-1/2}\|}\|h\|$, which shows that $K_\mu$ has closed range.
 Since 
 $$W_i^* e_\alpha =\begin{cases}
\mu_{g_i \gamma}e_\gamma& \text{ if }
\alpha=g_i\gamma \\
0& \text{ otherwise }
\end{cases}
 $$
 for every $\alpha  \in \FF_n^+$ and  every $i\in \{1,\ldots, n\}$,
 we deduce that
 \begin{equation*}
 \begin{split}
(W_i^*\otimes I)K_{\boldsymbol\mu} h&=\sum_{\alpha\in \FF_n^+} \frac{1}{\boldsymbol\mu (\alpha,g_0)} W_i^* e_\alpha \otimes  Q^{1/2}T_\alpha^* h\\
&= \sum_{\gamma\in \FF_n^+} \frac{1}{\boldsymbol\mu (g_i\gamma, g_0)}  \mu_{g_i \gamma} e_\gamma\otimes Q^{1/2} T_{g_i \gamma}^* h.
 \end{split}
 \end{equation*}
Taking into account that $K_{\boldsymbol\mu} T_i^*h=  \sum_{\gamma\in \FF_n^+} \frac{1}{\boldsymbol\mu (\gamma,g_0)} e_\gamma\otimes Q^{1/2}T_\gamma^* T_i^*h$ and $\boldsymbol\mu(g_i\gamma, g_0)= \boldsymbol\mu(\gamma, g_0) \mu_{g_i\gamma} $, we conclude that
$$
K_{\boldsymbol\mu} T_i^*=(W_i^*\otimes I)K_{\boldsymbol\mu},\qquad i\in \{1,\ldots, n\}.
$$
Consequently,  $\cM:=K_{\boldsymbol\mu} \cH$  is a joint invariant subspace under   the operators $W_i^*\otimes I$, where $i\in \{1,\ldots, n\}$.  Define $X:\cH\to \cM$ by setting $Xh:=K_{\boldsymbol\mu} h$, $h\in \cH$. Note that $X$ is an invertible operator  and $T_i^*=X^{-1}(W_i^*\otimes I)|_\cM X$  for any $i\in \{1,\ldots, n\}$ and 
 $ \|X^{-1}\|\|X\|\leq \sqrt{\frac{b}{a}}$. Hence, it is easy to see that  the map $\Phi_T:\cA(\boldsymbol\mu)\to B(F^2(H_n))$ defined by 
 $$\Phi_T(p(W_1,\ldots, W_n)):=p(T_1,\ldots, T_n)$$
  is completely bounded and $\|\Phi_T\|_{cp}\leq \sqrt{\frac{b}{a}} $.
On the other hand, we have
 \begin{equation*}
 \begin{split}
\left\| \sum_{|\alpha|=k} T_\alpha T_\alpha^*\right\|&=\left\|X^* \left(\sum_{|\alpha|=k}P_\cM W_\alpha |_\cM  (X^*)^{-1}X^{-1} P_\cM W_\alpha^*|_\cM\right)X\right\|\\
&\leq
 \|X^{-1}\|^2\|X\|^2
 \left\|\sum_{|\alpha|=k}P_\cM W_\alpha   W_\alpha^*|_\cM\right\|
 \leq  \|X^{-1}\|^2\|X\|^2
 \left\|\sum_{|\alpha|=k}W_\alpha   W_\alpha^*\right\|
 \end{split}
 \end{equation*}
 for any $k\in \NN$.
Hence, we deduce that  $r(T)\leq r(W)$. 
 This completes the proof.
 \end{proof}

\begin{corollary} \label{T1}
Let  $\boldsymbol\mu=\{\mu_\beta\}_{ |\beta|\geq 1}$ be a  weight sequence and  let  $W=(W_1,\ldots, W_n)$  be the associated    weighted left multi-shift.
 If
 $T=(T_1,\ldots, T_n)\in B(\cH)^n$  is  such that  the series 
$$
\Sigma_{\boldsymbol\mu}(T,T^*):=\sum_{k=0}^\infty\sum_{\alpha\in \FF_n^+, |\alpha|=k} \frac{1}{\boldsymbol\mu (\alpha,g_0)^2} T_\alpha T_\alpha^*
$$
is convergent in the weak operator topology,
then there is a joint invariant subspace $\cM\subset F^2(H_n)\otimes \cH$ under the operators $W_i^*\otimes I_\cH$,  $i\in \{1,\ldots, n\}$, such that  the $n$-tuple $(T_1,\ldots, T_n)$ is jointly similar to  
$$(P_\cM(W_1\otimes I_\cH)|_\cM,\ldots, P_\cM(W_n\otimes I_\cH)|_\cM).$$
In this case,  the map $\Phi_T:\cA(\boldsymbol\mu)\to B(F^2(H_n)$ defined by 
 $$\Phi_T(p(W_1,\ldots, W_n)):=p(T_1,\ldots, T_n)$$
  is completely bounded and $\|\Phi_T\|_{cb}\leq  \|\Sigma_{\boldsymbol\mu}(T,T^*)\|^{1/2} $.
\end{corollary} 
 \begin{proof}  Taking $Q=I$ in Theorem \ref{Rota}, the result follows.
 \end{proof}

 \begin{corollary} \label{Rota1}$($\cite{Po-models}$)$  Let  
    $T=(T_1,\ldots, T_n)\in B(\cH)^n$  be an $n$-tuple of operators with the joint spectral radius $r(T)<1$ and let $S_1,\ldots, S_n$ be the left creation operators on the full Fock space.  Then  there is a joint invariant subspace $\cM\subset F^2(H_n)\otimes \cH$ under the operators $S_i^*\otimes I_\cH$,  $i\in \{1,\ldots, n\}$, such that    $(T_1,\ldots, T_n)$ is jointly similar to  
  $$(P_\cM(S_1\otimes I_\cH)|_\cM,\ldots, P_\cM(S_n\otimes I_\cH)|_\cM),$$
  where $S_1,\ldots, S_n$   are the left creation operators on the full Fock space $F^2(H_n)$. 
  In this case,  the map $\Phi_T:\cA_n\to B(F^2(H_n)$ defined by 
 $$\Phi_T(p(S_1,\ldots, S_n)):=p(T_1,\ldots, T_n)$$
  is completely bounded and $\|\Phi_T\|_{cb}\leq \left\|\sum_{k=0}^\infty\sum_{\alpha\in \FF_n^+, |\alpha|=k}  T_\alpha T_\alpha^*\right\|^{1/2} $.

 \end{corollary} 
   \begin{proof} Since $r(T)<1$, we have   $
\sum_{k=0}^\infty\left\|\sum_{\alpha\in \FF_n^+, |\alpha|=k}  T_\alpha T_\alpha^*\right\|<\infty
$.
 Note that if we take  $\mu_\beta=1$ for any $\beta\in \FF_n^+, |\beta|\geq 1$,  then $\boldsymbol\mu (\alpha,g_0)=1$, $\alpha\in \FF_n^+$, and the condition in Corollary \ref{T1} is satisfied.     
 Consequently, the results follows.
\end{proof}

 \begin{corollary} Let  $\boldsymbol\mu=\{\mu_\beta\}_{ |\beta|\geq 1}$ be a  weight sequence and  let  $W=(W_1,\ldots, W_n)$  be the associated    weighted left multi-shift. If  
   $T=(T_1,\ldots, T_k)\in B(\cH)^n $ is a  row power bounded $n$-tuple of operators and 
   $$ \sum_{k=0}^\infty \max_{|\alpha|=k} \frac{1}{\boldsymbol\mu (\alpha,g_0)^2}<\infty,
 $$   
 then there is a joint invariant subspace $\cM\subset F^2(H_n)\otimes \cH$ under the operators $W_i^*\otimes I_\cH$,  $i\in \{1,\ldots, n\}$, such that    $(T_1,\ldots, T_n)$ is jointly similar to  
$$(P_\cM(W_1\otimes I_\cH)|_\cM,\ldots, P_\cM(W_n\otimes I_\cH)|_\cM).$$
 \end{corollary}
 \begin{proof}  Assume that  there is a constant $M>0$ such that $
 \left\|\sum_{\alpha\in \FF_n^+, |\alpha|=k} T_\alpha T_\alpha^*\right\|\leq M$ for any $k\in \NN$.   Then 
 $$
\sum_{k=0}^\infty\left\|\sum_{\alpha\in \FF_n^+, |\alpha|=k} \frac{1}{\boldsymbol\mu (\alpha,g_0)^2} T_\alpha T_\alpha^*\right\|\leq   \sum_{k=0}^\infty \max_{|\alpha|=k} \frac{1}{\boldsymbol\mu (\alpha,g_0)^2} M<\infty
$$
and, using Corollary \ref{T1}, we complete the proof.  The particular case is quite obvious.
 \end{proof}

 \begin{theorem}  \label{main1} If $T=(T_1,\ldots, T_n)\in B(\cH)^n$  is not a nilpotent $n$-tuple of operators, then there exists     an injective  weighted multi-shift $W=(W_1,\ldots, W_n)$ with  the following properties:
 \begin{enumerate}
\item[(i)] $\|\sum_{|\alpha|=k} W_\alpha W_\alpha^*||^{1/2}\leq (k+1) \|\sum_{|\alpha|=k} T_\alpha T_\alpha^*||^{1/2}$ for  any $k\in \NN$;
\item[(ii)] $r(W)= r(T)$;
\item[(iii)]    $(T_1,\ldots, T_n)$ is jointly similar to  $$(P_\cM(W_1\otimes I_\cH)|_\cM,\ldots, P_\cM(W_n\otimes I_\cH)|_\cM),$$
 where  $\cM\subset F^2(H_n)\otimes \cH$  is a joint invariant subspace under the operators $W_i^*\otimes I_\cH$,  $i\in \{1,\ldots, n\}$;
  \item[(iv)] $r(p(T_1,\ldots, T_n))\leq r(p(W_1,\ldots, W_n))$ and    the map $\Phi_T:\cA(\boldsymbol\mu)\to B(F^2(H_n)$ defined by 
 $$\Phi_T(p(W_1,\ldots, W_n)):=p(T_1,\ldots, T_n)$$
  is completely bounded and $\|\Phi_T\|_{cb}\leq \frac{\pi}{\sqrt{6}} $.
 \end{enumerate}
 If  $T=(T_1,\ldots, T_n)$ is a nilpotent $n$-tuple of  index $m\geq 2$,  then there is a  truncated weighted multi-shift $W=(W_1,\ldots, W_n)$ which is nilpotent of index $m$ such that all the properties above hold and 
 $$\|\Phi_T\|_{cb}\leq \sqrt{\sum_{k=0}^{m-1}\frac{1}{(k+1)^2}}.
 $$
 \end{theorem}
  \begin{proof}  First, we assume that $T=(T_1,\ldots, T_n)$ is not a nilpotent $n$-tuple of operators. Note that this is equivalent to  
  $\left\|\sum_{\alpha\in \FF_n^+, |\alpha|=k} T_\alpha T_\alpha^*\right\|\neq 0$ for any $k\in \NN$.
   For each $\beta\in \FF_n^+$ with $|\beta|\geq 1$, we define
  \begin{equation}
  \label{frac}
  \mu_\beta:=\frac{|\beta|+1}{|\beta|}\frac {\left\|\sum\limits_{\sigma\in \FF_n^+, |\sigma|=
  |\beta|}T_\sigma T_\sigma^*\right\|^{1/2}} {\left\|\sum\limits_{\sigma\in \FF_n^+, |\sigma|=|\beta|-1}
 T_\sigma T_\sigma^*\right\|^{1/2}}. 
  \end{equation}
 Let $W=(W_1,\ldots, W_n)$ be the  weighted multi-shift associated with the weights  $\{\mu_\beta\}_{ |\beta|\geq 1}.
$
 Note that 
if   $\beta=g_{i_1}\cdots g_{i_p}\in \FF_n^+$ and  $\alpha\in \FF_n^+$, then the relation  above   and the inequality
$$
\left\|\sum\limits_{\sigma\in \FF_n^+, |\sigma|=
  |\alpha|+|\beta|}T_\sigma T_\sigma^*\right\|\leq 
\left\|\sum\limits_{\sigma\in \FF_n^+, |\sigma|=
  |\alpha|}T_\sigma T_\sigma^*\right\| \left\|\sum\limits_{\sigma\in \FF_n^+, |\sigma|=
  |\beta|}T_\sigma T_\sigma^*\right\|
  $$
imply
\begin{equation*}
\begin{split}
\boldsymbol\mu (\beta,\alpha) &=\mu_{g_{i_1}\cdots g_{i_p}\alpha} \mu_{g_{i_2}\cdots g_{i_p}\alpha} \cdots \mu_{g_{i_p}\alpha}\\
&=
\frac{|\alpha|+|\beta|+1}{|\alpha|+|\beta|} \cdot 
 \frac{|\alpha|+|\beta|}{|\alpha|+|\beta|-1}\cdots  \frac{|\alpha|+2}{|\alpha|+1}
 \\
 &\times 
  \frac {\left\|\sum\limits_{\sigma\in \FF_n^+, |\sigma|=
  |\alpha|+|\beta|}T_\sigma T_\sigma^*\right\|^{1/2}} {\left\|\sum\limits_{\sigma\in \FF_n^+, 
  |\sigma|=|\alpha|+|\beta|-1}
 T_\sigma T_\sigma^*\right\|^{1/2}} \cdot
   \frac {\left\|\sum\limits_{\sigma\in \FF_n^+, |\sigma|=
  |\alpha|+|\beta|-1}T_\sigma T_\sigma^*\right\|^{1/2}} {\left\|\sum\limits_{\sigma\in \FF_n^+, 
  |\sigma|=|\alpha|+|\beta|-2}
 T_\sigma T_\sigma^*\right\|^{1/2}}\cdots 
\frac {\left\|\sum\limits_{\sigma\in \FF_n^+, |\sigma|=
  |\alpha|+1}T_\sigma T_\sigma^*\right\|^{1/2}} {\left\|\sum\limits_{\sigma\in \FF_n^+, 
  |\sigma|=|\alpha|}
 T_\sigma T_\sigma^*\right\|^{1/2}}\\
 &= \frac{|\alpha|+|\beta|+1}{|\alpha|+1} \cdot\frac {\left\|\sum\limits_{\sigma\in \FF_n^+, |\sigma|=
  |\alpha|+|\beta|}T_\sigma T_\sigma^*\right\|^{1/2}} {\left\|\sum\limits_{\sigma\in \FF_n^+, 
  |\sigma|=|\alpha|}
 T_\sigma T_\sigma^*\right\|^{1/2}}\\
 &\leq (|\beta|+1) \left\|\sum\limits_{\sigma\in \FF_n^+, |\sigma|=
   |\beta|}T_\sigma T_\sigma^*\right\|^{1/2}
\end{split}
\end{equation*}
for any $\alpha\in \FF_n^+$.
Consequently, we obtain
$$
\sup_{\alpha, \beta\in \FF_n^+, |\beta|=p} \boldsymbol\mu (\beta,\alpha)
\leq (p+1)\left\|\sum\limits_{\sigma\in \FF_n^+, |\sigma|=p}T_\sigma T_\sigma^*\right\|^{1/2}.
$$ 
Using the proof of Proposition \ref{radius}, we conclude that
$$
\left\|\sum_{\beta\in \FF_n^+, |\beta|=p} W_\beta W_\beta^*\right\|=\sup_{\alpha, \beta\in \FF_n^+, |\beta|=p} \boldsymbol\mu (\beta,\alpha)^2\leq (p+1)\left\|\sum\limits_{\sigma\in \FF_n^+, |\sigma|=p}T_\sigma T_\sigma^*\right\|, 
$$
which proves item (i), 
and 
$$
r(W)=\lim_{p\to\infty} \sup_{\alpha, \beta\in \FF_n^+, |\beta|=p} \boldsymbol\mu (\beta,\alpha)^{1/p}
\leq \lim_{p\to\infty} (p+1)^{1/p}\left\|\sum\limits_{\sigma\in \FF_n^+, |\sigma|=
   p}T_\sigma T_\sigma^*\right\|^{1/2p}=r(T).
$$ 
On the other hand, 
  we have 
$$
 \boldsymbol\mu (\beta,g_0)=\mu_{g_{i_1}\cdots g_{i_p}} \mu_{g_{i_2}\cdots g_{i_p}} \cdots \mu_{g_{i_p}}=(|\beta|+1) \left\|\sum\limits_{\sigma\in \FF_n^+, |\sigma|=
   |\beta|}T_\sigma T_\sigma^*\right\|^{1/2}
$$
for any $\beta=g_{i_1}\cdots g_{i_p}\in \FF_n^+$. Now, we deduce that
 $$
\sum_{k=0}^\infty\left\|\sum_{\alpha\in \FF_n^+, |\alpha|=k} \frac{1}{\boldsymbol\mu (\alpha,g_0)^2} T_\alpha T_\alpha^*\right\|=\sum_{k=0}^\infty  \frac{1}{(k+1)^2}<\infty.
$$
Applying Corollary \ref{T1},   we deduce  items (iii) and (iv). Hence,  we also  deduce that $r(T)\leq r(W)$.   Therefore, item (ii) holds.

Now, we  assume that $T=(T_1,\ldots, T_n)$ is a nilpotent $n$-tuple of operators. Let $m\in \NN$, $m\geq 2$,  be with the property that $T_\alpha=0$ for any $\alpha\in \FF_n^+$ with $|\alpha|=m$ and   there is $\beta\in \FF_n^+$ with $|\beta|=m-1$ such that  $T_\beta\neq 0$.
 Define $\mu_\beta>0$ as in relation \eqref{frac} for any $\beta\in \FF_n^+$ with $1\leq| \beta|\leq m-1$ and $\mu_\beta:=0$ if  $|\beta|\geq m$.   Let $W=(W_1,\ldots, W_n)$ be the  truncated weighted multi-shift associated with the weights  $\{\mu_\beta\}_{  |\beta|\geq 1}$ and note that $W$ is nilpotent $n$-tuple of operators of index $m$. In this case we have $\boldsymbol\mu(\sigma, g_0)>0$ if  $|\sigma|\leq m-1$ and  $\boldsymbol\mu(\sigma, g_0)=0$  if  $|\sigma|\geq m$. Moreover, we have  $r(W)=r(T)=0$.
 Modifying the proof of Theorem \ref{Rota}, we define the operator
 $K_\mu:\cH\to F^2(H_n)\otimes \cH$ by setting
 $$
 K_\mu h:= \sum_{\alpha\in \FF_n^+,|\alpha|\leq m-1}   \frac{1}{\boldsymbol\mu (\alpha,g_0)}e_\alpha \otimes T_\alpha^* h.\qquad h\in \cH.
 $$
Note also that, since  $T_\alpha=0$ for any $\alpha\in \FF_n^+$ with $|\alpha|=m$, and 
$\boldsymbol\mu(g_i\gamma, g_0)= \boldsymbol\mu(\gamma, g_0) \mu_{g_i\gamma} $, we have
 \begin{equation*}
 \begin{split}
(W_i^*\otimes I)K_\mu h&=\sum_{\alpha\in \FF_n^+, |\alpha|\leq m-1}   \frac{1}{\boldsymbol\mu (\alpha,g_0)} W_i^* e_\alpha \otimes T_\alpha^* h\\
&= \sum_{\gamma\in \FF_n^+, |\gamma|\leq m-2}   \frac{1}{\boldsymbol\mu (g_i\gamma, g_0)}  \mu_{g_i \gamma}  e_\gamma\otimes T_{g_i \gamma}^* h\\
&= \sum_{\gamma\in \FF_n^+, |\gamma|\leq m-1}   \frac{1}{\boldsymbol\mu (\gamma,g_0)}  e_\gamma\otimes T_{g_i \gamma}^* h\\
&=K_\mu (T_i^*h)
 \end{split}
 \end{equation*}
for any $h\in \cH$.  Since $K_\mu$ has closed range, the relation above shows that  $\cM:=K_{\boldsymbol\mu} H$  is a joint invariant subspace under  $W_i^*\otimes I$, where $i\in \{1,\ldots, n\}$.  Hence, item (iii) follows. Note that  if $\alpha, \beta\in \FF_n^+$    and $|\alpha|+|\beta|<m$, then
$$
\boldsymbol\mu (\beta,\alpha)=\frac{|\alpha|+|\beta|+1}{|\alpha|+1} \cdot\frac {\left\|\sum\limits_{\sigma\in \FF_n^+, |\sigma|=
  |\alpha|+|\beta|}T_\sigma T_\sigma^*\right\|^{1/2}} {\left\|\sum\limits_{\sigma\in \FF_n^+, 
  |\sigma|=|\alpha|}
 T_\sigma T_\sigma^*\right\|^{1/2}}
$$
and $\boldsymbol\mu (\beta,\alpha)=0$ otherwise. Now, as above, one can prove items (i) and (iv). Moreover, in this case  we have
$$
\|K_{\boldsymbol\mu} \|^2\leq \sum_{k=0}^{m-1}\left\|\sum_{\alpha\in \FF_n^+, |\alpha|=k} \frac{1}{\boldsymbol\mu (\alpha,g_0)^2} T_\alpha T_\alpha^*\right\|=\sum_{k=0}^{m-1}  \frac{1}{(k+1)^2} .
$$
Since $\|K_{\boldsymbol\mu}^{-1}\|\leq 1$, we can use item (iii) to   complete the proof.
 \end{proof}
 At the moment we don't now whether the constant $\frac{\pi}{\sqrt{6}}$ is sharp. 
\begin{corollary}   Let  $\{\mu_\beta\}_{ |\beta|\geq 1}$ be a sequence of    real numbers with $\mu_\beta>0$ if $|\beta|\leq m-1$ and $0$ otherwise, and let  $W=(W_1,\ldots, W_n)$  be the associated    weighted left multi-shift.  
If $T=(T_1,\ldots, T_n) \in B(\cH)^n$ is a nilpotent $n$-tuple    such  that $T_\alpha=0$ for any $\alpha\in \FF_n^+$ with $|\alpha|=m\geq 2$,   then $T$ is jointly similar  to 
$$(P_\cM(W_1\otimes I_\cH)|_\cM,\ldots, P_\cM(W_n\otimes I_\cH)|_\cM),$$
 where  $\cM\subset F^2(H_n)\otimes \cH$  is  a  joint invariant subspace under the operators $W_i^*\otimes I_\cH$,  $i\in \{1,\ldots, n\}$.
 As a
 consequence, any nilpotent $n$-tuple is similar to a nilpotent row operator  of arbitrarily  small norm.
\end{corollary}
\begin{proof} The proof can be extracted from the proof of Theorem \ref{main1}.
\end{proof}

Note that in  the particular case  when  $r(T)<1$,  Theorem \ref{main1} provides a  Rota type similarity result which is different from that  of Corollary \ref{Rota1}.  

 \begin{corollary}  If $T=(T_1,\ldots, T_n)\in B(\cH)^n$ is a quasi-nilpotent $n$-tuple of operators, then there exists     a   quasi-nilpotent weighted multi-shift $W=(W_1,\ldots, W_n)$ and a joint invariant subspace $\cM\subset F^2(H_n)\otimes \cH$ under the operators $W_i^*\otimes I_\cH$,  $i\in \{1,\ldots, n\}$, such that    $(T_1,\ldots, T_n)$ is jointly similar to  
$$(P_\cM(W_1\otimes I_\cH)|_\cM,\ldots, P_\cM(W_n\otimes I_\cH)|_\cM)$$
and   the map $\Phi_T:\cA(\boldsymbol\mu)\to B(F^2(H_n)$ defined by 
 $$\Phi_T(p(W_1,\ldots, W_n)):=p(T_1,\ldots, T_n)$$
  is completely bounded and $\|\Phi_T\|_{cb}\leq \frac{\pi}{\sqrt{6}} $.
In addition, if $T$ is nilpotent,  then $W$ can be chosen to be nilpotent  and the inequality above can be improved as in Theorem \ref{main1}.  
\end{corollary}

\begin{corollary}  Every quasi-nilpotent  $n$-tuple of operators  is jointly  similar to a row contraction  of arbitrarily small norm.
\end{corollary}
\begin{proof} Since $r(T)=0$, the series
$\sum_{k=0}^\infty \left\|\sum\limits_{\sigma\in \FF_n^+, 
  |\sigma|=k}
 T_\sigma T_\sigma^*\right\|$ is convergent.
Let $\mu_\alpha=\rho\in (0,1)$ for any $\alpha\in \FF_n^+$ and let  $W=(W_1,\ldots, W_n)$  be the associated    weighted left multi-shift.  Taking into account that $\lim_{k\to \infty} \left(\frac{1}{\rho^{2k}} \left\|\sum\limits_{\sigma\in \FF_n^+, 
  |\sigma|=k}
 T_\sigma T_\sigma^*\right\|\right)^{1/k}=0$, we deduce that
 $$
 \left\|\sum_{k=0}^\infty\sum_{\alpha\in \FF_n^+, |\alpha|=k} \frac{1}{\boldsymbol\mu (\alpha,g_0)^2} T_\alpha  T_\alpha^*\right\|\leq 
 \sum_{k=0}^\infty \frac{1}{\rho^{2k}} \left\|\sum_{\alpha\in \FF_n^+, |\alpha|=k} T_\alpha  T_\alpha^*\right\|<\infty.
 $$
  Applying Corollary \ref{T1} and taking into account that   $\|W\|=\rho$, the result follows.
\end{proof}

 \bigskip
 
 \section{Multivariable   analogue of Foia\c s--Pearcy model for quasinilpotent operators}
 
 In this section we establish the existence   of a model (up to a similarity)  for every quasi-nilpotent $n$-tuple of operators on  a   Hilbert space. This is a multivariable noncommutative extension of  the Foia\c s--Pearcy model for quasinilpotent operators.

 \begin{theorem}\label{Foias-Pearcy}
 Let  $T=(T_1,\ldots, T_n)\in B(\cH)^n$ be  a quasi-nilpotent $n$-tuple of operators. Then there exists     a   quasi-nilpotent weighted multi-shift $W=(W_1,\ldots, W_n)$ of compact operators  and a joint invariant subspace $\cM\subset F^2(H_n)\otimes \cH$ under the operators $W_i^*\otimes I_\cH$,  $i\in \{1,\ldots, n\}$, such that    $(T_1,\ldots, T_n)$ is jointly similar to  
$$(P_\cM(W_1\otimes I_\cH)|_\cM,\ldots, P_\cM(W_n\otimes I_\cH)|_\cM).$$
 \end{theorem}

 \begin{proof}  First, we assume that $T$ is not a nilpotent $n$-tuple of operators  and consider  the completely positive map 
 $
 \varphi_T:B(\cH)\to B(\cH)$ defined by 
 $$
 \varphi_T(X):=T_1XT_n^*+\cdots + T_nXT_n^*,\qquad X\in B(\cH).
 $$
Seting   
 $a_k:=\|\varphi_T^k(I)\|^{1/4}$ if $k\in \NN$, we   note that
 \begin{equation}\label{ine-a}
 a_{k+m}\leq a_k a_m,\qquad k,m\in \NN,
 \end{equation}
 and 
 $$\lim_{k\to \infty} a_k^{1/k}= \lim_{k\to \infty}  \left\|\sum_{\alpha\in \FF_n^+, |\alpha|=k} T_\alpha  T_\alpha^*\right\|^{1/4k}=r(T)^{1/2}=0.
 $$
 Define the weight sequence 
 $\boldsymbol\mu=\{\mu_\beta\}_{|\beta|\geq 1}$ by  setting
 
  $$\mu_\beta:=\frac{a_{|\beta|}}{a_{|\beta|-1}},\qquad |\beta|\geq 1,
  $$
  where $a_{0}:=1$. The sequence $\{\mu_\beta\}_{|\beta|\geq 1}$ is bounded since  $\mu_\beta\leq \frac{a_{|\beta|-1} a_1}{a_{|\beta|-1}}=a_1$.  Let   $V=(V_1,\ldots, V_n)$  be the associated    weighted left multi-shift with $\{\mu_\beta\}_{|\beta|\geq 1}$.
  Due to the definition of the weights  $\{\mu_\beta\}_{|\beta|\geq 1}$, for any $\beta=g_{i_1}\cdots g_{i_k}\in \FF_n^+$,  we have 
\begin{equation}
\label{ab}
\boldsymbol\mu(\beta, g_0)=  \mu_{g_{i_1}\cdots g_{i_k}} \mu_{g_{i_2}\cdots g_{i_k}}\cdots  \mu_{ g_{i_k}}=a_{|\beta|}.
 \end{equation}
 Consequently,  taking into account that  that $r(T)=\lim_{k\to \infty} \left\|\sum_{\alpha\in \FF_n^+, |\alpha|=k} T_\alpha  T_\alpha^*\right\|^{1/2k} <1$, we deduce that 
  $$
\sum_{k=0}^\infty \left\|\sum_{\alpha\in \FF_n^+, |\alpha|=k} \frac{1}{\boldsymbol\mu (\alpha,g_0)^2} T_\alpha  T_\alpha^*\right\|\leq 
 \sum_{k=0}^\infty  \left\|\sum_{\alpha\in \FF_n^+, |\alpha|=k} T_\alpha  T_\alpha^*\right\|^{1/2}<\infty.
 $$
According to the proof of Theorem \ref{Rota} when $Q=I$,  we have 
\begin{equation}
\label{inv1}
XT_i^*= (V_i^*\otimes I)|_\cN X,\qquad i\in \{1,\ldots, n\},
\end{equation}
where the  operator $X:\cH\to \cN:=K_{\boldsymbol\mu} \cH$ is defined by $Xh:=K_{\boldsymbol\mu} h$ for  $h\in \cH$   and 
 $K_{\boldsymbol\mu}:\cH\to F^2(H_n)\otimes \cH$ is given  by
 $$
 K_{\boldsymbol\mu}h:= \sum_{\alpha\in \FF_n^+} \frac{1}{\boldsymbol\mu (\alpha,g_0)} e_\alpha \otimes  T_\alpha^* h,\qquad h\in \cH. 
 $$
 
 Every $k\in \NN$ has a unique representation $k=2^{m_k}+q_k$, where $m_k,q_k\in \NN_0$ and $0\leq q_k<2^{m_k}$. Note that $m_k$ is an  increasing sequence and   $m_k\to \infty$.
Now, we define the weight sequence $\boldsymbol\kappa=\{\kappa_\beta\}_{\beta\in \FF_n^+}$ by setting $\kappa_{g_0}:=1$ and 
 $$
 \kappa_\beta:=  a_1^{1/2}  \left(a_{2^{m_{|\beta|}}} \right)^{\frac{1}{ 2^{m_{|\beta|}+1}}} ,\qquad |\beta|\geq 1.$$
 Since  $\lim_{k\to \infty} a_k^{1/k}=0$, we have  $ \kappa_\beta\to 0$ as $|\beta|\to \infty$.
 Let  $W=(W_1,\ldots, W_n)$ be the weighted  multi-shift associated with the weight sequence $\boldsymbol\kappa$. Due to Proposition \ref{proprieties}, $r(W)=0$ and each operator $W_i$ is compact.
 In what follows, we prove that
 \begin{equation}\label{ka-ine}
  \kappa_\beta\geq  (a_{|\beta|})^{1/|\beta|}, \qquad |\beta|\geq 1.
 \end{equation}
 Indeed,  assume that $|\beta|=k=2^{m_k}+q_k$. Since $a_{2^{m_k}}\leq a_1^{2^{m_k}}$  and $\frac{2^{m_k}}{k}\geq \frac{1}{2}$, we deduce that
 \begin{equation*}
 \begin{split}
 \kappa_\beta&=a_1\left( \frac{a_{2^{m_k}}^{\frac{1}{2^{m_k}}}}{a_1}  \right)^{1/2}  
 \geq a_1\left( \frac{a_{2^{m_k}}^{\frac{1}{2^{m_k}}}}{a_1}  \right)^{\frac{2^{m_k}}{k}} \\
 &=a_{2^{m_k}}^{\frac{1}{ k}} a_1^{\frac{q_k}{k}}=\left(a_{2^{m_k}} a_1^{q_k}\right)^{1/k}
 \geq \left(a_{2^{m_k}} a_{q_k}\right)^{1/k}\\
 &\geq \left(a_{2^{m_k}+q_k}\right)^{1/k}=a_k^{1/k}. 
 \end{split}
 \end{equation*}
 The last two inequalities are due to inequality \eqref{ine-a}. 
 Therefore, inequality \eqref{ka-ine} holds.
 Now,  let
 \begin{equation}\label{be}
\sigma_\beta:=\frac{\mu_{g_{i_1}\cdots g_{i_k}} \mu_{g_{i_2}\cdots g_{i_k}} \cdots \mu_{g_{i_k}}}{\kappa_{g_{i_1}\cdots g_{i_k}} \kappa_{g_{i_2}\cdots g_{i_k}} \cdots \kappa_{g_{i_k}}},\qquad \text{ if } \beta=g_{i_1}\cdots g_{i_k},
 \end{equation}
 and $\sigma_{g_0}:=1$.
 We note that $\kappa_\alpha\leq \kappa_\beta$ if $|\beta|\leq  |\alpha|$. Indeed,  taking into account that  
 $a_{2^{p+1}}\leq a_{2^p}^2$, we deduce that
  $\left(a_{2^{p+1}}\right)^{\frac{1}{2^{p+1}}}\leq \left(a_{2^{p}}\right)^{\frac{1}{2^{p}}}$, which proves our assertion.  Using relation \eqref{ab}, we obtain 
  $$
 0<\sigma_\beta:= \frac{\mu_{g_{i_1}\cdots g_{i_k}} \mu_{g_{i_2}\cdots g_{i_k}} \cdots \mu_{g_{i_k}}}{\kappa_{g_{i_1}\cdots g_{i_k}} \kappa_{g_{i_2}\cdots g_{i_k}} \cdots \kappa_{g_{i_k}}} \leq \frac{a_{|\beta|}}{\kappa_\beta^k}\leq 1,\qquad \text{ if } \beta=g_{i_1}\cdots g_{i_k}. 
    $$
 Now, we define the   operator $Y\in B( F^2(H_n)\otimes \cH)$ by setting
$$
Y(e_\alpha\otimes h):=\sigma_\alpha e_\alpha\otimes h,\qquad \alpha\in \FF_n^+, h\in \cH.
$$ 
 Note that $Y$ is a positive contraction  and $\ker Y=\ker Y^*=\{0\}$.
 In what follows, we prove that
 $Y(V_i^*\otimes I)=(W_i^*\otimes I)Y$ for any $i\in \{1,\ldots, n\}$.
 Indeed, for any $\alpha\in \FF_n^+$ and $i\in \{1,\ldots, n\}$, we have
 \begin{equation*}
 \begin{split}
 YV_i^* e_\alpha& =\begin{cases}
 Y(\mu_\alpha e_\gamma) & \text{ if } \alpha=g_i\gamma\\
 0&  \text{ otherwise }
 \end{cases}\\
 &=\begin{cases}
  \mu_\alpha \sigma_\gamma e_\gamma & \text{ if } \alpha=g_i\gamma\\
 0&  \text{ otherwise}
\end{cases}
 \end{split}
 \end{equation*}
 and 
  \begin{equation*}
 \begin{split}
 W_iY^* e_\alpha = W_i^*(\sigma_\alpha e_\alpha)=\begin{cases}
  \sigma_\alpha \kappa_\alpha e_\gamma & \text{ if } \alpha=g_i\gamma\\
 0&  \text{ otherwise}.
 \end{cases}
  \end{split}
 \end{equation*}
 According to relation \eqref{be}, if $\alpha=g_i\gamma$, then 
 $\sigma_{g_i\gamma}=\sigma_\gamma \frac{\mu_{g_i\gamma}}{\mu_{g_i\gamma}}$.
 Consequently,  $Y(V_i^*\otimes I)=(W_i^*\otimes I)Y$ for any $i\in \{1,\ldots, n\}$.
 Combining this  with relation \eqref {inv1}, we obtain
 $$
 YXT_i^*=Y(V_i^*\otimes I)X=(W_i^*\otimes I)YX,\qquad i\in \{1,\ldots, n\}.
 $$
 On the other hand, the operator $YX:\cH\to F^2(H_n)\otimes \cH$  is bounded below by $1$. Indeed,  taking into account that  $\sigma_{g_0}=a_0=1$, we have
 $$
\| YXh\|=\left\| Y\left(\sum_{\alpha\in \FF_n^+} \frac{1}{a_{|\alpha|}} e_\alpha \otimes  T_\alpha^* h
\right) \right\|=\left\|  \sum_{\alpha\in \FF_n^+} \frac{1}{a_{|\alpha|}} \sigma_\alpha e_\alpha \otimes  T_\alpha^* h
  \right\|\geq \|h\|
 $$
 for any $h\in \cH$. Consequently, the operator $YX$ has closed range. Now, we can  define the operator
 $S:\cH\to \cM:=YX (\cH)$ by $Sh:=YXh$, $h\in \cH$, and note that
 $ T_i^*=S^{-1}(W_i^*\otimes I)|_\cM S$ for any $i\in \{1,\ldots, n\}$.
 
 Since the case when  $T$ is  a nilpotent $n$-tuple of operators of index $m$  was already discussed  in Theorem \ref{main1}, the proof is complete.
 \end{proof}
 
 In what follows, we assume that the Hilbert space $\cH$ is separable and infinite-dimensional.
 
  \begin{corollary}\label{Foias-Pearcy2}
 Let  $T=(T_1,\ldots, T_n)\in B(\cH)^n$ be  a quasi-nilpotent $n$-tuple of operators. Then there exists     a   quasi-nilpotent weighted multi-shift $K=(K_1,\ldots, K_n)$ of compact operators  in $B(\cH)$ and a joint invariant subspace $\cR\subset  \cH\otimes F^2(H_n)$ under the operators $K_i^*\otimes I_{F^2(H_n)}$,  $i\in \{1,\ldots, n\}$, such that    $(T_1,\ldots, T_n)$ is jointly similar to  
$$(P_\cR(K_1\otimes I_{F^2(H_n)})|_\cR,\ldots, P_\cR(K_n\otimes I_{F^2(H_n)})|_\cR).$$
 \end{corollary}
 \begin{proof}
 Since $\cH$ is a separable infinite dimensional Hilbert space, we can consider an 
 orthonormal basis $\{v_\beta\}_{\beta\in \FF_n^+}$ and define  the weighted left multi-shift $K=(K_1,\ldots, K_n)$  by setting
 $$
 K_iv_\alpha:= \kappa_{g_i \alpha} v_{g_i\alpha},\qquad \alpha \in \FF_n^+, i\in \{1,\ldots, n\},
 $$
 where the weight sequence  $\boldsymbol\kappa=\{\kappa_\beta\}_{\beta\in \FF_n^+}$ was introduced in the proof of Theorem \ref{Foias-Pearcy}.
 Under the identification of $\cH$ with $F^2(H_n)$ via the mapping $v_\alpha\mapsto e_\alpha$, it is clear that $K=(K_1,\ldots, K_n)$ is unitarily equivalent to the weighted multi-shift $W=(W_1,\ldots, W_n)$ which was defined in Theorem \ref{Foias-Pearcy}. Consequently, each $K_i$ is a compact operator and $K$  is a quasi-nilpotent $n$-tuple.
 Define the unitary operator $U: F^2(H_n)\otimes \cH\to \cH\otimes F^2(H_n)$  by setting
 $U(e_\alpha\otimes v_\beta):= v_\alpha\otimes e_\beta$, $\alpha,\beta\in \FF_n^+$. It is straightforward to check that
 $U(W_i^*\otimes \cH)=(K_i^*\otimes I_{F^2(H_n)})U$ and, using the proof of Theorem  \ref{Foias-Pearcy}, that
 $$
 UYXT_i^*=(K_i^*\otimes I_{F^2(H_n)})UYX,\qquad i\in \{1,\ldots, n\}.
 $$
 Since the operator $UYX$ is bounded below from $0$, it has closed range $\cR:=UYX \cH\subset \cH\otimes F^2(H_n)$.  Define $S:\cH\to \cR$ by setting $Sh:=UYXh$, $h\in \cH$. The relations above show that  $(K_i^*\otimes I_{F^2(H_n)})\cR\subset \cR$ and 
 \begin{equation}\label{STS}
 ST_i^* S^{-1} =(K_i^*\otimes I_{F^2(H_n)})|_\cR, \qquad i\in \{1,\ldots, n\}.
 \end{equation}
 This completes the proof.
 \end{proof}

 \begin{theorem}  \label{quasi} Let  $T=(T_1,\ldots, T_n)\in B(\cH)^n$ be  a quasi-nilpotent $n$-tuple of operators. The $n$-tuple $(K_1\otimes  I_{F^2(H_n)} ,\ldots, K_n\otimes  I_{F^2(H_n)})$ of Corollary \ref{Foias-Pearcy2} is jointly quasi-similar to an $n$-tuple $(L_1,\ldots, L_n)$, where each 
 $L_i\in B(\cH\otimes F^2(H_n))$  is a compact operator, i.e there exist quasi-affinities $X$ and $Y$ in  $B(\cH\otimes F^2(H_n))$ such that
 \begin{equation*}
 \begin{split}
 (K_i^*\otimes  I_{F^2(H_n)})X&=XL_i^*\\
 Y(K_i^*\otimes  I_{F^2(H_n)})&=L_i^*Y
 \end{split}
 \end{equation*}
 for any $i\in \{1,\ldots, n\}$.
 \end{theorem} 
 
  \begin{proof}
   Since  $K=(K_1,\ldots, K_n)$ is an   $n$-tuple of operators with $r(K)=0$,  we can apply Proposition 3.5 from  \cite{Po-models} (which is a noncommutative multivariable version of Rota's result \cite{R}) and   find, for each $\alpha\in \FF_n^+$, an invertible operator $Q_{(\alpha)}\in B(\cH)$ such that 
  $$
  \|[L^{(\alpha)}_1,\ldots, L_n^{(\alpha)}]\|\leq \frac{1}{|\alpha|},\qquad \alpha\in \FF_n^+,
  $$ 
  where  $L^{(\alpha)}_i:=Q_{(\alpha)} K_i^* Q_{(\alpha)}^{-1}$. Consequently, since  each   operator  $L_i^{(\alpha)}$ is compact, so is the operator defined by
  $$
  L_i:=\bigoplus_{\alpha\in \FF_n^+}  L_i^{(\alpha)},\qquad i\in \{1,\ldots, n\}.
  $$
  Choose some sequences of strictly positive numbers $\{c_\alpha\}_{\alpha\in \FF_n^+}$ and  $\{d_\alpha\}_{\alpha\in \FF_n^+}$ such that
  $$
  \sup_{\alpha\in \FF_n^+} \|c_\alpha Q_{(\alpha)}^{-1}\|<\infty \quad \text{and}\quad 
   \sup_{\alpha\in \FF_n^+} \|d_\alpha Q_{(\alpha)}\|<\infty
  $$
  and define the operators 
  $$
  X:=\bigoplus_{\alpha\in \FF_n^+} c_\alpha Q_{(\alpha)}^{-1}
  \quad \text{and}\quad 
 Y:=\bigoplus_{\alpha\in \FF_n^+} d_\alpha Q_{(\alpha)}.
  $$
  Note that $X$ and $Y$ are quasi-affinities and
  $$
  d_\alpha Q_{(\alpha)}K_i^*=d_\alpha (L_i^{(\alpha)})^* Q_\alpha,\qquad 
  c_\alpha  Q_{(\alpha)}^{-1}(L_i^{(\alpha)})^*=c_\alpha K_i^* Q_{(\alpha)}^{-1}
 $$
 for any $i\in \{1,\ldots, n\}$ and any $\alpha\in \FF_n^+$.
 Now, note that, under the identification of $\cH\otimes F^2(H_n)$ with $\bigoplus_{\alpha\in \FF_n^+} \cH$, we have 
 \begin{equation*}
 \begin{split}
(K_i^*\otimes I_{F^2(H_n)})X&= \left(\bigoplus_{\alpha\in \FF_n} K_i^*\right)X
 =\bigoplus_{\alpha\in \FF_n}   c_\alpha K_i^* Q_{(\alpha)}^{-1}\\
 &=\bigoplus_{\alpha\in \FF_n} c_\alpha  Q_{(\alpha)}^{-1}(L_i^{(\alpha)})^*
 =X\left(\bigoplus_{\alpha\in \FF_n} (L_i^{(\alpha)})^*\right)
 =XL_i^*
 \end{split}
 \end{equation*}
 and, similarly,  one can  show that $Y(K_i^*\otimes I_{F^2(H_n)})=L^*_i Y$ for any $i\in \{1,\ldots, n\}$. 
 The proof is complete.
   \end{proof}
   
  Due to Corollary  \ref{Foias-Pearcy2} and Theorem \ref{quasi}, one can   deduce the following result.
  
  \begin{corollary}  \label{KK} Let  $T=(T_1,\ldots, T_n)\in B(\cH)^n$ be  a nonzero quasi-nilpotent $n$-tuple of operators. Then there is a nonzero $n$-tuple $K=(K_1,\ldots, K_n)\in B(\cL)$ of compact operators  such that $K$ is a quasi-affine transform of $T$, i.e. there is a quasi-affinity $\Omega:\cL\to \cH$ such that $\Omega K_i=T_i\Omega$ for any $i\in \{1,\ldots, n\}$.
  \end{corollary}
  \begin{proof} According to Theorem \ref{quasi}, there is a quasi-affinity  $Y$ in  $B(\cH\otimes F^2(H_n))$ such that
 \begin{equation*}
 Y(K_i^*\otimes  I_{F^2(H_n)})=L_i^*Y, \qquad  i\in \{1,\ldots, n\}.
 \end{equation*}
 Consequently, using relation \eqref{STS} and setting  $\cL:=\overline{Y\cR}$, we have $L_i^*\cL\subset \cL$ for any $i\in \{1,\ldots, n\}$. 
 It is clear that  $K_i^*:=L_i^*|_\cL:\cL\to \cL$ is a compact operator. Let $\Gamma:\cR\to \cL$ be the quasi-affinity defined by $\Gamma x:=Yx$ for any $x\in \cR$ and note that the relation above implies
   \begin{equation*}
\Gamma (K_i^*\otimes  I_{F^2(H_n)})|_\cR=L_i^*Y|_\cR=K_i^* \Gamma , \qquad  i\in \{1,\ldots, n\}.
 \end{equation*}
  Due to relation \eqref{STS}, we have
  \begin{equation*} 
 ST_i^* S^{-1} =(K_i^*\otimes I_{F^2(H_n)})|_\cR, \qquad i\in \{1,\ldots, n\}.
 \end{equation*}
Combining these relations,  and setting $\Omega:=(\Gamma S)^*$ , we conclude that    $T_i\Omega=\Omega K_i$ for any $i\in \{1,\ldots, n\}$, which proves that $K$ is a  quasi-affine transform of $T$.
  This completes the proof.
   \end{proof}
   
   We    remark that if  $(T_1^*,\ldots T_n^*)$ is quasi-nilpotent $n$-tuple, then  applying the later corollary we find   an   $n$-tuple $ (M_1,\ldots, M_n)$ of compact operators  such that $(T_1,\ldots, T_n)$ is a quasi-affine transform of $ (M_1,\ldots, M_n)$.
  We note that if $(T_1,\ldots, T_n)$ is a nilpotent $n$-tuple of operators, then so is $(T^*_1,\ldots, T^*_n)$.
  
  On the other hand,  we mention that, due to Corollary \ref{KK},  if  $\cM\subset  \cL$ is a joint invariant  subspace under $K_1,\ldots, K_n$, then 
  $\overline{\Omega\cM}$ is jointly invariant under $T_1,\ldots, T_n$. In particular, if $\cap_{i=1}^n \ker K_i\neq \{0\}$, then $\cap_{i=1}^n \Omega (\ker K_i)\neq \{0\}$  is a jointly invariant subspace  under $T_1,\ldots, T_n$.

 \bigskip

 \section{Noncommutative Hardy spaces associated with weighted shifts}

Let  $\boldsymbol\mu=\{\mu_\beta\}_{\beta\in \FF_n^+}$ be a sequence of strictly positive real numbers  with $\mu_{g_0}=1$
 and let $F^2(\boldsymbol\mu)$ be the Hilbert space of formal power series in indeterminates $Z_1,\ldots, Z_n$ with orthogonal basis $\{Z_\alpha:\ \alpha\in \FF_n^+\}$ such that $\|Z_\alpha\|_{\boldsymbol\mu}:=\boldsymbol \mu(\alpha, g_0)$. Note that
$$
F^2(\boldsymbol\mu)=\left\{ \zeta=\sum_{\alpha\in \FF_n^+}c_\alpha Z_\alpha: \   \|\zeta\|_{\boldsymbol\mu}^2:=\sum_{\alpha\in \FF_n^+}\boldsymbol \mu(\alpha, g_0)^2 |c_\alpha|^2<\infty, \ c_\alpha\in \CC\right\}
$$
and $\left\{\frac{1}{\boldsymbol \mu(\alpha, g_0)}Z_\alpha\right\}_{\alpha\in \FF_n^+}$ is an orthonormal basis for $F^2(\boldsymbol\mu)$.
The {\it left multiplication} operator  $L_{Z_i}$ on $F^2(\boldsymbol\mu)$ is  defined by  $L_{Z_i}\zeta:=Z_i\zeta$ for all $\zeta\in F^2(\boldsymbol\mu)$.  
One can easily see that $L_{Z_i}$ is a bounded operator  if and only if 
\begin{equation}  \label{bound1}
\sup_{\alpha\in \FF_n^+} \mu_{g_i\alpha}<\infty.
\end{equation}
Similarly, we define the {\it right multiplication} operator  $R_{Z_i}:F^2(\boldsymbol\mu)\to F^2(\boldsymbol\mu)$  by   setting $R_{Z_i}\zeta:=\zeta Z_i$ for all $\zeta\in F^2(\boldsymbol\mu)$.   Note that $R_{Z_i}$ is a  bounded operator  if and only if 
\begin{equation} \label{bound2}
\sup_{\alpha\in \FF_n^+} \frac{\boldsymbol \mu(\alpha g_i, g_0)}{\boldsymbol \mu(\alpha, g_0)}<\infty.
\end{equation}
 It is easy to see   that $L_{Z_i} R_{Z_j}=  R_{Z_j}L_{Z_i}$ for any $i,j\in \{1,\ldots, n\}$.
 
 We remark that  the map
 $ \mu_\alpha  \mapsto b_\alpha:=\frac{1}{\boldsymbol \mu(\alpha, g_0)^2}$  is a 
 one-to-one correspondence  between  the set
  \begin{equation*} 
{\bf M}:= \{\{\mu_\alpha\}_{\alpha\in \FF_n^+}: \  \mu_\alpha>0, \mu_{g_0}=1, \text{ and }  \sup_{\alpha\in \FF_n^+} \mu_\alpha<\infty\}
 \end{equation*}
 and the set
  \begin{equation*} 
{\bf B}:=\{\{b_\alpha\}_{\alpha\in \FF_n^+}: \  b_\alpha>0, b_{g_0}=1, \text{ and }  \sup_{\alpha\in \FF_n^+} \frac{b_\alpha}{b_{g_i\alpha}}<\infty\}.
 \end{equation*}
Moreover, we have $\mu_{g_i\alpha}=\sqrt{\frac{b_\alpha}{b_{g_i\alpha}}}$ for all $i\in \{1,\ldots, n\}$. Due to this correspondence,  all the results of the present  paper   apply in the particular setting   of the regular domains \cite{Po-domains}, the noncommutative weighted Bergman spaces  \cite{Po-Berezin}, \cite{Po-Bergman}, and the more recent extension to admissible domains \cite{Po-NC}.

\begin{example} \label{Ex1}
For each $s\in (0,\infty)$, consider the weights $\boldsymbol\mu^s:=\{\mu^s_\alpha\}_{\alpha\in \FF_n^+}$ defined by $\mu^s_{g_0}:=1$ and 
$$
\mu^s_\alpha:=\sqrt{\frac{|\alpha|}{s+|\alpha|-1}},\qquad  |\alpha|\geq 1.
$$
The associated Hilbert space is 
$$
F^2(\boldsymbol\mu^s)=\left\{ \zeta=\sum_{\alpha\in \FF_n^+}c_\alpha Z_\alpha: \   \|\zeta\|_{\boldsymbol\mu^s}^2:=\sum_{\alpha\in \FF_n^+} \frac{1}{\left(\begin{matrix} s+k-1 \\k \end{matrix}\right)} |c_\alpha|^2<\infty, \ c_\alpha\in \CC\right\}.
$$
If $s=1$, then the corresponding subspace $F^2(\boldsymbol \mu^1)$ is   the full Fock space on $n$ generators.
On the other hand, if $s=n$ (resp. $s=n+1$) we obtain the noncommutative Hardy (resp. Bergman)  space over the unit ball
 $[B(\cH)^n]_1$. In the particular case when $s\in \NN$, we obtain the noncommutative weighted Bergman space over $[B(\cH)^n]_1$, which was extensively studied in  \cite{Po-Berezin}, \cite{Po-domains}, and \cite{Po-Bergman}. One can easily check that $[W_1,\ldots, W_n]$ is a row contraction if $s\geq 1$.
     We remark that when $s\in (0,1]$,   the Hilbert  spaces $F^2(\boldsymbol \mu^s)$  are  noncommutative generalizations of the classical   Besov-Sobolev spaces on the unit ball $\BB_n\subset \CC^n$ with representing kernel $\frac{1}{(1-\left<z,w\right>)^s}$, $z,w\in \BB_n$.
\end{example}

\begin{example} \label{Ex2} Let $s\in \RR$ and consider the weights
$\boldsymbol\omega^s:=\{\omega^s_\alpha\}_{\alpha\in \FF_n^+}$ defined by $w^s_{g_0}:=1$ and 
$$
\omega^s_\alpha:=\sqrt{\left(\frac{|\alpha|}{|\alpha|+1}\right)^s},\qquad  |\alpha|\geq 1.
$$
The associated Hilbert space is 
$$
F^2(\boldsymbol\omega^s)=\left\{ \zeta=\sum_{\alpha\in \FF_n^+}c_\alpha Z_\alpha: \   \|\zeta\|_{\boldsymbol\omega^s}^2:=\sum_{\alpha\in \FF_n^+} \frac{1}{ (|\alpha|+1)^s} |c_\alpha|^2<\infty, \ c_\alpha\in \CC\right\}.
$$
We remark that the scale of spaces $\{F^2(\boldsymbol\omega^s):  s\in \RR\}$ contains  the noncommutative Dirichlet space  $F^2(\boldsymbol\omega^{-1})$over the unit ball
 $[B(\cH)^n]_1$. Note that  $[W_1,\ldots, W_n]$ is a row contraction if $s\geq 0$. When $s=0$, we recover again the full Fock space with $n$ generators.
\end{example}

\begin{example} \label{Ex3}
Let  $s\in [1,\infty)$ and let $\varphi=\sum_{|\alpha|\geq 1} d_\alpha Z_\alpha$ be a formal power series such that $d_\alpha \geq 0$ and $d_{g_i}>0$ for $i\in \{1,\ldots, n\}$. 
Consider the formal power series
\begin{equation*}
\begin{split}
g&= 1+\sum_{k=1}^\infty\left(\begin{matrix} s+k-1 \\k \end{matrix}\right) \varphi^k
=1+\sum_{\alpha\in \FF_n^+, |\alpha|\geq 1}\left(\sum_{j=1}^{|\alpha|} \left(\begin{matrix} s+j-1 \\j \end{matrix}\right)
\sum_{{\gamma_1\cdots \gamma_j=\alpha }\atop {|\gamma_1|\geq 1,\ldots,
 |\gamma_j|\geq 1}} d_{\gamma_1}\cdots d_{\gamma_j} \right)  
  Z_\alpha
\end{split}
\end{equation*}
and let $b_\alpha $ be the coefficient  of $Z_\alpha$.
Consider  the weights  $\boldsymbol \mu^\varphi=\{\mu^\varphi_\alpha\}_{\alpha\in \FF_n^+}$ defined  by $\mu^\varphi_{g_0}:=1$ and $\mu^\varphi_{g_i\alpha}:=\sqrt{\frac {b_\alpha}{b_{g_i\alpha}}}$. In this case, the associated Hilbert space is 
$$
F^2(\boldsymbol \mu^\varphi)=\left\{ \zeta=\sum_{\alpha\in \FF_n^+}c_\alpha Z_\alpha: \   \|\zeta\|_{\boldsymbol\omega^s}^2:=\sum_{\alpha\in \FF_n^+} \frac{1}{ b_\alpha} |c_\alpha|^2<\infty, \ c_\alpha\in \CC\right\}.
$$
Note that when $\varphi=Z_1+\cdots +Z_n$ we recover the weights considered in Example \ref{Ex1} when $s\geq1$. We remark that if $d_{g_i}\geq 1$ for all $i\in \{1,\ldots, n\}$, then $[W_1,\ldots, W_n]$ is a row contraction.  This is due to the fact that
$\mu^\varphi_{g_i\alpha}:=\sqrt{\frac {b_\alpha}{b_{g_i\alpha}}}\leq \frac{1}{\sqrt{d_{g_i}}}$ (see  \cite{Po-NC}).
\end{example}

We mention that for all the examples presented above the conditions  \eqref{bound1} and \eqref{bound2} are satisfied.
 We remark that  the operator $U:F^2(H_n)\to F^2(\boldsymbol\mu)$,  defined by $U(e_\alpha):=\frac{1}{\boldsymbol \mu(\alpha, g_0)}Z_\alpha$, $\alpha\in \FF_n^+$, is unitary and
 $UW_i=L_{Z_i}U, \  i\in \{1,\ldots,n\},
 $
 where $W=(W_1,\ldots, W_n)$ is the weighted left multi-shift associated with the weights 
 $\boldsymbol\mu=\{\mu_\beta\}_{\beta\in \FF_n^+ }$.  
We also have $U\Lambda_i=R_{Z_i}U$ for every $ i\in \{1,\ldots,n\}$, where
 ${\Lambda}:=(\Lambda_1,\ldots, \Lambda_n)$  is the weighted right multi-shift associated with the weights $\widetilde{\boldsymbol\mu}=\{\widetilde\mu_\beta\}_{\beta\in \FF_n^+ }$  defined  on the full Fock space $F^2(H_n)$  by  $\Lambda_i e_\alpha:=
\widetilde\mu_{\alpha g_i }
e_{\alpha g_i}$,  where 
 $$
 \widetilde\mu_{\alpha g_i}:=\frac{\boldsymbol \mu(\alpha g_i, g_0)}{\boldsymbol \mu(\alpha, g_0)},\qquad i\in \{1,\ldots, n\}, \alpha\in \FF_n^+,
 $$
 and $\widetilde\mu_{g_0}:=1$.   In view of the correspondence $ \mu_\alpha  \mapsto b_\alpha$ mentioned above,  we remark that $\widetilde\mu_{\alpha g_i}=\sqrt{\frac{b_\alpha}{b_{\alpha g_i}}}$ for all $i\in \{1,\ldots, n\}$ and 
  $b_\alpha=\frac{1}{\widetilde{\boldsymbol \mu}(\alpha, g_0)^2}$ for any $\alpha\in \FF_n^+$, where
  $$\widetilde{\boldsymbol\mu} (\beta,g_0)
  :=\begin{cases}\widetilde{\mu}_{g_{i_1}\cdots g_{i_p}} \widetilde{\mu}_{g_{i_1}\cdots g_{i_{p-1}}} \cdots \widetilde{\mu}_{g_{i_1}} & \text{if}\  \beta=g_{i_1}\cdots g_{i_p}\in \FF_n^+\\
 1&\text{if} \ \beta=g_0.
 \end{cases}
$$

The considerations above will allow us to represent  injective weighted left multi-shifts 
$W_1,\ldots, W_n$ as ordinary multiplications by  $Z_1,\ldots, Z_n$ on a Hilbert space of formal power series, which leads naturally to analytic function theory in several complex variables. We will use these two viewpoints interchangeable  as convenient. One of the goal for the remainder  of the  paper is to analyze the extent to which  our noncommutative formal power series  represent analytic functions in  several noncommutative (resp. commutative)  variables.

 A formal power series  $\varphi=\sum_{\alpha\in \FF_n^+} a_\alpha Z_\alpha $ is called   {\it right multiplier} on  $F^2(\boldsymbol\mu) $   and  we denote $\varphi\in \cM^r(F^2(\boldsymbol\mu)) $ if,  for any
$\zeta=\sum_{\beta\in \FF_n^+} d_\beta Z_\beta $ in $F^2(\boldsymbol\mu)$, we have
$$
\zeta\varphi:=\sum_{\gamma\in \FF_n^+}  \left( \sum_{\alpha,\beta\in\FF_n^+, \beta \alpha =\gamma}  d_{\beta} a_{\alpha}\right)Z_\gamma\in F^2(\boldsymbol\mu).
$$
Due to the closed graph theorem,  the right multiplication operator
$R_\varphi :F^2(\boldsymbol\mu) \to F^2(\boldsymbol\mu) $ defined by $R_\varphi \zeta:=\zeta\varphi$  is a  bounded  linear operator.   The (right) multiplier norm of $\varphi$ is  given by 
$\|\varphi\|^{\cM^r}:=\|R_\varphi\|$.  
    Note that  
$\cM^r(F^2(\boldsymbol\mu))$ becomes a Banach algebra with respect to the multiplier norm.
Similarly, we  introduce the the set of   {\it left multipliers} on $F^2(\boldsymbol\mu) $  and denote  it by $ \cM^\ell(F^2(\boldsymbol\mu))$. In this case, the left multiplication operator
$L_\varphi :F^2(\boldsymbol\mu) \to F^2(\boldsymbol\mu) $   is defined by $L_\varphi \zeta:=\varphi \zeta$   and it is a  bounded  linear operator. The (left) multiplier norm of $\varphi$ is  given by 
$\|\varphi\|^{\cM^\ell}:=\|L_\varphi\|$.

As in the proof  of Theorem 4.3 from \cite{Po-Bergman}, one can show that  a bounded operator $X:F^2(\boldsymbol\mu) \to F^2(\boldsymbol\mu)$ satisfies  the relation
$$
XL_{Z_i}=L_{Z_i} X,\qquad i\in \{1,\ldots, n\},
$$
if and only if   there is $\varphi\in \cM^r(F^2(\boldsymbol\mu)) $ such that
$X=R_\varphi$. Similarly, one can prove that a bounded operator $X:F^2(\boldsymbol\mu) \to F^2(\boldsymbol\mu)$ satisfies  the relation
$$
XR_{Z_i}=R_{Z_i} X,\qquad i\in \{1,\ldots, n\},
$$
if and only if the is  there is $\varphi\in \cM^l(F^2((\boldsymbol\mu)) $ such that
$X=L_\varphi$.
Consequently, we have 
\begin{equation*}
\begin{split}
\{W_1,\ldots, W_n\}' &=R^\infty(\boldsymbol \mu):=\{{U}^*  R_\varphi  U : \  \varphi\in \cM^r (F^2(\boldsymbol\mu))\}\\
\{\Lambda_1,\ldots, \Lambda_n\}' &=F^\infty(\boldsymbol \mu):=\{{U}^*  L_\varphi  U : \  \varphi\in \cM^l (F^2(\boldsymbol\mu))\}
\end{split}
\end{equation*}
where ${}'$ stands for the commutant. 
Note that a power series $\zeta=\sum_{\alpha\in \FF_n^+}c_\alpha Z_\alpha$ is in  $ \cM^l (F^2(\boldsymbol\mu))$  if and only if
$\zeta(W_1,\ldots, W_n)p:=\sum_{\alpha\in \FF_n^+}c_\alpha W_\alpha p\in F^2(H_n)$ for any  polynomial $p\in \cP:=\Span \{e_\alpha:\ \alpha\in \FF_n^+\}$ and 
$$
\|\zeta(W_1,\ldots, W_n)p\|\leq M\|p\|,\qquad p\in\cP,
$$
for some constant $M>0$. We  denote by $\zeta(W_1,\ldots, W_n)$ the unique bounded linear extension of the operator  to $F^2(H_n)$. Therefore $F^\infty(\boldsymbol\mu)=\{\zeta(W_1,\ldots, W_n):\ \zeta\in  \cM^l (F^2(\boldsymbol\mu))\}$ and  $\zeta(W_1,\ldots, W_n) =U^*L_\varphi U$.
In a similar manner, one can see that
 $R^\infty(\boldsymbol\mu)=\{\zeta(\Lambda_1,\ldots, \Lambda_n):\ \zeta\in  \cM^r (F^2(\boldsymbol\mu))\}$ and   $\zeta(\Lambda_1,\ldots, \Lambda_n) =U^*R_\varphi U$.

Similarly to the proof of Theorem 2.5 from \cite{Po-NC}, one can obtain the following result.

\begin{theorem}  \label{closure} Let  $\boldsymbol\mu=\{\mu_\beta\}_{\beta\in \FF_n^+}$ be a sequence of strictly positive real numbers  satisfying relations \eqref{bound1} and \eqref{bound2} and let $W=(W_1,\ldots, W_n)$ be the associated weighted left multi-shift.
Then   the noncommutative  Hardy algebra  $F^\infty(\boldsymbol\mu)$  satisfies the relation
  $$
  F^\infty(\boldsymbol\mu)=\overline{\cP({W})}^{SOT}=\overline{\cP({W})}^{WOT}=\overline{\cP({W})}^{w*},
  $$
where $\cP({W})$ stands for the algebra of all polynomials in $W_1,\ldots, W_n$ and the identity.
Moreover,  $F^\infty(\boldsymbol\mu)$ is the sequential SOT-(resp. WOT-, w*-) closure of   $\cP({W})$. In addition, if  $\zeta(W)\in F^\infty(\boldsymbol\mu)$ has the Fourier representation  $\sum_{\alpha\in \FF_n^+}c_\alpha W_\alpha$, then
$$
\zeta(W)=\text{\rm SOT-}\lim_{N\to\infty}\sum_{ |\alpha|\leq N} \left(1-\frac{|\alpha|}{N+1}\right) c_\alpha W_\alpha, 
$$
$$
\left\|\sum_{\alpha\in \FF_n^+, |\alpha|=k}  c_\alpha W_\alpha \right\|\leq \|\zeta(W)\|,\qquad k\in \NN,
$$
   and
$$
\left\|\sum_{ |\alpha|\leq N} \left(1-\frac{|\alpha|}{N+1}\right) c_\alpha W_\alpha   \right\|\leq \|\zeta(W)\|,\qquad N\in \NN.
$$

\end{theorem}
We remark that a similar result holds for the weighted right multi-shift and the Hardy algebra $R^\infty(\boldsymbol\mu)$.

Let $\lambda=(\lambda_1,\ldots, \lambda_n)\in \CC^n$ and define the linear functional  $\varphi_\lambda:\CC\left<Z_1,\ldots, Z_n\right>\to \CC$ of evaluation  at $\lambda$ by setting
$$\varphi_\lambda(p):=p(\lambda),\qquad p\in \CC\left<Z_1,\ldots, Z_n\right>.$$

\begin{definition}  We say that $\lambda\in \CC^n$ is a  bounded point   evaluation on the Hilbert space $F^2(\boldsymbol\mu)$ if $\varphi_\lambda$ extends to a bounded linear functional   on
$F^2(\boldsymbol\mu)$, i.e. there is $C>0$ such that
$$
|\varphi_\lambda(p)|=|p(\lambda)|\leq C \|p\|_{\boldsymbol\mu}, \qquad p\in \CC\left<Z_1,\ldots, Z_n\right>.$$
Denote 
$$\cD(\boldsymbol\mu):=\{\lambda\in \CC^n:\   \varphi_\lambda \text{ is  a bounded  linear functional on } F^2(\boldsymbol\mu)\}.
$$
\end{definition}

\begin{definition}
The joint point spectrum of    $(L_{Z_1}^*,\ldots, L_{Z_n}^*)$ is the set  
$
\sigma_p(L_{Z_1}^*,\ldots, L_{Z_n}^*)$  of all  tuples $\zeta=(\zeta_1,\ldots, \zeta_n)\in \CC^n$ with the property that  there is  $z\in F^2(\boldsymbol\mu)$, $z\neq 0$ such that  $L_{Z_i}^*z=\zeta_i z$ for any $i\in \{1,\ldots, n\}$. In this case, we say that  $z$ is a   joint eigenvector for the operators    $L_{Z_1}^*,\ldots, L_{Z_n}^*$.
\end{definition}

 Given  a $k$-tuple  $\lambda=(\lambda_1,\ldots, \lambda_n)\in \CC^n$   we  set
$\lambda_{\alpha}:=\lambda_{j_1}\cdots \lambda_{j_m}$ if
$\alpha=g_{j_1}\cdots g_{j_m}\in \FF_{n}^+$ and $\lambda_{g_{0}}:=1$.

\begin{theorem}  \label{evaluation}  Let  $\boldsymbol\mu=\{\mu_\beta\}_{\beta\in \FF_n^+}$ be a sequence of strictly positive real numbers  with $\mu_{g_0}:=1$  and  $$\sup_{\alpha\in \FF_n^+, i\in \{1,\ldots,n\}} \mu_{g_i\alpha}<\infty$$
and  let $F^2(\boldsymbol\mu)$ be the Hilbert space of formal power series in indeterminates $Z_1,\ldots, Z_n$ associated with $\boldsymbol\mu$.  If $L_Z=(L_{Z_1},\ldots, L_{Z_n})$ is the left multi-shift associated with $\boldsymbol\mu$,
then the following statements hold.
\begin{enumerate}
\item[(i)]
 $$\cD(\boldsymbol\mu)=\left\{ (\lambda_1,\ldots, \lambda_n)\in \CC^n : \  \sum_{\alpha\in \FF_n^+} \frac{|\lambda_\alpha|^2}{\boldsymbol \mu(\alpha, g_0)^2} <\infty\right\}
 =\sigma_p(L_{Z_1}^*,\ldots, L_{Z_n}^*).
 $$
 \item[(ii)]  If    $\lambda \in   \cD(\boldsymbol\mu)$ and $f=\sum_{\alpha\in \FF_n^+}c_\alpha Z_\alpha$ is in the space $F^2(\boldsymbol\mu)$, then the series
 $f(\lambda):=\sum_{\alpha\in \FF_n^+}c_\alpha \lambda_\alpha$ is absolutely convergent to $\varphi_\lambda(f)$ and 
 $$
 |f(\lambda)|\leq \|f\|_{\boldsymbol\mu}\left(\sum_{\alpha\in \FF_n^+} \frac{|\lambda_\alpha|^2}{\boldsymbol \mu(\alpha, g_0)^2}  \right)^{1/2}.
 $$
 
 \item[(iii)] If  $\lambda=(\lambda_1,\ldots, \lambda_n)\in \CC^n$ and    the series   
 $\sum_{k=0}^\infty\sum_{\alpha\in \FF_n^+, |\alpha|=k}c_\alpha \lambda_\alpha$ is convergent for any  element $f=\sum_{\alpha\in \FF_n^+}c_\alpha Z_\alpha$ is 
   $F^2(\boldsymbol\mu)$, 
 then $\lambda\in \cD(\boldsymbol\mu)$.
 \item[(iv)]  If $\varphi=\sum_{\alpha\in \FF_n}c_\alpha Z_\alpha$ is a bounded free holomorphic function on the open ball of $B(\cH)^n$ of radius $\|L_Z\|$, then $\varphi\in  \cM^\ell(F^2(\boldsymbol\mu))$ and 
 $$
 \|L_\varphi\|=\|\varphi(L_{Z_1},\ldots, L_{Z_n})\|\leq \sup \{\|\varphi(X_1,\ldots, X_n)\|:\ \|(X_1,\ldots, X_n)\|<\|L_Z\|\}.
 $$
 
  \end{enumerate}
\end{theorem}

\begin{proof} 
 
If $\lambda=(\lambda_1,\ldots, \lambda_n)\in   \cD(\boldsymbol\mu)$, then   $ \varphi_\lambda$  is  a bounded linear functional on   $F^2(\boldsymbol\mu)$ and the  Riesz representation theorem implies the existence  of an element 
$k_\lambda\in F^2(\boldsymbol\mu)$ such that $\varphi_\lambda(f)=\left<f, k_\lambda\right>_{\boldsymbol\mu}$ for any $f\in  F^2(\boldsymbol\mu)$. Since  $\left<Z_\alpha, k_\lambda\right>_{\boldsymbol\mu}=\lambda_\alpha$ for any $\alpha\in \FF_n^+$, we deduce that  
$k_\lambda=\sum_{\alpha\in \FF_n^+} \bar \lambda_\alpha \frac{1}{\boldsymbol \mu(\alpha, g_0)^2} Z_\alpha$.
Since $k_\lambda\in F^2(\boldsymbol\mu)$, we must have $\sum_{\alpha\in \FF_n^+}\frac{ |\lambda_\alpha|^2}{\boldsymbol \mu(\alpha, g_0)^2} <\infty$. The converse is obvious.

To prove the second equality in the part (i),  let   $\lambda\in   \cD(\boldsymbol\mu)$ and note that \begin{equation*}
\begin{split}
\left< f, L_{Z_i}^*k_\lambda\right>_{\boldsymbol\mu}&= \left< L_{Z_i}f, \sum_{\alpha\in \FF_n^+} \bar \lambda_\alpha b_\alpha Z_\alpha\right>_{\boldsymbol\mu}\\
&=\lambda_i\left<  Z_i f,  k_\lambda\right>_{\boldsymbol\mu}
=\left<  Z_i f, \bar\lambda_i,  k_\lambda\right>_{\boldsymbol\mu}
\end{split}
\end{equation*}
for any $f\in F^2(\boldsymbol\mu)$ and $i\in \{1,\ldots, n\}$. Consequently, $L_{Z_i}^* k_\lambda=
\bar\lambda_i,  k_\lambda$  which shows that 
$(\bar\lambda_1,\ldots, \bar\lambda_n)\in \sigma_p(L_{Z_1}^*,\ldots, L_{Z_n}^*)$ and 
$(\lambda_1,\ldots, \lambda_n)\in \sigma_p(L_{Z_1}^*,\ldots, L_{Z_n}^*)$. 

Conversely, assume that $(\bar\lambda_1,\ldots, \bar\lambda_n)\in \sigma_p(L_{Z_1}^*,\ldots, L_{Z_n}^*)$ and let $h\in F^2(\boldsymbol\mu)$, $h\neq 0$, with the property that   $L_{Z_i}^* h=
\bar\lambda_i h$ for any $i\in \{1,\ldots, n\}$. Consider the bounded linear functional
$\psi:F^2(\boldsymbol\mu)\to \CC$ given by $\psi(f):=\left<f, wh\right>_{\boldsymbol\mu}$ with $w\in \CC$, $w\neq 0$, to be determined. Note that
\begin{equation*}
\begin{split}
\psi(Z_if)=\left<Z_i f, wh\right>_{\boldsymbol\mu}=\left<  f, w L_{Z_i}^*h\right>_{\boldsymbol\mu}=\lambda_i\left<f, wh\right>_{\boldsymbol\mu}=\lambda_i\psi(f)
\end{split}
\end{equation*}
for any $i\in \{1,\ldots, n\}$
which implies   $\psi(Z_\alpha)=\lambda_\alpha \psi(1)$ for any $\alpha\in \FF_n^+$ with $|\alpha|\geq 1$. Since $\psi\neq 0$, we must have $\psi(1)\neq 0$. Setting $w:=\frac{1}{\overline{\left<1, h\right>_{\boldsymbol\mu}}}$, we have $\psi(1)=1$ and, consequently,
$\psi(p)=p(\lambda)$ for any $p\in  \CC\left<Z_1,\ldots, Z_n\right>$. Therefore, $(\lambda_1,\ldots, \lambda_n)\in \cD(\boldsymbol\mu)$ which completes the proof of part (i).

To prove part (ii),  note that
  $g:=\sum_{\alpha\in \FF_n^+}|c_\alpha| Z_\alpha$ is also in   $F^2(\boldsymbol\mu)$. If  $\lambda \in   \cD(\boldsymbol\mu)$, then, 
   setting 
 $g_m:=\sum_{\alpha\in \FF_n^+, |\alpha|\leq m} |c_\alpha| Z_\alpha$, we deduce that
 $g_m(\lambda)=\varphi_\lambda(g_m)=\left<g_m, k_\lambda\right>_{\boldsymbol\mu}\to  \varphi_\lambda(g)$ as $m\to \infty$. Consequently, the series  $\sum_{\alpha\in \FF_n^+}|c_\alpha| \lambda_\alpha$ is   convergent for any $\lambda \in   \cD(\boldsymbol\mu)$. Since $(\lambda_1 e^{i\theta_1},\ldots, \lambda_n e^{i\theta_n})\in \cD(\boldsymbol\mu)$ for any $\theta_i\in \RR$, we also deduce that
$\sum_{\alpha\in \FF_n^+}|c_\alpha| |\lambda_\alpha|$ is   convergent for any $\lambda \in   \cD(\boldsymbol\mu)$.  As above, we deduce that, for any $\lambda \in   \cD(\boldsymbol\mu)$,
$$
\sum_{\alpha\in \FF_n^+, |\alpha|\leq m} c_\alpha \lambda_\alpha=\varphi_\lambda\left(\sum_{\alpha\in \FF_n^+, |\alpha|\leq m} c_\alpha \lambda_\alpha\right)\to \varphi_\lambda(f)=\sum_{\alpha\in \FF_n^+}c_\alpha \lambda_\alpha
$$
as $m\to \infty$. Since $f(\lambda)=\varphi_\lambda(f)=\left<f, k_\lambda\right>_{\boldsymbol\mu}$ for any $f\in  F^2(\boldsymbol\mu)$ and $\|k_\lambda\|_{\boldsymbol\mu}=\left(\sum_{\alpha\in \FF_n^+}  \frac{|\lambda_\alpha|^2}{\boldsymbol \mu(\alpha, g_0)^2}\right)^{1/2}$, the Cauchy-Schwartz inequality implies
 $$
 |f(\lambda)|\leq \|f\|_{\boldsymbol\mu}\left(\sum_{\alpha\in \FF_n^+}  \frac{|\lambda_\alpha|^2}{\boldsymbol \mu(\alpha, g_0)^2}\right)^{1/2},\qquad \lambda \in   \cD(\boldsymbol\mu), 
 $$
for any  $f\in F^2(\boldsymbol\mu)$.

Now, we prove part (iii). Assume that $\lambda=(\lambda_1,\ldots, \lambda_n)\in \CC^n$ and    the series   
 $\sum_{k=0}^\infty\sum_{\alpha\in \FF_n^+, |\alpha|=k}c_\alpha \lambda_\alpha$ is convergent for any   $f=\sum_{\alpha\in \FF_n^+}c_\alpha Z_\alpha$ is   $F^2(\boldsymbol\mu)$. 
 Define the linear functional on $F^2(\boldsymbol\mu)$ by setting
 $\varphi_\lambda (f):=\sum_{k=0}^\infty\sum_{\alpha\in \FF_n^+, |\alpha|=k}c_\alpha \lambda_\alpha$. For each $m\in \NN$, set  $k^{(m)}_\lambda:= \sum_{\alpha\in \FF_n^+, |\alpha|\leq m} \bar \lambda_\alpha \frac{1}{\boldsymbol \mu(\alpha, g_0)^2} Z_\alpha$. Note that
 $$
 \left< f, k^{(m)}_\lambda\right>_{\boldsymbol\mu}=\sum_{\alpha\in \FF_n^+, |\alpha|\leq m}  c_\alpha\lambda_\alpha\to \varphi_\lambda(f)
 $$
 as $m\to \infty$. Applying the uniform boundedness principle we deduce that 
$k_\lambda:=\sum_{\alpha\in \FF_n^+} \bar \lambda_\alpha \frac{1}{\boldsymbol \mu(\alpha, g_0)^2} Z_\alpha$ is in the Hilbert $F^2(\boldsymbol\mu)$. Due to part (i), we conclude that $\lambda\in \cD(\boldsymbol\mu)$.

To prove part (iv), note that  $\psi:=\sum_{\alpha\in \FF_n}c_\alpha\|L_Z\|^{|\alpha|} Z_\alpha$
is a bounded free holomorphic function  on the open unit ball  of $B(\cH)^n$ and, consequently, $\Psi(S_1,\ldots, S_n)=\sum_{\alpha\in \FF_n}c_\alpha\|L_Z\|^{|\alpha|} S_\alpha$ is the Fourier representation of a unique element  $\Psi(S_1,\ldots, S_n)$ in the Hardy algebra $F^\infty(H_n)$.
Consider the weighted shift $\Gamma:=(\frac{1}{\|L_Z\|} L_{Z_1},\ldots, \frac{1}{\|L_Z\|} L_{Z_n})$ and note that $\|\Gamma\|=1$. According to Theorem \ref{power bounded}, 
$\Gamma$ is a pure row contraction. Due to the  $F^\infty(H_n)$-functional calculus for pure row contractions \cite{Po-funct}, we deduce that $\|\psi(\Gamma_1,\ldots \Gamma_n)\|
 \leq \|\psi(S_1,\ldots, S_n)\|$. Consequently, we have
\begin{equation*}
\begin{split}
 \|L_\varphi\|&=\|\varphi(L_{Z_1},\ldots, L_{Z_n})=\|\psi(\Gamma_1,\ldots \Gamma_n)\|\\
 &\leq \|\psi(S_1,\ldots, S_n)\|=\sup_{\|X\|<\|L_Z\|} \|\varphi(X_1,\ldots, X_n)\|.
\end{split}
\end{equation*}
The proof is complete.
\end{proof}

Let  $\boldsymbol\mu=\{\mu_\beta\}_{\beta\in \FF_n^+}$ be a sequence of strictly positive real numbers  satisfying relations \eqref{bound1} and \eqref{bound2}  and let 
 $W=(W_1,\ldots, W_n)$ and $\Lambda=(\Lambda_1,\ldots, \Lambda_n)$ be the associated weighted multi-shifs. According to  Theorem \ref{evaluation}, we have a complete description of the joint invariant subspaces  of $W_1,\ldots, W_n$ of co-dimension  one. However, we don't have such a description for all the invariant subspaces.  
 We remark that  a Beurling type  description of all the invariant subspaces  under  $W_1,\ldots, W_n$  was obtained  for certain particular classes of weights
  (see \cite{Po-domains}, \cite{Po-invariant}, \cite{Po-Bergman},  \cite{Po-NC}). 
  
  \begin{proposition} There are no proper joint invariant subspaces $\cM$  and $\cN$ under
  $W_1,\ldots, W_n$ such that   $F^2(H_n)=\cM\oplus \cN$ (direct sum).
  \end{proposition}
  \begin{proof} Assume that the decomposition above holds. Then there a nontrivial idempotent  $P\in B(F^2(H_n))$, i.e. $P^2=P$, $P\neq I$, $P\neq 0$,   such that $PW_i=W_iP$ for any $i\in \{1,\ldots, n\}$.
  Consequently, $P=\psi(\Lambda_1,\ldots, \Lambda_n)\in R^\infty(\boldsymbol\mu)$.
  Since $\psi(\Lambda_1,\ldots, \Lambda_n)\neq 0$, one can easily show that it is injective.
  The proof is similar to that corresponding to the particular case when $\mu_\beta=1$ for any $\beta\in \FF_n^+$ (see \cite{DP1}).
  Consequently, the relation  $\psi(\Lambda_1,\ldots, \Lambda_n)(\psi(\Lambda_1,\ldots, \Lambda_n)-I)=0$ implies $\psi(\Lambda_1,\ldots, \Lambda_n)=I$, which is a contradiction. The proof is complete.
  \end{proof}

 \begin{proposition}  \label{joint} Let $W=(W_1,\ldots, W_n)\in B(F^2(H_n))$  be   the weighted left multi-shift  associated with  $\boldsymbol\mu=\{\mu_\beta\}_{\beta\in \FF_n^+, |\beta|\geq 1}$. Then the following statements hold.
 \begin{enumerate}
 \item[(i)]  The joint point spectrum 
$\sigma_p(W_1^*,\ldots, W_n^*)$ coincides with the set
$$
\{\lambda:=(\lambda_1,\ldots, \lambda_n)\in \CC^n:\  \sum_{\alpha\in \FF_n^+} \frac{|\lambda_\alpha|^2}{{\boldsymbol \mu(\alpha, g_0)^2}} <\infty\}.
$$
and the joint eigenvectors for the operators    $W_1^*,\ldots, W_n^*$ are precisely 
 the vectors 
 $$
 z_\lambda:=\sum_{\alpha\in \FF_n^+} \frac{\bar \lambda_\alpha}{{\boldsymbol \mu(\alpha, g_0)}} e_\alpha,
 $$
 where  $\lambda:=(\lambda_1,\ldots, \lambda_n)\in \sigma_p(W_1^*,\ldots, W_n^*)$.   In this case, we have  
 $$W_i^*z_\lambda=\bar\lambda_i z_\lambda,\qquad i\in \{1,\ldots, n\}.
 $$
 \item[(ii)] The joint point spectrum 
$\sigma_p(W_1^*,\ldots, W_n^*)$ is a complete Reinhardt set in  $\CC^n$ containing the origin and 
$$
 \sigma_p(W_1^*,\ldots, W_n^*) \subset \{(\lambda_1,\ldots, \lambda_n)\in \CC^n:\  
 (|\lambda_1|^2+\cdots + |\lambda_n|^2)^{1/2}\leq r(W)\},
 $$
 where $r(W)$ is the joint spectral radius of $W=(W_1,\ldots, W_n)$.
  \item[(iii)] If  $\lambda=(\lambda_1,\ldots, \lambda_n)\in \CC^n$ is such that $\|\lambda\|=\|W\|$, then $\lambda\notin \sigma_p(W_1^*,\ldots, W_n^*)$.
 \item[(iv)] The joint point spectrum   $\sigma_p(W_1^*,\ldots, W_n^*)$ contains a ball centered at zero if and only if   
  $$\limsup_{k\to \infty} \left(\sum_{\alpha\in \FF_n^+, |\alpha|=k} \frac{1}{{\boldsymbol \mu(\alpha, g_0)^2}}\right)^{1/k}<\infty.
  $$
     In particular, the later relation   holds if 
  the weights are bounded above from zero.
 
 \end{enumerate}
 \end{proposition} 
 \begin{proof}  As mentioned before,  the operator $U:F^2(H_n)\to F^2(\boldsymbol\mu)$,  defined by $U(e_\alpha):=\frac{1}{\boldsymbol \mu(\alpha, g_0)}Z_\alpha$, $\alpha\in \FF_n^+$, is unitary and
 $UW_i=L_{Z_i}U, \  i\in \{1,\ldots,n\},
 $
 where $W=(W_1,\ldots, W_n)$ is the weighted left multi-shift associated with the weights $\{\mu_\beta\}_{\beta\in \FF_n^+ }$. Now item (i) can be extracted from Theorem \ref{evaluation}. For the benefit of the reader, we present another proof.
 
 Let $z=\sum_{\beta\in \FF_n^+} c_\beta e_\beta\in F^2(H_n)$, $z\neq 0$,  and assume that
 $W_i^*z=\bar\lambda_i z$ for any  $i\in \{1,\ldots, n\}$ for some $(\lambda_1,\ldots, \lambda_n)\in \CC^n$.
 Using the definition of the weighted left multi-shift, we deduce that
 \begin{equation*}
 \begin{split}
 c_\alpha&=\left< z,e_\alpha\right>= \left< z,e_\alpha\right>
 =\left<z, \frac{1}{{\boldsymbol \mu(\alpha, g_0)}} W_\alpha 1\right>\\
 &=\frac{1}{{\boldsymbol \mu(\alpha, g_0)}} \left< W_\alpha^*z,1\right>=
 \frac{1}{{\boldsymbol \mu(\alpha, g_0)}} \bar\lambda_\alpha\left<z,1\right>
 = \frac{c_{g_0}\bar\lambda_\alpha}{{\boldsymbol \mu(\alpha, g_0)}}  
 \end{split}
 \end{equation*} 
 for any $\alpha\in \FF_n^+$.
Consequently, $z=c_{g_0} \sum_{\alpha\in \FF_n^+} \frac{\bar \lambda_\alpha}{{\boldsymbol \mu(\alpha, g_0)}} e_\alpha\in  F^2(H_n)$, which implies 
$$ \sum_{\alpha\in \FF_n^+} \frac{|\lambda_\alpha|^2}{{\boldsymbol \mu(\alpha, g_0)^2}} <\infty.$$
Conversely, assume that   $\lambda=(\lambda_1,\ldots, \lambda_n)\in \CC^n$ satisfies the later relation. Then $z_\lambda:=\sum_{\alpha\in \FF_n^+} \frac{\bar \lambda_\alpha}{{\boldsymbol \mu(\alpha, g_0)}} e_\alpha$ is in the full Fock space $F^2(H_n)$.
 Using the fact that 
 $\boldsymbol \mu(g_i\gamma, g_0)=\mu_{g_i\gamma} \boldsymbol \mu(\gamma, g_0)$, we obtain
 \begin{equation*}
 \begin{split}
 W_i^*z_\lambda&= \sum_{\alpha\in \FF_n^+} \frac{\bar \lambda_\alpha}{{\boldsymbol \mu(\alpha, g_0)}} W_i^*e_\alpha
 =\sum_{\gamma\in \FF_n^+} \frac{\bar \lambda_{g_i\gamma}}{{\boldsymbol \mu(g_i\gamma, g_0)}} \mu_{g_i \gamma}e_\gamma\\
 &=\bar \lambda_i \sum_{\alpha\in \FF_n^+} \frac{\bar \lambda_\alpha}{{\boldsymbol \mu(\alpha, g_0)}} e_\alpha=\bar\lambda_iz_\lambda
 \end{split}
 \end{equation*}
 for any $i\in \{1,\ldots, n\}$. This completes the proof of item (i).

 To prove part (ii),
 fix $\lambda:=(\lambda_1,\ldots, \lambda_n)\in\sigma_p(W_1^*,\ldots, W_n^*)$ and
 note that 
 \begin{equation*}
 \begin{split}
 \left<\sum_{|\alpha|=k} W_\alpha W_\alpha^* z_\lambda, z_\lambda\right>
 &= \sum_{|\alpha|=k}\| W_\alpha^* z_\lambda\|^2\\
 &=\|z_\lambda\|^2 \sum_{|\alpha|=k} |\lambda_\alpha|^2=\|z_\lambda\|^2 (|\lambda_1|^2+\cdots + |\lambda_n|^2)^k
 \end{split}
 \end{equation*}
 
 and, consequently,
 $$
 (|\lambda_1|^2+\cdots + |\lambda_n|^2)^{1/2}\leq \left\|\sum_{|\alpha|=k} W_\alpha W_\alpha^*\right\|^{1/2k}
 $$
 for any $k\in \NN$. Hence, we deduce that  $(|\lambda_1|^2+\cdots + |\lambda_n|^2)^{1/2}\leq r(W)$.  The fact that  $\sigma_p(W_1^*,\ldots, W_n^*)$ is a complete Reinhardt set in  $\CC^n$ containing the origin is obvious due to part (i).

   To prove item (iii), let $\lambda:=(\lambda_1,\ldots, \lambda_n)\in \CC^n$. Since $$
\left\|\sum_{\beta\in \FF_n^+, |\beta|=k} W_\beta W_\beta^*\right\|=\sup_{\alpha, \beta\in \FF_n^+, |\beta|=k} \boldsymbol\mu (\beta,\alpha)^2\geq \boldsymbol\mu (\beta,g_0)^2
$$  
  we deduce that
\begin{equation*}
\begin{split}
\sum_{\alpha\in \FF_n^+} \frac{|\lambda_\alpha|^2}{{\boldsymbol \mu(\alpha, g_0)^2}} &\geq 
\sum_{\alpha\in \FF_n^+} \frac{|\lambda_\alpha|^2}{\left\|\sum_{\alpha\in \FF_n^+, |\alpha|=k} W_\alpha W_\alpha^*\right\|}\\
&\geq \sum_{k=0}^\infty \frac{ (|\lambda_1|^2+\cdots + |\lambda_n|^2)^k}{\left\|\sum_{i=1}^n W_iW_i^*\right\|^k}
\end{split}
\end{equation*}  
If   $\|\lambda\|=\|W\|$, the later series is divergent, which proves that $\lambda\notin \sigma_p(W_1^*,\ldots, W_n^*)$.

 If the joint point spectrum   $\sigma_p(W_1^*,\ldots, W_n^*)$ contains a ball centered at zero, then there is $r_0\in (0,1)$ such that 
 $$
 \sum_{k=0}^\infty\left(\sum_{|\alpha|=k} \frac{1}{{\boldsymbol \mu(\alpha, g_0)^2}}\right)z^k
 \leq  \sum_{k=0}^\infty\sum_{|\alpha|=k}\frac{r_0^{|\alpha|}}{{\boldsymbol \mu(\alpha, g_0)^2}}<\infty
 $$
 for any complex number $z$  with $|z|\leq r_0$.  Standard arguments  concerning the radius of convergence of a power series show that 
  $$\limsup_{k\to \infty} \left(\sum_{\alpha\in \FF_n^+, |\alpha|=k} \frac{1}{{\boldsymbol \mu(\alpha, g_0)^2}}\right)^{1/k}\leq \frac{1}{r_0}.
  $$
 Conversely, assume that the later relation  holds for some $r_0>0$. Then  there exists $t_0>0$ such that 
 $ \sum_{k=0}^\infty\left(\sum_{|\alpha|=k} \frac{1}{{\boldsymbol \mu(\alpha, g_0)^2}}\right)t_0^k<\infty$. Now, it is clear that, for any $\lambda=(\lambda_1,\ldots, \lambda_n)\in \CC^n$ with $|\lambda_i|\leq t_0$, we have  $ \sum_{\alpha\in \FF_n^+} \frac{|\lambda_\alpha|^2}{{\boldsymbol \mu(\alpha, g_0)^2}}<\infty$, which completes the proof of item (iv).

 Now,  assume that  there exists $M>0$ such that  $\mu_\alpha\geq M$ for any $\alpha\in \FF_n^+$.
  Then   and  $\boldsymbol \mu(\alpha, g_0)\geq M^k$ if $|\alpha|=k$ and 
  $$ \sum_{\alpha\in \FF_n^+} \frac{|\lambda_\alpha|^2}{{\boldsymbol \mu(\alpha, g_0)^2}}\leq \sum_{k=0}^\infty \frac{1}{M^{2k}}(|\lambda_1|^2+\cdots + |\lambda_n|^2)^k<\infty
  $$
   for any $(\lambda_1,\ldots, \lambda_k)$ with $|\lambda_1|^2+\cdots + |\lambda_n|^2<M^2$. Using item (i), we complete the proof.
 \end{proof}
 We remark that, in the particular case when $\mu_\beta=1$,  part (i) of Proposition \ref{joint} was proved in  \cite{ArPo} and \cite{DP1}.

 \begin{example}  Let $\mu_\beta=1$ for any $\beta\in \FF_n^+$.  Then 
 $\boldsymbol\mu(\alpha,g_0)=1$ and condition $\sum_{\alpha\in \FF_n^+} \frac{|\lambda_\alpha|^2}{{\boldsymbol \mu(\alpha, g_0)^2}} <\infty$ is equivalent to $\sum_{k=0}^\infty  (|\lambda_1|^2+\cdots + |\lambda_n|^2)^k<\infty$, which is equivalent to $ |\lambda_1|^2+\cdots + |\lambda_n|^2<1$. In this case we have 
 $$
 \sigma_p(W_1^*,\ldots, W_n^*)=\BB_n=\{\lambda\in \CC^n:\  \|\lambda\|_2<1\}.
 $$
 \end{example}
 
  \begin{example}  Let $\mu_\beta:=\left(\frac{|\beta|+1}{|\beta|}\right)^2$ for any $\beta\in \FF_n^+$ with $|\beta|\geq 1$, and $\mu_{g_0}:=1$. In this case, we have
  $\boldsymbol\mu(\alpha,g_0)=(|\alpha|+1)^2$ for any $\alpha\in \FF_n^+$, which implies 
  $$
  \sum_{\alpha\in \FF_n^+} \frac{|\lambda_\alpha|^2}{{\boldsymbol \mu(\alpha, g_0)^2}} = \sum_{k=0}^\infty \frac{(|\lambda_1|^2+\cdots + |\lambda_n|^2)^k}{(k+1)^4}.
  $$
  It is easy to see that
  $$
 \sigma_p(W_1^*,\ldots, W_n^*)=\overline{\BB}_n.
 $$
  \end{example}

 \begin{example}  Let $\mu_\beta:=\left(\frac{ 1}{|\beta|+1}\right)^2$ for any $\beta\in \FF_n^+$.  Then $\boldsymbol\mu(\alpha,g_0)=\frac{1}{(|\alpha|+1) !}$ and 
 $$
  \sum_{\alpha\in \FF_n^+} \frac{|\lambda_\alpha|^2}{{\boldsymbol \mu(\alpha, g_0)^2}} = \sum_{k=0}^\infty ((k+1) ! )^2 (|\lambda_1|^2+\cdots + |\lambda_n|^2)^k<\infty
  $$
 if and only if $\lambda_1=\cdots \lambda_n=0$. Therefore,
  $$
 \sigma_p(W_1^*,\ldots, W_n^*)=\{0\}.
 $$
\end{example} 
  
  We recall   that the joint  right Harte spectrum
  $\sigma_r(T_1,\ldots, T_n)$ of an   $n$-tuple
 $(T_1,\ldots, T_n)\in B(\cH)^n$ is the set of all $n$-tuples
    $(\lambda_1,\ldots, \lambda_n)$  of complex numbers such that the
     right ideal of $B(\cH)$  generated by the operators
     $\lambda_1I-T_1,\ldots, \lambda_nI-T_n$ does
      not contain the identity operator. The joint left Harte  spectrum $ \sigma_l(T^*_1,\ldots, T^*_n)$ is the complex conjugate of  $\sigma_r(T_1,\ldots, T_n)$.
We recall \cite{Po-disc} that
  $(\lambda_1,\ldots, \lambda_n)\notin \sigma_r(T_1,\ldots, T_n)$
  if and only if  there exists $\delta>0$ such that
  $\sum\limits_{i=1}^n (\lambda_iI-T_i)
  (\overline{\lambda}_iI-T_i^*)\geq \delta I$.

  The following result was proved in \cite{Ha} in a more general case. We include a proof in our particular setting.
 \begin{corollary}  $ \sigma_p(W_1^*,\ldots, W_n^*)\subset  \sigma_l(W^*_1,\ldots, W^*_n)$.
 \end{corollary}
 \begin{proof}
 Let  $\mu=(\mu_1,\ldots, \mu_n)\in \cD(\boldsymbol\mu)= \sigma_p(W_1^*,\ldots, W_n^*)$  and assume
that
 there is $\delta>0$ such that
 $$\sum_{i=1}^n \|(W_i-\mu_iI)^*h\|^2\geq \delta \|h\|^2 \quad
 \text{ for all } \ h\in F^2(H_n).
 $$
 Taking $h=z_\lambda:=\sum_{\alpha\in \FF_n^+} \frac{\bar \lambda_\alpha}{{\boldsymbol \mu(\alpha, g_0)}} e_\alpha$,
 $\lambda\in \cD(\boldsymbol\mu)$, in the
 inequality above, and using  Theorem \ref{joint}, we obtain
 $$
 \sum_{i=1}^n |\lambda_i-\mu_i|^2\geq \delta\quad \text{ for all }\
 \lambda\in \cD(\boldsymbol\mu),
 $$
 which is a contradiction. Due to the remarks preceding the corollary,
 we deduce that $\mu=(\mu_1,\ldots, \mu_n)\in \sigma_l(W^*_1,\ldots,
 W^*_n)$, which completes the proof.
 \end{proof}

 \begin{theorem}\label{w*-funct} A map $\Phi: F^\infty(\boldsymbol\mu)\to
\CC$ is a $w^*$-continuous multiplicative linear functional  if and
only if there exists $\lambda\in \cD(\boldsymbol\mu)$  such that
$$
\Phi(\psi(W))= \psi(\lambda), \qquad \psi(W)\in F^\infty(\boldsymbol\mu).
$$
In this case, we have
$$
\psi(\lambda)=\left<\psi(W)1, z_\lambda\right>=\left<\psi(W)u_\lambda, u_\lambda\right> \ \text{ and } \  |\psi(\lambda)|\leq \|\psi(W)\|,
$$
where $u_\lambda:=\frac{z_\lambda}{\|z_\lambda\|}$ and 
$$z_\lambda:=\sum_{\alpha\in \FF_n^+} \frac{\bar \lambda_\alpha}{{\boldsymbol \mu(\alpha, g_0)}} e_\alpha.$$
\end{theorem}

\begin{proof} 
First, we prove  that  if $\lambda=(\lambda_1,\ldots, \lambda_n)\in \cD(\boldsymbol\mu)$, then the map  $\Phi_\lambda:
F^\infty(\boldsymbol\mu)\to \CC$  defined by
$$\Phi_\lambda(\psi(W)):=\psi(\lambda), \qquad  \psi(W)\in
F^\infty(\boldsymbol\mu),
$$ 
is a $w^*$-continuous multiplicative linear
functional, and 
$$\Phi_\lambda(\psi(W)) =\left<\psi(W)1, z_\lambda\right>, \qquad  \psi(W_1,\ldots, W_n)\in
F^\infty(\boldsymbol\mu).
$$ 
 Indeed, if $\psi(W):=\sum_{\beta\in \FF_n^+} c_\beta W_\beta$
is in the Hardy algebra $F^\infty(\boldsymbol\mu)$, then, due to Theorem \ref{evaluation} and the identification of multiplier algebra $\cM^l (F^2(\boldsymbol\mu))$  with $F^\infty(\boldsymbol\mu)$, we have   $\sum_{\beta\in
\FF_n^+} |c_\beta||\lambda_\beta|<\infty$ and 
\begin{equation*}
\begin{split}
\left<\psi(W)1, z_\lambda\right>&= \left<
\sum_{\beta\in \FF_n^+}c_\beta \boldsymbol\mu(\beta,g_0) e_\beta,
\sum_{\beta\in \FF_n^+} \frac{1}{\boldsymbol\mu(\beta,g_0)} \overline{\lambda}_\beta
e_\beta \right>\\
&=\sum_{\beta\in\FF_n^+} c_\beta \lambda_\beta
=\psi(\lambda).
\end{split}
\end{equation*}
Hence,  $\Phi_\lambda$ is a $w^*$-continuous multiplicative linear
functional.
Moreover, for each $\beta\in \FF_n^+$, we have
\begin{equation*}
\begin{split}
\left<\psi(W)^* z_\lambda,
\boldsymbol\mu(\beta,g_0)e_\beta\right>&= \left< z_\lambda,
\psi(W) W_\beta (1)\right>\\
&=\overline {\lambda_\beta \psi(\lambda)} =\left<\overline {
\psi(\lambda)} z_\lambda,\boldsymbol\mu(\beta,g_0)
e_\beta\right>.
\end{split}
\end{equation*}
Hence, we deduce that
$ \psi(W)^* z_\lambda=\overline {
\varphi(\lambda)} z_\lambda$ and, consequently,
\begin{equation*}
\begin{split}
\left< \psi(W)u_\lambda, u_\lambda\right>&=
\frac{1}{\|z_\lambda\|^2}\left< z_\lambda,\psi(W)^*
z_\lambda\right>\\
&=\frac{1}{\|z_\lambda\|^2}\left< z_\lambda,\overline {
\psi(\lambda)} z_\lambda\right>= \psi(\lambda).
\end{split}
\end{equation*}

Now, let $\Phi: F^\infty(\boldsymbol\mu)\to
\CC$ be  a $w^*$-continuous multiplicative linear functional.
Note that $J:=\ker \Phi$ is a $w^*$-closed two-sided ideal of
$F^\infty(\boldsymbol\mu)$  of  codimension one.  The subspace
$\cM_{J}:=\overline{J(1)}\subset F^2(H_n)$  has codimension one and $\cM_J+\CC
1=F^2(H_n)$. Indeed, assume that there exists a vector $y\in
(\cM_J+\CC 1)^\perp$, $\|y\|=1$. Then we can choose a polynomial
$p:=p(W_1,\ldots, W_n)\in F^\infty(\boldsymbol\mu)$ such that
$\|p(1)-y\|<1$. Since $p-\varphi(p)I\in \ker \varphi=J$, we have
$p(1)-\varphi(p)\in \cM_J$. Since $y\perp \cM_J+\CC 1$, we deduce
that
\begin{equation*}
\begin{split}
1&=\|y\|=\left<y-\varphi(p),y\right>\leq
|\left<h-p(1),y\right>|+|\left<p(1)-\varphi(p),y\right>|\\
&=|\left<h-p(1),y\right>|<1,
\end{split}
\end{equation*}
which is a contradiction. Therefore, $M_J$ is an invariant subspace
under $W_1,\ldots, W_n$  and has codimension one. Due to Theorem
\ref{joint}, there exists $\lambda\in \cD(\boldsymbol\mu)$ such
that $\cM_J=\{z_\lambda\}^\perp$.

Hence, $J:=\ker \Phi\subseteq \ker \Phi_\lambda\subset
F^\infty(\boldsymbol\mu)$.  Since both $\ker \Phi$ and $\ker
\Phi_\lambda$ are $w^*$-closed two-sided maximal ideals  of
$F^\infty(\boldsymbol\mu)$ of codimension one, we must have $\ker
\Phi=\ker \Phi_\lambda$. Therefore, $\Phi=\Phi_\lambda$
and the proof is complete.
\end{proof}

\bigskip

 \section{Functional calculus for  arbitrary  $n$-tuples of operators}

 Let  $\boldsymbol\mu=\{\mu_\beta\}_{\beta\in \FF_n^+}$ be a sequence of strictly positive real numbers  satisfying the boundedness conditions  \eqref{bound1} and \eqref{bound2} and let  
     $F^\infty(\boldsymbol\mu)$  be the  noncommutative  Hardy algebra  associated.
     Define the noncommutative set
     $$
     \cD_{\boldsymbol\mu}(\cH):=\left\{ (X_1,\ldots, X_n)\in B(\cH)^n:  \sum_{\alpha\in \FF_n^+} \frac{1}{\boldsymbol \mu(\alpha, g_0)^2} X_\alpha X_\alpha^* <\infty
     \right\}.
     $$
 We remark that if $$\limsup_{k\to \infty} \left(\sum_{\alpha\in \FF_n^+, |\alpha|=k} \frac{1}{{\boldsymbol \mu(\alpha, g_0)^4}}\right)^{1/2k}<\infty,
  $$
  then the noncommutative Reinhardt set $\cD_{\boldsymbol\mu}(\cH)$ contains a ball centered at the origin  in $B(\cH)^n$.
  Indeed, according to \cite{Po-holomorphic},   the condition above implies that the  formal series  $\sum_{\alpha\in \FF_n^+} \frac{1}{{\boldsymbol \mu(\alpha, g_0)^2}} Z_\alpha$ is a free holomorphic function on a neighborhood of the origin in $B(\cH)^n$. Consequently,   
if  $\rho$ denotes the reciprocal of the $\limsup$  above  and $0<t_0<\frac{\rho}{\sqrt{n}}$,
then $(t_0I_\cH,\ldots t_0 I_\cH)\in [B(\cH)^n]_\rho$ and
$\sum_{k=0}^\infty \sum_{|\alpha|=k} a_\alpha t_0^k<\infty$.  Take $0<\epsilon<\sqrt{t_0}$ and assume that $X=(X_1,\ldots, X_n)\in  [B(\cH)^n]_\epsilon$. Then we have
$$
\sum_{k=0}^\infty \sum_{|\alpha|=k}   \frac{1}{{\boldsymbol \mu(\alpha, g_0)^2}} \|X_\alpha X_\alpha^*\|\leq  \sum_{k=0}^\infty \sum_{|\alpha|=k} \frac{1}{{\boldsymbol \mu(\alpha, g_0)^2}} \epsilon^2k\leq  \sum_{k=0}^\infty \sum_{|\alpha|=k} a_\alpha t_0^k<\infty,
$$
 which proves our assertion.

  In what follows, we present a $w^*$-continuous $F^\infty(\boldsymbol\mu)$-functional calculus  for the elements in the noncommutative Reinhardt set $\cD_{\boldsymbol\mu}(\cH)$, where $\cH$ is a separable Hilbert space.
 \begin{theorem}\label{calculus}  Let $T=(T_1,\ldots, T_n)\in  \cD_{\boldsymbol\mu}(\cH)$ and let  $\Psi_T:F^\infty(\boldsymbol\mu)\to B(\cH)$  be defined  by 
 $$
 \Psi_T(\varphi(W))=\varphi(T)=: \text{\rm SOT-}\lim_{N\to\infty}\sum_{ |\alpha|\leq N} \left(1-\frac{|\alpha|}{N+1}\right) c_\alpha T_\alpha,
 $$
 where $\varphi(W)\in F^\infty(\boldsymbol\mu)$ has the Fourier representation $\sum_{\alpha\in \FF_n}c_\alpha W_\alpha$. Then $\Psi_T$ has the following properties.

 \begin{enumerate}
 \item[(i)]  
 $\Psi_T\left(\sum_{|\alpha|\leq m} c_\alpha W_\alpha\right)=\sum_{|\alpha|\leq m} c_\alpha T_\alpha$,  $m\in \NN.$
 \item[(ii)]   $\Psi_T$ is sequentially WOT-(resp. SOT-) continuous.
 \item[(iii)]   $\Psi_T$ is  a completely bounded  algebra homomorphism and 
 $$
 \|\Psi_T\|_{cb}\leq \left\|\sum_{k=0}^\infty\sum_{\alpha\in \FF_n^+, |\alpha|=k} \frac{1}{\boldsymbol\mu (\alpha,g_0)^2} T_\alpha T_\alpha^*\right\|^{1/2}.
 $$
 \item[(iv)]  $\Psi_T$ is $w^*$-continuous.
 \item[(v)]  $r(\varphi(T))\leq r(\varphi(W))$ for any $\varphi(W)\in F^\infty(\boldsymbol\mu)$,  where  $r(X)$ denotes the spectral radius of   $X$.
  \end{enumerate}
 \end{theorem}
 \begin{proof} 
 According to to Corollary \ref{T1}, if $T=(T_1,\ldots, T_n)\in   \cD_{\boldsymbol\mu}(\cH)$, then there is a joint invariant subspace $\cM\subset F^2(H_n)\otimes \cH$ under the operators $W_1^*,\ldots, W_n^*$ such that
 \begin{equation}\label{sim}
 T_i^*=X^{-1}(W_i^*\otimes I)|_\cM X ,\qquad i\in \{1,\ldots, n\}.
 \end{equation}
 Define the map $\Phi_T:F^\infty(\boldsymbol\mu)\to B(\cH)$ by setting 
 \begin{equation}
 \label{PhiT}
 \Phi_T(\varphi(W)):=X^* P_\cM(\varphi(W)\otimes I_\cH)|_\cM (X^*)^{-1},\qquad \varphi(W)\in  F^\infty(\boldsymbol\mu).
 \end{equation}
 Note that relation \eqref{sim} implies 
 $T_\alpha=X^* P_\cM( W_\alpha\otimes I_\cH)|_\cM (X^*)^{-1}$ for any 
 $\alpha\in \FF_n^+$. This implies 
  $\Psi_T\left(\sum_{|\alpha|\leq m} c_\alpha W_\alpha\right)=\sum_{|\alpha|\leq m} c_\alpha T_\alpha$ for any $ m\in \NN$.

Now, we prove that $\Phi_T$ is sequentially WOT-(resp.~SOT-) continuous. Let $\{\zeta_k(W)\}_k\subset F^\infty(\boldsymbol\mu)$ be a sequence such that 
WOT-$\lim_{k\to\infty} \zeta_k(W)=\varphi(W)$.
Since the WOT and $w^*$ topologies concide on bounded set, we have 
$w^*$-$\lim_{k\to\infty} \zeta_k(W)=\varphi(W)$ which implies 
$w^*$-$\lim_{k\to\infty} (\zeta_k(W)\otimes I_\cH)=\varphi(W)\otimes I_\cH$.
Hence,  WOT-$\lim_{k\to\infty} (\zeta_k(W)\otimes I_\cH)=\varphi(W)\otimes I_\cH$ and 
WOT-$\lim_{k\to\infty}\Phi_T(\zeta_k(W))=\Phi_T(\varphi(W))$.
Similarly, if  SOT-$\lim_{k\to\infty} \zeta_k(W)=\varphi(W)$, then  $\| \zeta_k(W)\|\leq M$, $k\in \NN$, for some $M>0$. Consequently, we have 
SOT-$\lim_{k\to\infty} (\zeta_k(W)\otimes I_\cK)=\varphi(W)\otimes I_\cH$ which implies that
SOT-$\lim_{k\to\infty}\Phi_T(\zeta_k(W))=\Phi_T(\varphi(W))$. This proves our assertion.

 Let $\{\varphi_\iota(W)\}_\iota$ be a net in $F^\infty(\boldsymbol\mu)$ such that
 $w^*$-$\lim_\iota  \varphi_\iota(W)=\varphi(W)\in F^\infty(\boldsymbol\mu)$. Then we deduce that  
 $w^*$\text{-}$\lim_\iota  \varphi_\iota(W)\otimes I_\cH=\varphi(W)\otimes I_\cH$ and, consequently,
 $$w^*\text{-}\lim_\iota  X^* P_\cM(\varphi_\iota(W)\otimes I_\cH)|_\cM (X^*)^{-1}=X^* P_\cM(\varphi(W)\otimes I_\cH)|_\cM (X^*)^{-1}.
 $$
  This proves that  
 $w^*$-$\lim_\iota  \Phi_T(\varphi_\iota(W))=\Phi_T(\varphi(W))$. Therefore, $\Phi_T$ is $w^*$-continuous.
 
 Let $[\varphi_{i,j}(W)]_{s\times s}$ be an $s\times s$ matrix with entries in the Hardy algebra $F^\infty(\boldsymbol\mu)$. Note that relation \eqref{PhiT} implies
  \begin{equation*}
 [\Phi_T(\varphi_{ij}(W))]_{s\times s}:=(\oplus_1^s X^* P_\cM) [\Phi_T(\varphi_{ij}(W))]_{s\times s}\otimes I_\cH)|_\cM]_{s\times s} (\oplus_1^s(X^*)^{-1}). 
  \end{equation*}
 Hence, 
 $$
 \| [\Phi_T(\varphi_{ij}(W))]_{s\times s}\|
 \leq \|X\|\|X^{-1}\| [\Phi_T(\varphi_{ij}(W))]_{s\times s}\|.
 $$
 Therefore, $\Phi_T$ is a completely bounded linear map and $\|\Phi_T\|_{cb}\leq \|X\|\|X^{-1}\|$ . Using relations \eqref{sim} and \eqref{PhiT}, it is easy to see that $\Phi_T$ is a  homomorphism on the polynomials in $W_1,\ldots, W_n$. Since these polynomials are $w^*$-dense in $F^\infty(\boldsymbol\mu)$ and $\Phi_T$ is $w^*$-continuous, an approximation argument shows that $\Phi_T$ is a homomorphism on $F^\infty(\boldsymbol\mu)$.
 
 Now,   let $\varphi(W)\in  F^\infty(\boldsymbol\mu)$ have the Fourier representation $\sum_{\alpha\in \FF_n}c_\alpha W_\alpha$ and define 
 $$p_N(W):=\sum_{ |\alpha|\leq N} \left(1-\frac{|\alpha|}{N+1}\right) c_\alpha W_\alpha.
 $$
 According to Theorem \ref{closure}, $ \text{\rm SOT-}\lim_{N\to\infty} p_N(W)=\varphi(W)$
 and $\| p_N(W)\|\leq\|\varphi(W)\|$ for any $N\in \NN$.
Then $ \text{\rm SOT-}\lim_{N\to\infty} (p_N(W)\otimes I_\cH)=\varphi(W)\otimes I_\cH$. 
On the other hand, due to relation \eqref{PhiT}, we have 
$$p_N(T)=X^* P_\cM(p_N(W)\otimes I_\cH)|_\cM (X^*)^{-1}.$$    
Hence, we deduce that
$$
  \text{\rm SOT-}\lim_{N\to\infty} p_N(T)= X^* P_\cM(\varphi(W)\otimes I_\cH)|_\cM (X^*)^{-1}=\Phi_T(\varphi(W)).
  $$
  Therefore, $ \Psi_T(\varphi(W))= \Phi_T(\varphi(W))$ for any $\varphi(W)\in  F^\infty(\boldsymbol\mu)$.
  
  It remains to prove part(v). To this end, note that due to the properties of the map $\Psi_T$, we have
  $$
  \varphi(T)^k=X^* P_\cM(\varphi(W)^k\otimes I_\cH)|_\cM (X^*)^{-1},\qquad k\in \NN.
  $$
  Hence we deduce that $\|\varphi(T)^k\|^{1/k}\leq \|X\|^{1/k} \|X^{-1}\|^{1/k}\|\varphi(W)^k\|^{1/k}$. Passing to the limit as $k\to \infty$, we conclude that    $r(\varphi(T))\leq r(\varphi(W))$.
  The proof is complete.
 \end{proof}

 \begin{theorem}  \label{calc2} If $T=(T_1,\ldots, T_n)\in B(\cH)^n$, then there exist a  weighted multi-shift $W=(W_1,\ldots, W_n)$ associated with  some  weights
 $\boldsymbol\mu=\{\mu_\beta\}_{\beta\in \FF_n^+}$    satisfying the conditions  \eqref{bound1}, \eqref{bound2},  and  $r(W)= r(T)$, and such that
   the map $\Psi_T: F^\infty(\boldsymbol\mu)\to B(\cH)$   defined  by 
 $$
 \Psi_T(\varphi(W))=\varphi(T)=: \text{\rm SOT-}\lim_{N\to\infty}\sum_{ |\alpha|\leq N} \left(1-\frac{|\alpha|}{N+1}\right) c_\alpha T_\alpha,
 $$
 where $\varphi(W)\in F^\infty(\boldsymbol\mu)$ has the Fourier representation $\sum_{\alpha\in \FF_n}c_\alpha W_\alpha$,  has all the properties listed in Theorem \ref{calculus} and  $\|\Psi_T\|_{cb}\leq  \frac{\pi}{\sqrt{6}} $. When $T=(T_1,\ldots, T_n)$ is a nilpotent $n$-tuple  of order $m\geq 2$, 
 we have 
$$\|\Psi_T\|_{cb}\leq  \sqrt{\sum_{k=0}^{m-1}\frac{1}{(k+1)^2}}.
$$ 
 \end{theorem}
\begin{proof}
Assume that $T=(T_1,\ldots, T_n)$ is not a nilpotent $n$-tuple of operators.  As in the proof of Theorem \ref{main1}, we define
  \begin{equation*}
  \mu_\beta:=\frac{|\beta|+1}{|\beta|}\frac {\left\|\sum\limits_{\sigma\in \FF_n^+, |\sigma|=
  |\beta|}T_\sigma T_\sigma^*\right\|^{1/2}} {\left\|\sum\limits_{\sigma\in \FF_n^+, |\sigma|=|\beta|-1}
 T_\sigma T_\sigma^*\right\|^{1/2}}, \qquad     |\beta|\geq 1. 
  \end{equation*}
In this case, we have 
\begin{equation*}
\begin{split}
\boldsymbol\mu(\alpha, g_0)&=(|\alpha|+1)   \left\|\sum\limits_{\sigma\in \FF_n^+, |\sigma|
=|\alpha|}T_\sigma T_\sigma^*\right\|^{1/2}\\
\frac{\boldsymbol \mu(\alpha g_i, g_0)}{\boldsymbol \mu(\alpha, g_0)} 
&=\frac{|\alpha|+2}{|\alpha|+1} \cdot
\frac{ \left\|\sum\limits_{\sigma\in \FF_n^+, |\sigma|
=|\alpha|+1}T_\sigma T_\sigma^*\right\|^{1/2}}{ \left\|\sum\limits_{\sigma\in \FF_n^+, |\sigma|
=|\alpha|}T_\sigma T_\sigma^*\right\|^{1/2}}
\leq  \frac{|\alpha|+2}{|\alpha|+1}\left\|\sum_{i=1}^nT_iT_i^*\right\|^{1/2}.
\end{split}
\end{equation*} 
Consequently, the conditions  \eqref{bound1}, \eqref{bound2} are satisfied.   
 Let $W=(W_1,\ldots, W_n)$ be the  weighted multi-shift associated with the weights  
 $\{\mu_\beta\}_{\beta\in \FF_n^+, |\beta|\geq 1}.$
The fact that   $r(W)= r(T)$ is due to Theorem \ref{main1}. On the other hand, Theorem
 \ref{calculus} shows that the map $\Psi_T$ has all the  properties listed in the same theorem.  It remains to prove that  $\|\Psi_T\|_{cb}\leq  \frac{\pi}{\sqrt{6}} $.  Note that, according to  the proof of Theorem \ref{calculus}, $\|\Psi_T\|_{cb}\leq \|X\|\|X^{-1}\|$, where $X$ is the invertible operator implementing    the similarity  of $T$ with the compression of $W$ to the appropriate subspace. In our case, we have $X=K_{\boldsymbol\mu}$ and 
 $\|K_{\boldsymbol\mu}\|\|K_{\boldsymbol\mu}^{-1}\|\leq \frac{\pi}{\sqrt{6}}$ (see the proof of Theorem \ref{main1}).
  The case when $T=(T_1,\ldots, T_n)$ is a nilpotent $n$-tuple of operators can be treated in a similar manner.
The proof is complete
\end{proof}

 \begin{corollary} \label{inclus} In the setting of Theorem  \ref{calc2}, the joint point  spectrum $\sigma_p(W_1^*,\ldots, W_n^*)$ satisfies the following relations
 $$
 (\CC^n)_{r(T)}\subset \sigma_p(W_1^*,\ldots, W_n^*)\subset \overline{(\CC^n)}_{r(W)}.
 $$
 
 \end{corollary}
 \begin{proof}
 According to Corollary \ref{joint}, the joint point spectrum 
$\sigma_p(W_1^*,\ldots, W_n^*)$ coincides with the set
$$
\{\lambda=(\lambda_1,\ldots, \lambda_n)\in \CC^n:\  \sum_{\alpha\in \FF_n^+} \frac{|\lambda_\alpha|^2}{{\boldsymbol \mu(\alpha, g_0)^2}} <\infty\}
$$
and  $\sigma_p(W_1^*,\ldots, W_n^*)\subset \overline{(\CC^n)}_{r(W)}$.
Analysing the convergence of the series 
$$
\sum_{k=0}^\infty \frac{1}{(k+1)^2 \left\|\sum\limits_{\sigma\in \FF_n^+, |\sigma|
=|\alpha|}T_\sigma T_\sigma^*\right\|} (|\lambda_1|^2+\cdots +|\lambda_n|^2)^k
$$
and taking into account that 
$r(T)=\lim_{k\to \infty} \left\|\sum\limits_{\sigma\in \FF_n^+, |\sigma|
=k}T_\sigma T_\sigma^*\right\|^{1/2k}$, one can easily deduce  the inclusion  $ (\CC^n)_{r(T)}\subset \sigma_p(W_1^*,\ldots, W_n^*)$.
 \end{proof}

 In what follows we present a spectral  version of the Schwarz lemma for the noncommutative Hardy algebra $F^\infty(\boldsymbol\mu)$.

 \begin{theorem} \label{Schwarz} Let  $\boldsymbol\mu=\{\mu_\beta\}_{\beta\in \FF_n^+}$ be a sequence of strictly positive real numbers  with $\mu_\alpha\geq M>0 $ for any $\alpha\in \FF_n^+$ and  satisfying relations \eqref{bound1} and \eqref{bound2}.  
  If $\varphi(W)\in F^\infty(\boldsymbol\mu)$ has the properties that $\varphi(0)=0$ and $\|\varphi(W)\|\leq1$, then  
 $$
 r(\varphi(X))\leq r(X),\qquad X\in \cD_{\boldsymbol\mu}(\cH).
 $$
 \end{theorem}
 \begin{proof}
 Let $\varphi(W)\in  F^\infty(\boldsymbol\mu)$ have the Fourier representation 
 $\sum_{\beta\in \FF_n^+, |\alpha|\geq 1} c_\beta W_\beta$. Then for any polynomial  $p\in \Span \left\{ e_\alpha:\ \alpha\in \FF_n^+\right\}$, we have
 \begin{equation*}
 \varphi(W)p=\sum_{i=1}^n W_i\left(\sum_{\gamma\in \FF_n^+} c_{g_i\gamma} W_\gamma p\right)= \sum_{i=1}^n W_i\Phi_i(W)p
 \end{equation*}
 where $\Phi_i(W)p:=\sum_{\gamma\in \FF_n^+} c_{g_i\gamma} W_\gamma p$.
 Since the weights $\{\mu_\beta\}_{\beta\in \FF_n^+}$ are bounded away from zero, we have $W_i=S_iD_i$, where $D_i$ is an invertible positive diagonal operator. Note that
 $$
 \|W_i^*\varphi(W)\|\|p\|\geq \|W_i^*\varphi(W) p\|=\|D_i^2\Phi_i(W)p\|\geq \frac{1}{\|D_i^{-1}\|^2} \|\Phi_i(W)p\|
 $$
 for any  $p\in \Span \left\{ e_\alpha:\ \alpha\in \FF_n^+\right\}$. Hence, we deduce that $\Phi_i(W)\in F^\infty(\boldsymbol\mu)$, $i\in \{1,\ldots, n\}$, and we have the Gleason type decomposition 
 \begin{equation}
  \label{Gleason}
  \varphi(W)=\sum_{i=1}^n W_i\Phi_i(W).
 \end{equation}
Taking into account that the operators $W_1,\ldots, W_n$ have orthogonal ranges, we deduce that 
$$
\|\varphi(W)^*\varphi(W)\|=\|\sum_{i=1}^n \Phi_i(W)^*D_i^2 \Phi_i(W)\|\geq M \|\sum_{i=1}^n  \Phi_i(W)^* \Phi_i(W)\|
$$
 for some constant $M>0$.
 Due to Theorem \ref{calculus}, relation \eqref{Gleason} implies 
 $  \varphi(X)=\sum_{i=1}^n X_i\Phi_i(X)$ for any $X\in \cD_{\boldsymbol\mu}(\cH)$. The same theorem shows that the map $\Psi_T$ is completely  bounded, which  show that
 $$
 \left\| \left[ \begin{matrix} \Phi_1(X)\\ \vdots\\ \Phi_n(X)\end{matrix}\right]\right\|
 \leq K  \left\| \left[ \begin{matrix} \Phi_1(W)\\ \vdots\\ \Phi_n(W)\end{matrix}\right]\right\|
 $$
 for some constant $K>0$.  Using this inequality, we obtain
 \begin{equation*}
 \begin{split}
 \|\varphi(X)\|&\leq \|[X_1\cdots X_n]\| \left\| \left[ \begin{matrix} \Phi_1(X)\\ \vdots\\ \Phi_n(X)\end{matrix}\right]\right\|
 \leq K \|[X_1\cdots X_n]\| \left\| \left[ \begin{matrix} \Phi_1(W)\\ \vdots\\ \Phi_n(W)
  \end{matrix}\right]\right\|\\
  &=  K \|[X_1\cdots X_n]\|  \|\sum_{i=1}^n  \Phi_i(W)^* \Phi_i(W)\|^{1/2}\\
  &\leq \frac{K}{\sqrt{M}}  \|[X_1\cdots X_n]\| \|\varphi(W)^*\varphi(W)\|^{1/2}\\
  &\leq  \frac{K}{\sqrt{M}}  \|[X_1\cdots X_n]\| 
 \end{split}
 \end{equation*}
 In a similar manner, we can show that
 $$
 \|\varphi(X)^k\|\leq C\|[X_\alpha:\ |\alpha|=k]||
 $$
 for some constant $C>0$ which depends only on $\boldsymbol\mu$ and $X$. Hence, we deduce that
 $$
 \ \|\varphi(X)^k\|^{1/k}\leq C^{1/k} \left\|\sum_{|\alpha|=k}X_\alpha X_\alpha^*\right\|^{1/2k},\qquad k\in \NN.
 $$
 Taking $k\to \infty$, we obtain  
 $
 r(\varphi(X))\leq r(X) 
 $
 and complete the proof.
 \end{proof}

  \begin{theorem} \label{ext}  Let $W=(W_1,\ldots, W_n)$  be   the weighted left multi-shift  associated with  $\boldsymbol\mu=\{\mu_\beta\}_{ |\beta|\geq 1}$, where $\mu_\beta> 0$ and   conditions  \eqref{bound1}, \eqref{bound2} are satisfied.  If $W$  is   a row power bounded  with bound $M>0$, then the map $\Phi_W:\cA_n\to B(F^2(H_n))$ defined by $\Phi_W(p(S_1,\ldots, S_n)):=p(W_1,\ldots, W_n)$ can be extended to a $w^*$-continuous completely bounded map  $\Phi_W:F^\infty(H_n)\to F^\infty(\boldsymbol\mu)$  with  $\|\Phi_W\|_{cb}\leq M$ such that
  $$
  \Phi_W(\varphi(S_1,\ldots, S_n))= \varphi(W_1,\ldots, W_n)
  $$ for any 
   $\varphi(S_1,\ldots, S_n)\in F^\infty(H_n)$.
      \end{theorem}
 \begin{proof}
 Since $W$  is   a  row power bounded multi-shift with bound $M>0$, we can apply Theorem 
 \ref{power bounded} and deduce that $W$  is jointly similar to a pure row contraction and the opearator $A$  implementing  the similarity can be chosen such that $\|A\|\|A^{-1}\|\leq  M$.  Consequently, $p(W_1,\ldots, W_n)=A p(S_1,\ldots, S_n) A^{-1}$ for any noncommutative polynomial $p$.   Now, using the $w^*$-continuous $F^\infty(H_n)$-functional calculus for pure row contractions, one can easily complete the proof.
 \end{proof}

 \begin{corollary} \label{calculus2}   Let $\boldsymbol\mu=\{\mu_\beta\}_{ |\beta|\geq 1}$ be a sequence of strictly positive weights such that the associated  the weighted left multi-shift   $W=(W_1,\ldots, W_n)$ is    row power bounded  with bound $M>0$. 
 If $T=(T_1,\ldots, T_n)\in  \cD_{\boldsymbol\mu}(\cH)$, then the  map $\Gamma_T:F^\infty(H_n)\to B(\cH)$    defined  by 
 $$
 \Gamma_T(\varphi(S))=\varphi(T)=: \text{\rm SOT-}\lim_{N\to\infty}\sum_{ |\alpha|\leq N} \left(1-\frac{|\alpha|}{N+1}\right) c_\alpha T_\alpha,
 $$
 where $\varphi(S)$ has the Fourier representation $\sum_{\alpha\in \FF_n}c_\alpha S_\alpha$,  has the following properties.

 \begin{enumerate}
 \item[(i)]  
 $\Gamma_T\left(\sum_{|\alpha|\leq m} c_\alpha S_\alpha\right)=\sum_{|\alpha|\leq m} c_\alpha T_\alpha,\qquad m\in \NN.
 $
 \item[(ii)]   $\Gamma_T$ is sequentially WOT-(resp. SOT-)continuous.
 \item[(iii)]   $\Gamma_T$ is  a completely bounded  algebra homomorphism and 
 $$
 \|\Gamma_T\|_{cb}\leq M \left\|\sum_{k=0}^\infty\sum_{\alpha\in \FF_n^+, |\alpha|=k} \frac{1}{\boldsymbol\mu (\alpha,g_0)^2} T_\alpha T_\alpha^*\right\|^{1/2}.
 $$
 \item[(iv)]  $\Gamma_T$ is $w^*$-continuous.
 \item[(v)]  $r(\varphi(T))\leq r(\varphi(S))$ for any $\varphi\in F^\infty(H_n)$,  where  $r(X)$ denotes the spectral radius of   $X$.
  \end{enumerate}
 \end{corollary}

According to \cite{Po-holomorphic}, 
a formal power series $\sum_{\alpha\in \FF_n^+} c_\alpha Z_\alpha$  is   a free holomorphic function in a neighborhood of the origin if there is  $r>0$ such that the series $\sum_{k=0}^\infty \sum_{|\alpha|=k} c_\alpha X_\alpha$ is convergent in norm for any  Hilbert space $\cH$ and any $n$-tuple 
$(X_1,\ldots, X_n)\in B(\cH)^n$ with $\|X_1X_1^*+\cdots +X_nX_n^*\|<r^2$.  This  condition is equivalent to the relation
\begin{equation}\label{sup}
 \limsup_{k\to \infty}
\left(\sum_{|\alpha|=k} |c_\alpha|^2\right)^{1/2k}<\infty.
\end{equation}
We denote by ${\rm Hol}_0(Z)$ the algebra of all free holomorphic function $f$ satisfying the  condition above.
 
\begin{theorem} \label{funct-calc-nil}
Let $T=(T_1,\ldots, T_n)\in B(\cH)^n$ be a quasi-nilpotent $n$-tuple of operators. Then there exists     a   quasi-nilpotent weighted multi-shift $W=(W_1,\ldots, W_n)$ with the following properties.
\begin{enumerate}
\item[(i)] $\|f(T_1,\ldots, T_n)\|\leq \frac{\pi}{\sqrt{6}}\|f(W_1,\ldots, W_n)\|$ for any $f\in {\rm Hol}_0(Z)$.
\item[(ii)] $r(f(T_1,\ldots, T_n))\leq r(f(W_1,\ldots, W_n))$ for any $f\in {\rm Hol}_0(Z)$.
\item[(iii)] Let $\{f_k\},  f\in  {\rm Hol}_0(Z)$ be such that $\{\|f_k(W_1,\ldots, W_n)\|\}_k$ is a bounded sequence and 
$$f_k(W_1,\ldots, W_n)\to  f(W_1,\ldots, W_n),\qquad \text {as } \ k\to \infty,
$$  in 
the operator norm (resp. WOT-, $w^*$, SOT-) topology, then  $f_k(T_1,\ldots, T_n)\to  f(T_1,\ldots, T_n)$   in the corresponding topology, respectively.
\end{enumerate}
In particular, if   $f(S_1,\ldots, S_n)\in F^\infty(H_n)$ then
$$
\|f(T_1,\ldots, T_n)\|\leq \frac{\pi}{\sqrt{6}} M\|f(S_1,\ldots, S_n)\|,
$$
where $M$ is the power bound of $W=(W_1,\ldots, W_n)$.
\end{theorem}
\begin{proof}  According to Theorem 4.1 from \cite{Po-holomorphic}, if $f=\sum_{\alpha\in \FF_n^+} c_\alpha Z_\alpha$ is a free holomorphic function  on a neighborhood of zero, then, for any $n$-tuple $T=(T_1,\ldots, T_n)$ with $r(T)=0$, the series 
$$
f(T_1,\ldots, T_n):=\sum_{k=0}^\infty \sum_{|\alpha|=k} c_\alpha T_\alpha
$$
is convergent in norm. Applying Theorem \ref{main1} we find  a weighted multi-shift $W=(W_1,\ldots, W_n)$ with $r(W)=r(T)=0$ and such 
$$
T_i=X^*P_\cM(W_i\otimes I)|_\cM (X^*)^{-1},\qquad i\in \{1,\ldots, n\},
$$
where $X$ is an invertible operator with $\|X\|\leq \frac{\pi}{\sqrt{6}}$. Now, it is easy to see 
$$
f(T_1,\ldots, T_n)=X^*P_\cM(f(W_1,\ldots, W_n)\otimes I)|_\cM (X^*)^{-1} .
$$
As in the the proof of Theorem \ref{calculus} and Theorem \ref{calc2}, one can prove items (i), (ii), and (ii).
 To prove the last part of the theorem, note that $W$ is power bounded and, consequently, one can apply Theorem \ref{ext}   and the first part of this theorem to complete the proof.
\end{proof}

\bigskip

\section{Function theory on Reinhardt domains and multiplier algebras}
  
Let  $\boldsymbol\mu=\{\mu_\beta\}_{ |\beta|\geq 1}$ be a  weight sequence, where  
$\mu_\beta>0$ and    conditions  \eqref{bound1}, \eqref{bound2} are satisfied.  
 For each $\lambda=(\lambda_1,\ldots,
\lambda_n)$ and each $n$-tuple ${\bf k}:=(k_1,\ldots, k_n)\in
\NN_0^n$, where $\NN_0:=\{0,1,\ldots \}$, let $\lambda^{\bf
k}:=\lambda_1^{k_1}\cdots \lambda_n^{k_n}$. For each ${\bf k}\in
\NN_0$, we denote
$$
\Lambda_{\bf k}:=\{\alpha\in \FF_n^+: \ \lambda_\alpha =\lambda^{\bf
k} \text{ for all } \lambda\in \CC^n\}.
$$
If ${\bf k}\in \NN_0^n$,  we define the vector
$$
y^{(\bf k)}:=\frac{1}{\omega_{\bf k}} \sum_{\alpha \in \Lambda_{\bf
k}} \frac{1}{\boldsymbol\mu(\alpha, g_0)} e_\alpha\in F^2(H_n), \quad  \text{ where } \
\omega_{\bf k}:=\sum_{\alpha\in \Lambda_{\bf k}} \frac{1}{\boldsymbol\mu(\alpha, g_0)^2}
$$
and $\boldsymbol\mu(\alpha, g_0)$ is defined in Proposition \ref{radius}.
   Note that the set  $\{y^{(\bf k)}:\ {\bf
k}\in \NN_0^n\}$ consists  of orthogonal vectors in $F^2(H_n)$ and
$\|y^{(\bf k)}\|=\frac{1}{\sqrt{\omega_{\bf k}}}$. We denote by
$F_s^2(\boldsymbol\mu)$ the closed span of these vectors, and call it the
{\it symmetric weighted  Fock space} associated with the weight sequence 
$\boldsymbol\mu=\{\mu_\beta\}_{|\beta|\geq 1}$.

\begin{theorem} \label{symmetric}  Let  $\boldsymbol\mu=\{\mu_\beta\}_{|\beta|\geq 1}$ be a  weight  sequence 
 with the property that $\cD(\boldsymbol\mu)$ contains a neighborhood of the origin and let $W=(W_1.\ldots, W_m)$ be the associated weighted left multi-shift.
If
$J_c$ is the $w^*$-closed two-sided ideal of the Hardy algebra
$F^\infty(\boldsymbol\mu)$ generated by the commutators
$$W_iW_j-W_j W_i,\qquad i,j\in\{1,\ldots, n\},
$$
 then the following statements hold.
 \begin{enumerate}
 \item[(i)]
 $
 F_s^2(\boldsymbol\mu)=\overline{\text{\rm span}}\{z_\lambda: \ \lambda\in
\cD(\boldsymbol\mu)\}= F^2(H_n)\ominus \overline{J_c(1)}$, where $z_\lambda$ is defined in Proposition \ref{joint}.
\item[(ii)] The symmetric weighted Fock space $F_s^2(\boldsymbol\mu)$ can be
identified with the Hilbert space $\HH^2(\cD(\boldsymbol\mu))$ of all
functions $\varphi:\cD(\boldsymbol\mu)\to \CC$ which admit a power
series representation $\varphi(\lambda)=\sum_{{\bf k}\in \NN_0}
c_{\bf k} \lambda^{\bf k}$ with
$$
\|\varphi\|_2=\sum_{{\bf k}\in \NN_0}|c_{\bf
k}|^2\frac{1}{\omega_{\bf k}}<\infty.
$$
More precisely, every  element  $\varphi=\sum_{{\bf k}\in \NN_0}
c_{\bf k} y^{(\bf k)}$ in $F_s^2(\boldsymbol\mu)$  has a functional
representation on $\cD(\boldsymbol\mu)$ given by
$$
\varphi(\lambda):=\left<\varphi, z_\lambda\right>=\sum_{{\bf k}\in
\NN_0} c_{\bf k} \lambda^{\bf k}, \quad \lambda=(\lambda_1,\ldots,
\lambda_n)\in \cD(\boldsymbol\mu),
$$
and
$$
|\varphi(\lambda)|\leq
 \|\varphi\|_2\left(\sum_{\alpha\in \FF_n^+}  \frac{|\lambda_\alpha|^2}{\boldsymbol \mu(\alpha, g_0)^2}\right)^{1/2},\quad \lambda=(\lambda_1,\ldots, \lambda_n)\in
\cD(\boldsymbol\mu).
$$
\item[(iii)] $\HH^2(\cD(\boldsymbol\mu))$ is the reproducing kernel Hilbert space with kernel
  $\boldsymbol\kappa:\cD(\boldsymbol\mu)\times
\cD(\boldsymbol\mu)\to \CC$ defined by
$$
\boldsymbol\kappa(\zeta,\lambda):= \left<z_\lambda,
z_\zeta\right>=\sum_{\beta\in \FF_n^+} \frac{1}{\boldsymbol\mu(\beta,g_0)^2}   \zeta_\beta\overline{\lambda}_\beta
\quad \text{ for all }\ \lambda,\zeta\in
\cD(\boldsymbol\mu).
$$
\end{enumerate}
\end{theorem}
\begin{proof}   First, we note that  $\overline{\text{\rm span}}\{z_\lambda: \ \lambda\in
\cD(\boldsymbol\mu)\}\subset F^2_s(\boldsymbol\mu)$ due to the fact that
$$z_\lambda:=\sum_{\alpha\in \FF_n^+} \frac{\bar \lambda_\alpha}{{\boldsymbol \mu(\alpha, g_0)}} e_\alpha=\sum_{{\bf k}\in \NN_0} \bar\lambda^{\bf k} \omega_{\bf k} y^{({\bf k})}.
$$
  On the other hand, for any $\gamma,\beta\in \FF_n^+$, $i,j\in \{1,\ldots, n\}$, and ${\bf
k}\in \NN_0^n$, we have
   \begin{equation*}
   \begin{split}
   &\left<y^{({\bf k})}, W_\gamma(W_jW_i-W_iW_j)W_\beta 1\right>\\
& \qquad= 
   \frac{1}{\omega_{\bf k}}\left< \sum_{\alpha \in \Lambda_{\bf
k}} \frac{1}{\boldsymbol\mu(\alpha, g_0)} e_\alpha, 
\boldsymbol\mu(\gamma g_jg_i \beta, g_0) e_{\gamma g_jg_i \beta}-
\boldsymbol\mu(\gamma g_ig_j \beta, g_0) e_{\gamma g_ig_j \beta}\right>=0.
\end{split}
\end{equation*}
Hence, we deduce that      $F_s^2(\boldsymbol\mu)\subset  F^2(H_n)\ominus \overline{J_c(1)}$. Now, assume that  there is a vector $v=\sum_{\beta\in \FF_n^+} c_\beta e_\beta\in F^2(H_n)\ominus \overline{J_c(1)}$ such that $\left<v, z_\lambda\right>=0$ for any $\lambda\in \cD(\boldsymbol \mu)$. Then we have
$$
\sum_{{\bf k}\in \NN_0}\left(\sum_{\beta\in \Lambda_{\bf k}}c_\beta \frac{1}{\boldsymbol\mu(\beta,g_0)}\right) \lambda^{\bf k}=0,\qquad \lambda\in \cD(\boldsymbol \mu).
$$
 Since    $\cD(\boldsymbol \mu)$ contains a ball centered at the origin, the uniqueness of the power series representations of holomorphic functions  on domains in $\CC^n$  implies 
 $\sum_{\beta\in \Lambda_{\bf k}}c_\beta \frac{1}{\boldsymbol\mu(\beta,g_0)}=0$ for any ${\bf k}\in \NN_0^n$.
   
 Let $\beta_0=\gamma g_j g_i\omega \in \Lambda_{\bf k}$ and let $\beta=\gamma g_i g_j\omega\in \Lambda_{\bf k}$, where
  $\gamma,\omega\in \FF_n^+$
and $i\neq j$. Since $v\in
 F^2(H_n)\ominus\overline{J_c(1)}$, we   have
$$
\left<v,W_\gamma(W_jW_i-W_iW_j)W_\omega(1)\right>=0,
$$
which implies
$c_{\beta_0}\boldsymbol\mu(\beta_0,g_0)= c_\beta\boldsymbol\mu(\beta,g_0)$.
Since any element $\gamma\in \Lambda_{\bf k}$ can be obtained from
$\beta_0$  by successive  transpositions, repeating the above
argument, we deduce that
$
c_{\beta_0}\boldsymbol\mu(\beta_0,g_0)= c_\gamma\boldsymbol\mu(\gamma,g_0)$
 for all $\gamma\in \Lambda_{\bf k}.
$
Since $\sum_{\beta\in \Lambda_{\bf k}}c_\beta \frac{1}{\boldsymbol\mu(\beta,g_0)}=0$ for any ${\bf k}\in \NN_0^n$, we conclude that 
  $c_\beta=0$ for any $\beta\in
\Lambda_{\bf k}$ and ${\bf k}\in \NN_0^n$, so $v=0$. Consequently,
we have $\overline{\text{\rm span}}\{z_\lambda: \ \lambda\in
\cD(\boldsymbol\mu)\}=F^2(H_n)\ominus \overline{J_c(1)}$, which completes the proof of part (i).

  To prove part (ii), we remark that
  \begin{equation*}
  \begin{split}
  \left<y^{({\bf k})}, z_\lambda\right>
  &=  \frac{1}{\omega_{\bf k}}\left< \sum_{\alpha \in \Lambda_{\bf
k}} \frac{1}{\boldsymbol\mu(\alpha, g_0)} e_\alpha, z_\lambda\right>\\
&= \frac{1}{\omega_{\bf k}}  \sum_{\alpha \in \Lambda_{\bf
k}} \frac{1}{\boldsymbol\mu(\alpha, g_0)} \lambda_\alpha=\lambda^{\bf k}
  \end{split}
  \end{equation*}
for any ${\bf k}\in \NN_0^n$ and $\lambda\in \cD(\boldsymbol\mu)$.   
 This shows that every element $\psi=  \sum_{{\bf k}\in \NN_0}
c_{\bf k} y^{({\bf k})}\in F_s^2(\boldsymbol\mu)$ has a functional
representation on $\cD(\boldsymbol\mu)$ given by
$
\psi(\lambda):=\left<\psi, z_\lambda\right>=\sum_{{\bf k}\in
\NN_0} c_{\bf k} \lambda^{\bf k} 
$
and, due to Cauchy-Schwartz inequality,
$
|\psi(\lambda)|\leq
 \|\psi\|_2\left(\sum_{\alpha\in \FF_n^+}  \frac{|\lambda_\alpha|^2}{\boldsymbol \mu(\alpha, g_0)^2}\right)^{1/2}$.  We remark that any function  $\varphi \in \HH^2(\cD(\boldsymbol\mu))$ is uniquely determined by the coefficients of its power series representation  
 $\varphi(\lambda)=\sum_{{\bf k}\in \NN_0}
c_{\bf k} \lambda^{\bf k}$ due  to the uniqueness of the power series representations of holomorphic functions  on domains in $\CC^n$.
  Now, it is clear that the map $U:F_s^2(\boldsymbol\mu)\to \HH^2(\cD(\boldsymbol\mu))$ defined by 
 $$
 U\left(\sum_{{\bf k}\in \NN_0}
c_{\bf k} y^{({\bf k})} \right):=\sum_{{\bf k}\in
\NN_0} c_{\bf k} \lambda^{\bf k} 
 $$
 is a unitary operator.    This completes the proof of part (ii).  Note that part (iii) follows easily from (ii).
 \end{proof}

We define the operator  $B_i\in B(F^2_s(\boldsymbol\mu))$ by setting $B_i:=P_{F^2_s(\boldsymbol\mu)} W_i|_{F^2_s(\boldsymbol\mu)}$,  for each $i\in \{1,\ldots, n\}$.

 \begin{theorem} \label{commutant} Let  $\boldsymbol\mu=\{\mu_\beta\}_{|\beta|\geq 1}$ be a weight sequence   
 with the property that $\cD(\boldsymbol\mu)$ contains a neighborhood of the origin and let $B_1, \ldots, B_n$ be the  compressions of $W_1,\ldots, W_n$  to the symmetric weighted Fock space $F_s^2(\boldsymbol\mu)$, respectively.
 Then the following statements hold.
 \begin{enumerate}
 \item[(i)]    The $n$-tuple $(B_1,\ldots, B_n)$ is unitarily equivalent to $(M_{\lambda_1},\ldots, M_{\lambda_n})$, where   $M_{\lambda_i}$  is  the multiplication on $\HH^2(\cD(\boldsymbol\mu))$ by the coordinate   function $\lambda_i$.

 \item[(ii)] $\sigma_p(B_1^*,\ldots, B_n^*)=\sigma_p(W_1^*,\ldots, W_n^*)=\cD(\boldsymbol\mu)$
 \item [(iii)]  $A\in \{B_1,\ldots, B_n\}'$ if and only if  there is a unique  multiplier $\varphi$ of 
 $\HH^2(\cD(\boldsymbol\mu))$ 
 such that $$M_\varphi=UAU^*,$$
  where the unitary operator $U:F_s^2(\boldsymbol\mu)\to \HH^2(\cD(\boldsymbol\mu))$ is defined by $U(y^{({\bf k})}):=\lambda^{\bf k}$.
 \item[(iv)] Each  $A\in \{B_1,\ldots, B_n\}'$  is uniquely determined  by
 the vector
 $$A1=\sum_{{\bf k}\in \NN_0^n} c_{\bf k} y^{({\bf k})}\in F_s^2(\boldsymbol\mu)
 $$
 and 
 $$A\zeta=\sum_{{\bf p}\in \NN_0^n}\sum_{{\bf k}\in \NN_0^n} c_{\bf k} a_{\bf p} y^{({\bf k}+{\bf p})},\qquad \text{ if } \  \zeta=\sum_{{\bf p}\in \NN_0^n}a_{\bf p} y^{({\bf p})}\in F_s^2(\boldsymbol\mu).
$$ 
\item[(v)]  $\left<A1, z_\lambda\right>=\varphi(\lambda):=\sum_{{\bf k}\in \NN_0^n} c_{\bf k}\lambda^{\bf k}$ and 
$A^*z_\lambda=\overline{\varphi(\lambda)}z_\lambda$ for any $\lambda \in \cD(\boldsymbol\mu)$.
 \end{enumerate}
  \end{theorem}
  
  \begin{proof}  According to the proof of Theorem \ref{w*-funct} , for any $\varphi(W)\in F^\infty(\boldsymbol\mu)$ and $\lambda=(\lambda_1,\ldots, \lambda_n)\in \cD(\boldsymbol\mu)$, we have 
 $\varphi(W)^* z_\lambda=\overline {\varphi(\lambda)} z_\lambda$. Due to the fact that  $z_\lambda\in F^2 (\boldsymbol\mu)$, for any  $f\in F_s^2(\boldsymbol\mu)$, we have
 \begin{equation*}
 \begin{split}
 \left<P_{F^2_s(\boldsymbol\mu)} \varphi(W)|_{F^2_s(\boldsymbol\mu)} f,z_\lambda\right>
 &=\left< f, \varphi(W)^*z_\lambda\right>\\
 &=\varphi(\lambda)\left<f,z_\lambda\right>=\varphi(\lambda) f(\lambda).
 \end{split}
 \end{equation*}
 Consequently,  $P_{F^2_s(\boldsymbol\mu)} \varphi(W)|_{F^2_s(\boldsymbol\mu)}$ is unitarily equivalent to the multiplication operator $M_\varphi\in B(\HH^2(\cD(\boldsymbol\mu)))$. In particular,  $(B_1,\ldots, B_n)$ is unitarily equivalent to $(M_{\lambda_1},\ldots, M_{\lambda_n})$.

  Taking into account  that  $W_i^*z_\lambda=\bar \lambda_i z_\lambda$ for any $\lambda=(\lambda_1,\ldots, \lambda_n)\in \cD(\boldsymbol\mu)$ and, due to Theorem \ref{symmetric}, 
  $F_s^2(\boldsymbol\mu)=\overline{\text{\rm span}}\{z_\lambda: \ \lambda\in
\cD(\boldsymbol\mu)\}$, we deduce   that $W_i^* F_s^2(\boldsymbol\mu)\subset F_s^2(\boldsymbol\mu)$ and   $B_i^* z_\lambda=\bar \lambda_i z_\lambda$.
Conversely, assume that  $\lambda=(\lambda_1,\ldots, \lambda_n)\in \CC^n$ and there is $z\in F_s^2(\boldsymbol\mu)$, $z\neq 0$, such that  $B_i^* z =\bar \lambda_i z $ for any $i\in \{1,\ldots, n\}$. Since $B_i^*=W_i^*|_{F_s^2(\boldsymbol\mu)}$, we also have  $W_i^*z=\bar \lambda_i z$ which, due to Proposition \ref{joint}, implies $\lambda\in 
\sigma_p(W_1^*,\ldots, W_n^*) $. This proves item (ii).

 To prove part (iii), note that,  for any ${\bf k}=(k_1,\ldots, k_n)\in \NN_0^n$, we have
\begin{equation*}
\begin{split}
\left< B_i y^{({\bf k})}, z_\lambda\right>
&=\left<   y^{({\bf k})}, W_i^* z_\lambda\right>=\left<   y^{({\bf k})}, \bar\lambda_i z_\lambda\right>\\
&=\lambda_i \lambda^{({\bf k})}=\left<y^{( k_1,\ldots, k_j+1,\ldots k_n)}, z_\lambda\right>
\end{split}
\end{equation*}
for any $\lambda \in \cD(\boldsymbol\mu)$. Hence, $B_iy^{({\bf k})}=y^{( k_1,\ldots, k_j+1,\ldots k_n)}$ which implies
$A(y^{({\bf k})})=AB^{{\bf k}} 1= B^{{\bf k}} A 1$. This proves that any 
$A\in \{B_1,\ldots, B_n\}'$ is uniquely determined  by $A1\in F_s^2(\boldsymbol\mu)$.
Setting $\varphi:=UA1\in \HH^2(\cD(\boldsymbol\mu))$, we complete the proof of part (iii).

To prove item (iv), assume that  $A1=\sum_{{\bf k}\in \NN_0^n} c_{\bf k} y^{({\bf k})}\in F_s^2(\boldsymbol\mu)$.
 Note that  
$$A(y^{({\bf p})})=B^{{\bf p}} A 1=\sum_{{\bf k}\in \NN_0^n} c_{\bf k} B^{{\bf p}}y^{({\bf k})}
=\sum_{{\bf k}\in \NN_0^n} c_{\bf k}  y^{({\bf p}+{\bf k})}, \qquad {\bf p}\in \NN_0^n.
$$
Now, one can easily deduce part (iv). The first part of item (v) is clear. To prove the second part, note that
\begin{equation*}
\begin{split}
\left<A^* z_\lambda, y^{({\bf k})}\right> &=\left<  z_\lambda,  Ay^{({\bf k})}\right>
=\left<  z_\lambda,  B^{{\bf k}} A 1\right>\\
&=\left<(B^{{\bf k})^ *} z_\lambda, A 1\right>=\left<\bar \lambda^{\bf k}z_\lambda, \sum_{{\bf k}\in \NN_0^n} c_{\bf k} y^{({\bf k})}\right>\\
&=\left<\overline{\varphi(\lambda)} z_\lambda, y^{({\bf k})}\right>
\end{split}
\end{equation*}
for any ${\bf k}=(k_1,\ldots, k_n)\in \NN_0^n$. Consequently, 
$A^*z_\lambda=\overline{\varphi(\lambda)}z_\lambda$ for any $\lambda \in \cD(\boldsymbol\mu)$.
This completes the proof.
\end{proof}
 
 Under the notation of the theorem above, we  call the formal series  $\sum_{{\bf k}\in \NN_0^n} c_{\bf k} B^{\bf k}$ the Fourier representation of $A$ and note that
 $\sum_{{\bf k}\in \NN_0^n} c_{\bf k} B^{\bf k} y^{({\bf p})}=Ay^{({\bf p})}$  for any ${\bf p}\in \NN_0^n$. This shows that $A$ is uniquely determined by its Fourier representation.
 
 Similarly to the proof of Theorem 9.2 from \cite{Po-univ2}, one can show that if  $A\in \{B_1,\ldots, B_n\}'$ has the Fourier representation  $\sum_{{\bf k}\in \NN_0^n} c_{\bf k} B^{\bf k}$, then 
 $$
 A=\text{SOT-}\lim_{N\to \infty}\sum_{{\bf k}\in \NN_0^n, |{\bf k}|\leq N} \left(1-\frac{|{\bf k}|}{N+1}\right)c_{\bf k}B^{\bf k},
 $$
 $$
 \left\|\sum_{{\bf k}\in \NN_0^n, |{\bf k}|\leq N} \left(1-\frac{|{\bf k}|}{N+1}\right)c_{\bf k}B^{\bf k}\right\|\leq \|A\|,\qquad N\in \NN,
 $$
 and $$ \left\|\sum_{{\bf k}\in \NN_0^n, |{\bf k}|=m}  c_{\bf k}B^{\bf k}\right\|\leq \|A\|,\qquad m\in \NN.
 $$
 
 We introduce the Hardy algebra $F_s^\infty(\boldsymbol\mu)$  to be the $w^*$-closed non-self-adjoint algebra generated by  the operators $B_1,\ldots, B_n$ and the identity. Since $B_iB_j=B_jB_i$ for any $i,j\in \{1,\ldots, n\}$, one can use the results above to show that
 $$
 F_s^\infty(\boldsymbol\mu)=\{B_1,\ldots, B_n\}'=\overline{\cP(B)}^{SOT}=\overline{\cP(B)}^{WOT}=\overline{\cP(B)}^{w^*},
 $$
 where $\cP(B)$  stands for the algebra of all polynomials in $B_1,\ldots, B_n$ and the identity. Moreover,  $F_s^\infty(\boldsymbol\mu)$ is the sequential SOT-(resp. WOT-, w*) closure of $\cP(B)$.
 
 Using Theorem \ref{commutant} and the above-mentioned results, one can easily prove the following.
 \begin{theorem} \label{SOT-B}
 The Hardy algebra $F_s^\infty(\boldsymbol\mu)$ can be identified with the multiplier algebra $\cM(\HH^2(\boldsymbol\mu))$. Moreover, each $\varphi(B)\in F_s^\infty(\boldsymbol\mu)$ has a unique Fourier representation  $\sum_{{\bf k}\in \NN_0^n} c_{\bf k} B^{\bf k}$ and 
 $$
 \varphi(B)=\text{\rm SOT-}\lim_{N\to \infty}\sum_{{\bf k}\in \NN_0^n, |{\bf k}|\leq N} \left(1-\frac{|{\bf k}|}{N+1}\right)c_{\bf k}B^{\bf k}.
 $$
 \end{theorem}

 \bigskip

   If $A\in B(\cH)$ then the set of all
invariant subspaces of $A$ is denoted by $\text{\rm Lat}A$. For
any $\cU\subset B(\cH)$ we define
$
\text{\rm Lat}\,\cU=\bigcap_{A\in\cU}\text{\rm Lat}\,A.
$
 If $\cS$ is any collection of subspaces of $\cH$,
then we define $\text{\rm Alg}\,\cS$ by setting $\text{\rm
Alg}\,\cS:=\{A\in B(\cH):\ \cS\subset\text{\rm Lat}\,A\}.$ We recall
that the algebra $\cU\subset B(\cH)$ is {\it reflexive} if
$\cU=\text{\rm Alg Lat}\,\cU$.

\begin{theorem}\label{reflexivity}
The algebra $F_s^\infty(\boldsymbol\mu)$ is reflexive.
\end{theorem}
\begin{proof}  Let $A\in B(F_s^2(\boldsymbol\mu))$ be an operator   that leaves
invariant all the invariant subspaces under the operators
$B_1,\ldots, B_n$. According to Theorem \ref{commutant}, for any $\lambda\in \cD(\boldsymbol\mu)$, the subspace   $\CC z_\lambda$ is invariant under $B_1^*,\ldots, B_n^*$. Consequently, $A^*z_\lambda=\overline{\varphi(\lambda)} z_\lambda$ for some function 
$\varphi:\cD(\boldsymbol\mu)\to \CC$.   Note that, if $f\in F_s^2(\boldsymbol\mu)$,
then, due to Theorem \ref{symmetric}, $Af$ has the functional
representation
$$
 \left<Af,z_\lambda\right>=\left<\psi,A^*z_\lambda\right>=
\varphi(\lambda)\psi(\lambda)\quad \text{ for all }\ \lambda\in
\cD(\boldsymbol\mu).
$$
In particular,  
$\left<A(1),z_\lambda\right>=\varphi(\lambda)$.
  Therefore $\varphi$   can be
identified with $A(1)\in F_s^2(\boldsymbol\mu)$ and the relation above  shows that  
$\varphi\psi\in \HH^2(\cD(\boldsymbol \mu))$ for any  
$\psi\in F_s^2(\boldsymbol\mu)$. Therefore, $A\in \cM(F_s^2(\boldsymbol\mu))$ and $\varphi\in \cM(\HH^2(\cD(\boldsymbol \mu)))$. Since the Hardy algebra $F_s^\infty(\boldsymbol\mu)$ is identified with  the multiplier algebra  $\cM(\HH^2(\cD(\boldsymbol \mu)))$, the proof is complete.
\end{proof}

\bigskip

 \section{Functional calculus for   commuting   $n$-tuples of operators}

Throughout this section, we assume that 
   $\boldsymbol\mu=\{\mu_\beta\}_{|\beta|\geq 1}$ is a  weight  sequence  of strictly positive numbers satisfying the boundedness conditions  \eqref{bound1} and \eqref{bound2}  and having 
   the property that $\cD(\boldsymbol\mu)$  is a subset of $\CC^n$ containing a neighborhood of the origin. 
 Let $B_1, \ldots, B_n$ be the  compressions of $W_1,\ldots, W_n$  to the symmetric weighted Fock space $F_s^2(\boldsymbol\mu)$, respectively.  We recall that $F_s^\infty(\boldsymbol\mu)$ is the $w^*$-closed non-self-adjoint algebra generated by  the operators $B_1,\ldots, B_n$ and the identity, which   is identified with the multiplier algebra $\cM(\HH^2(\cD(\boldsymbol\mu)))$.  Under this identification,  the $n$-tuple $(B_1,\ldots, B_n)$ is unitarily equivalent to $(M_{\lambda_1},\ldots, M_{\lambda_n})$, where   $M_{\lambda_i}$  is  the multiplication on $\HH^2(\cD(\boldsymbol\mu))$ by the coordinate   function $\lambda_i$.

We recall from Section 5 that  $$
\Lambda_{\bf k}:=\{\alpha\in \FF_n^+: \ \lambda_\alpha =\lambda^{\bf
k} \text{ for all } \lambda\in \CC^n\},\qquad  {\bf k}\in \NN_0^n,
$$
and 
$$
y^{(\bf k)}:=\frac{1}{\omega_{\bf k}} \sum_{\alpha \in \Lambda_{\bf
k}} \frac{1}{\boldsymbol\mu(\alpha, g_0)} e_\alpha\in F^2(H_n), \quad  \text{ where } \
\omega_{\bf k}:=\sum_{\alpha\in \Lambda_{\bf k}} \frac{1}{\boldsymbol\mu(\alpha, g_0)^2}.
$$
The set  $\{y^{(\bf k)}:\ {\bf
k}\in \NN_0^n\}$ consists  of orthogonal vectors in $F^2(H_n)$ and
$\|y^{(\bf k)}\|=\frac{1}{\sqrt{\omega_{\bf k}}}$, and the symmetric weighted  Fock space $F_s^2(\boldsymbol\mu)$  is the closed span of these vectors.

 We denote by 
     $
     \cD^c_{\boldsymbol\mu}(\cH)$ the set all all $n$-tuples $X=(X_1,\ldots, X_n)\in  \cD_{\boldsymbol\mu}(\cH)$ such that $X_iX_j=X_jX_i$ for any $i,j\in \{1,\ldots, n\}$.  Note that  $X\in \cD_{\boldsymbol\mu}(\cH)$ if and only if  $\sum_{{\bf k}\in \NN_0^n} \omega_{\bf k} T^{\bf k} {T^{\bf k}}^*$ is WOT-convergent.
     Let $\cA_s(\boldsymbol\mu)$ be the norm closed non-self-adjoint algebra  generated by $B_1,\ldots, B_n$ and the identity.

 \begin{theorem} \label{Rota-com}    If $T=(T_1,\ldots, T_n) \in B(\cH)^n$ is an $n$-tuple of commuting operators  and there is a positive invertible operator $Q\in B(\cH)$  and positive constants $0<a\leq b$ such that 
 $$
 a I\leq  \sum_{{\bf k}\in \NN_0^n} \omega_{\bf k} T^{\bf k}Q {T^{\bf k}}^*\leq bI,
 $$
 then 
  $(T_1,\ldots, T_n)$ is jointly similar to  $$(P_\cM(B_1\otimes I_\cH)|_\cM,\ldots, P_\cM(B_n\otimes I_\cH)|_\cM),$$
 where  $\cM\subset F_s^2(\boldsymbol\mu)\otimes \cH$  is a joint invariant subspace under the operators $B_i^*\otimes I_\cH$,  $i\in \{1,\ldots, n\}$.
 In this case, we have $r(T)\leq r(B)$ and   the map $\Phi_T:\cA_s(\boldsymbol\mu)\to B(F^2(H_n)$ defined by 
 $$\Phi_T(p(B_1,\ldots, B_n)):=p(T_1,\ldots, T_n)$$
  is completely bounded and $\|\Phi_T\|_{cb}\leq \sqrt{\frac{b}{a}} $.
 \end{theorem}

 \begin{proof}  A close look at the proof of Theorem \ref{Rota} reveals that, if 
  $T_1,\ldots, T_n$ are commuting operators,  then the  operator
 $K_{\boldsymbol\mu}:\cH\to F^2(H_n)\otimes \cD$   has the range in  $F_s^2(\boldsymbol\mu)\otimes \cD$. Indeed, we have  
$$ K_{\boldsymbol\mu} h=
  \sum_{{\bf k}\in \NN_0^n} \omega_{\bf k}y^{(\bf k)}\otimes Q^{1/2}{T^{\bf k}}^*h\in F_s^2(\boldsymbol\mu)\otimes \cD.
  $$
Moreover, since $K_{\boldsymbol\mu} T_i^*=(W_i^*\otimes I)K_{\boldsymbol\mu}$ and $W_i^*|_{F^2_s(\boldsymbol\mu)}=B_i^*$ for any  $ i\in \{1,\ldots, n\}$, one can see that the subspace $\cM:=
 K_{\boldsymbol\mu}\cH\subset F_s^2(\boldsymbol\mu)\otimes \cD$ is invariant under  all $B_i^*\otimes I$ and, consequently,
  $K_{\boldsymbol\mu} T_i^*=(B_i^*\otimes I)K_{\boldsymbol\mu}$.
Defining $X:\cH\to \cM$ by   $Xh:=K_{\boldsymbol\mu} h$, $h\in \cH$, it  is clear that $X$ is an invertible operator  and $T_i^*=X^{-1}(B_i^*\otimes I)|_\cM X$  for any $i\in \{1,\ldots, n\}$. 
The remaining of the proof is similar to that  of Theorem \ref{Rota}. This completes the proof.
\end{proof}

We remark  that if $Q=I$ in Theorem \ref{Rota-com}, we obtain a Rota type  similarity  result for the elements   $T\in \cD^c_{\boldsymbol\mu}(\cH)$.
In this case, we have   $\|\Phi_T\|_{cb}\leq \left\|\sum_{{\bf k}\in \NN_0^n} \omega_{\bf k} T^{\bf k} {T^{\bf k}}^*\right\|^{1/2}$.

 In what follows we present a $w^*$-continuous $F_s^\infty(\boldsymbol\mu)$-functional calculus  for the elements in the commutative Reinhardt set $\cD^c_{\boldsymbol\mu}(\cH)$, where $\cH$ is a separable Hilbert space.
 \begin{theorem}\label{calculus-com}  Let $T=(T_1,\ldots, T_n)\in  \cD^c_{\boldsymbol\mu}(\cH)$ and let  $\Psi_T:F_s^\infty(\boldsymbol\mu)\to B(\cH)$  be defined  by 
 $$
 \Psi_T(\varphi(B))=\varphi(T)=: \text{\rm SOT-}\lim_{N\to \infty}\sum_{{\bf k}\in \NN_0^n, |{\bf k}|\leq N} \left(1-\frac{|{\bf k}|}{N+1}\right)c_{\bf k}T^{\bf k}.
 $$
 where $\varphi(B)\in F_s^\infty(\boldsymbol\mu)$ has the Fourier representation  $\sum_{{\bf k}\in \NN_0^n} c_{\bf k} B^{\bf k}$. Then $\Psi_T$ has the following properties.

 \begin{enumerate}
 \item[(i)]  
 $\Psi_T\left(\sum_{|\alpha|\leq m} c_\alpha B_\alpha\right)=\sum_{|\alpha|\leq m} c_\alpha T_\alpha,\qquad m\in \NN.
 $
 \item[(ii)]   $\Psi_T$ is sequentially WOT-(resp. SOT-)continuous.
 \item[(iii)]   $\Psi_T$ is  a completely bounded  algebra homomorphism and 
 $$
 \|\Psi_T\|_{cb}\leq \left\|\sum_{{\bf k}\in \NN_0^n} \omega_{\bf k} T^{\bf k} {T^{\bf k}}^*\right\|^{1/2}.
 $$
 \item[(iv)]  $\Psi_T$ is $w^*$-continuous.
 \item[(v)]  $r(\varphi(T))\leq r(\varphi(B))$ for any $\varphi(B)\in  F_s^\infty(\boldsymbol\mu)$,  where  $r(X)$ denotes the spectral radius of   $X$.
  \end{enumerate}
 \end{theorem}
\begin{proof}  Applying  Theorem \ref{Rota-com}   in the particular case when  $Q=I$, we find a subspace 
 $\cM\subset F_s^2(\boldsymbol\mu)\otimes \cH$ which  is  jointly invariant subspace under all the operators $B_i^*\otimes I_\cH$ and an invertible operator $X:\cH\to \cM$ such that $T_i=X^* P_\cM(B_i\otimes I)|_\cM (X^*)^{-1}$,   $i\in \{1,\ldots, n\}$, and $\|X\|= \left\|\sum_{{\bf k}\in \NN_0^n} \omega_{\bf k} T^{\bf k} {T^{\bf k}}^*\right\|^{1/2}$.
 Define the map $\Phi_T:F_s^\infty(\boldsymbol\mu)\to B(\cH)$ by setting 
 \begin{equation*}
  \Phi_T(\varphi(B)):=X^* P_\cM(\varphi(B)\otimes I_\cH)|_\cM (X^*)^{-1},\qquad \varphi(B)\in  F_s^\infty(\boldsymbol\mu).
 \end{equation*}
 Following the lines of the proof of Theorem \ref{calculus}, one can similarly show that the map $\Phi_T$ has all the required properties and coincides with $\Psi_T$.  In addition, we need to use   Theorem \ref{SOT-B} and the remarks preceding it. 
 \end{proof}

 \begin{theorem}  \label{main1-com} If $T=(T_1,\ldots, T_n)\in B(\cH)^n$ is an $n$-tuple of commuting operators, then there exists    a weight sequence  $\boldsymbol\mu=\{\mu_\beta\}_{|\beta|\geq 1}$   with  the following properties:
 \begin{enumerate}
\item[(i)] $\|\sum_{|\alpha|=k} B_\alpha B_\alpha^*\|^{1/2}\leq (k+1) \|\sum_{|\alpha|=k} T_\alpha T_\alpha^*\|^{1/2}$ for  any $k\in \NN$;
\item[(ii)] $r(B)= r(T)$;
\item[(iii)]    $(T_1,\ldots, T_n)$ is jointly similar to  $$(P_\cM(B_1\otimes I_\cH)|_\cM,\ldots, P_\cM(B_n\otimes I_\cH)|_\cM),$$
 where  $\cM\subset F_s^2(\boldsymbol\mu )\otimes \cH$  is a joint invariant subspace under the operators $B_i^*\otimes I_\cH$,  $i\in \{1,\ldots, n\}$;
  \item[(iv)] $r(p(T_1,\ldots, T_n))\leq r(p(B_1,\ldots, B_n))$ and    the map $\Phi_T:\cA_s(\boldsymbol\mu)\to B(F^2(H_n)$ defined by 
 $$\Phi_T(p(B_1,\ldots, B_n)):=p(T_1,\ldots, T_n)$$
  is completely bounded and $\|\Phi_T\|_{cb}\leq \frac{\pi}{\sqrt{6}} $.
 \end{enumerate}
 If  $T=(T_1,\ldots, T_n)$ is a nilpotent $n$-tuple of  order $m\geq 2$, then there is a weighted commuting multi-shift $B=(B_1,\ldots, B_n)$ which is nilpotent of order $m$ such that all the properties above hold and 
 $$\|\Phi_T\|_{cb}\leq \sqrt{\sum_{k=0}^{m-1}\frac{1}{(k+1)^2}}.
 $$
 \end{theorem}
\begin{proof}  Assume that $T=(T_1,\ldots, T_n)$ is not nilpotent.  Consider the weight sequence $\boldsymbol\mu=\{\mu_\beta\}_{|\beta|\geq 1}$  defined in the proof of Theorem \ref{main1} and let $W=(W_1,\ldots, W_n)$ and $B=(B_1,\ldots, B_n)$ be the corresponding multi-shifts associated with  $\boldsymbol\mu$. According to the same proof, we have $r(W)=r(T)$ and 
$$
 \left\|\sum_{{\bf k}\in \NN_0^n} \omega_{\bf k} T^{\bf k} {T^{\bf k}}^*\right\|\leq \sum_{k=0}^\infty \frac{1}{(k+1)^2}=\frac{\pi^2}{6}.
 $$
Since $T_1,\ldots, T_n$ are commuting, we apply Theorem \ref{Rota-com}  when $Q=I$ and deduce (iii) and (iv). In particular, we have $r(T)\leq r(B)$. On the other hand, since $B_i^*=W_i^*|_{F_s^2(\boldsymbol\mu)}$, it is clear that
$\|\sum_{|\alpha|=k} B_\alpha B_\alpha^*\| \leq  \|\sum_{|\alpha|=k} W_\alpha W_\alpha^*\|$
for any $k\in \NN$ and, consequently, $r(B)\leq r(W)$. Using again Theorem \ref{main1}, one can complete the proof. The case when $T$ is  a commuting nilpotent  tuple  can be treated in a similar manner.
\end{proof}

 \begin{theorem}  \label{calc2-com} If $T=(T_1,\ldots, T_n)\in B(\cH)^n$ is an $n$-tuple of commuting operators, then there   exists    a weight sequence  $\boldsymbol\mu=\{\mu_\beta\}_{|\beta|\geq 1}$   such that 
 \begin{enumerate} 
 \item[(i)]  $r(B)= r(T)$;
 \item[(ii)]     the conditions  \eqref{bound1}, \eqref{bound2} are satisfied;
 \item[(iii)]   if $T$ is not nilpotent,   $\cD(\boldsymbol\mu)$    contains a neighborhood of the origin;   
 \item[(iv)]
   the map $\Psi_T:F_s^\infty(\boldsymbol\mu)\to B(\cH)$   defined  by 
 $$
 \Psi_T(\varphi(B))=\varphi(T)=:  \text{\rm SOT-}\lim_{N\to \infty}\sum_{{\bf k}\in \NN_0^n, |{\bf k}|\leq N} \left(1-\frac{|{\bf k}|}{N+1}\right)c_{\bf k}T^{\bf k},
 $$
 where $\varphi(B)\in F_s^\infty(\boldsymbol\mu)$ has the Fourier representation  $\sum_{{\bf k}\in \NN_0^n} c_{\bf k} B^{\bf k}$,  has all the properties listed in Theorem \ref{calculus-com} and  $\|\Psi_T\|_{cb}\leq  \frac{\pi}{\sqrt{6}} $. 
 \end{enumerate}
 
 When $T=(T_1,\ldots, T_n)$ is a nilpotent $n$-tuple  of order $m\geq 2$, 
 we have 
$$\|\Psi_T\|_{cb}\leq  \sqrt{\sum_{k=0}^{m-1}\frac{1}{(k+1)^2}}.
$$ 
 \end{theorem}
 
\begin{proof}  Assume that $r(T)>0$. According to Theorem \ref{main1-com}, we find a weight sequence  $\boldsymbol\mu=\{\mu_\beta\}_{|\beta|\geq 1}$ such that items (i) and (ii) hold.
 Due to Corollary \eqref{inclus}, we have 
 $
 (\CC^n)_{r(T)}\subset \sigma_p(W_1^*,\ldots, W_n^*)=\cD(\boldsymbol\mu)
 $
which proves item (iii). Now, we can apply Theorem \ref{calculus-com} and Theorem 
\ref{main1-com} to deduce item (iv). The case when $T$ is nilpotent can be treated in a similar manner.  
\end{proof}

Recall that ${\rm Hol}_0(Z)$ denotes  the algebra of all free holomorphic functions $f$ satisfying    condition \eqref{sup}.

\begin{theorem} \label{funct-calc-nil-com}
Let $T=(T_1,\ldots, T_n)\in B(\cH)^n$ be a quasi-nilpotent commuting  $n$-tuple of operators. Then there exists     a   quasi-nilpotent commuting weighted multi-shift $B=(B_1,\ldots, B_n)$ with the following properties.
\begin{enumerate}
\item[(i)] $\|f(T_1,\ldots, T_n)\|\leq \frac{\pi}{\sqrt{6}}\|f(B_1,\ldots, B_n)\|$ for any $f\in {\rm Hol}_0(Z)$.
\item[(ii)] $r(f(T_1,\ldots, T_n))\leq r(f(B_1,\ldots, B_n))$ for any $f\in {\rm Hol}_0(Z)$.
\item[(iii)] Let $\{f_k\},  f\in  {\rm Hol}_0(Z)$ be such that $\{\|f_k(B_1,\ldots, B_n)\|\}_k$ is a bounded sequence and 
$$f_k(B_1,\ldots, B_n)\to  f(B_1,\ldots, B_n),\qquad \text {as } \ k\to \infty,
$$  in 
the operator norm (resp. WOT-, $w^*$, SOT-) topology, then  $f_k(T_1,\ldots, T_n)\to  f(T_1,\ldots, T_n)$   in the corresponding topology, respectively.
\end{enumerate}
 \end{theorem}
\begin{proof} The proof is similar to the proof of Theorem \ref{funct-calc-nil}, but uses Theorem \ref{main1-com}.
\end{proof}

\bigskip

      \bigskip

       %

      \end{document}